\documentclass[11pt]{article}
\usepackage{amssymb,amsmath,amsfonts,amsthm,mathtools,enumerate,color}
\usepackage{mathrsfs}
\usepackage[title,titletoc,header]{appendix}
\usepackage{graphicx}

\usepackage[labelformat=simple]{subcaption}

\usepackage{paralist}

\usepackage{tikz}
\usetikzlibrary{arrows,positioning,shapes.geometric}

\usepackage{algorithm,algorithmic}

\graphicspath{ {./figures/} }

\usepackage{indentfirst}
\usepackage{multicol}
\usepackage{booktabs}
\usepackage{url}
\usepackage[outdir=./]{epstopdf} 

\usepackage[shortlabels]{enumitem}

\usepackage{hyperref}
\hypersetup{
    colorlinks=true, 
    linktoc=all,     
    linkcolor=blue,  
}
\numberwithin{equation}{section}

\newtheorem{Theorem}{Theorem}[section]
\newtheorem{Lemma}[Theorem]{Lemma}
\newtheorem{Proposition}[Theorem]{Proposition}
\newtheorem{Corollary}[Theorem]{Corollary}
\newtheorem{Assumption}{H.\!\!}

\theoremstyle{definition}
\newtheorem{Definition}{Definition}[section]

\theoremstyle{remark}
\newtheorem{Remark}{Remark}[section]

 \def\p{\partial} \def\nb{\nonumber}
\def\to{\rightarrow}
 \def\ol{\overline}    
\def\Om{\Omega}  \def\om{\omega} 
 
\newcommand{\q}{\quad}   \newcommand{\qq}{\qquad}

\def\l{\label}    \def\fa{\forall}
\def\b{\beta}  \def\a{\alpha} 
 
\def\eps{\varepsilon}

 \def\t{\times}  \def\lam{\lambda}
\def\ms{\medskip}

\def \la{\langle} \def\ra{\rangle}

\def\cC{\mathcal{C}}

\def\cE{\mathcal{E}}
\def\cF{\mathcal{F}}
\def\cG{\mathcal{G}}

\def\cI{\mathcal{I}}

\def\cM{\mathcal{M}}
\def\cN{\mathcal{N}}
\def\cO{\mathcal{O}}
\def\cP{\mathcal{P}}

\def\cR{\mathcal{R}}
\def\cS{\mathcal{S}}

\def\cU{\mathcal{U}}
\def\cV{\mathcal{V}}
\def\cW{\mathcal{W}}

\def\cY{\mathcal{Y}}
\def\cZ{\mathcal{Z}}

\def\sE{{\mathbb{E}}}
\def\sF{{\mathbb{F}}}
\def\sI{{\mathbb{I}}}
\def\sN{{\mathbb{N}}}
\def\sP{\mathbb{P}}

\def\sR{{\mathbb R}}
\def\sS{{\mathbb{S}}}

\newcommand{\fbsde}{FBS$\Delta$E }
\newcommand{\err}{\textnormal{ERR}}

\newcommand{\tr}{\textnormal{tr}}

\DeclareMathOperator*{\argmin}{arg\,min}

\newcommand{\lc}
{\mathrel{\raise2pt\hbox{${\mathop<\limits_{\raise1pt\hbox
{\mbox{$\sim$}}}}$}}}

\newcommand{\gc}
{\mathrel{\raise2pt\hbox{${\mathop>\limits_{\raise1pt\hbox{\mbox{$\sim$}}}}$}}}

\newcommand{\ec}
{\mathrel{\raise2pt\hbox{${\mathop=\limits_{\raise1pt\hbox{\mbox{$\sim$}}}}$}}}

\def\bb{\begin{equation}} \def\ee{\end{equation}}
\def\bbn{\begin{equation*}} \def\een{\end{equation*}}

\def\beqn{\begin{eqnarray}}  \def\eqn{\end{eqnarray}}

\def\beqnx{\begin{eqnarray*}} \def\eqnx{\end{eqnarray*}}

\def\bn{\begin{enumerate}} \def\en{\end{enumerate}}

\def\bd{\begin{description}} \def\ed{\end{description}}



\begin{document}

\title{
A posteriori error estimates
for fully coupled 
McKean--Vlasov
forward-backward SDEs
}
\author{
Christoph Reisinger\thanks{
Mathematical Institute, University of Oxford, Oxford OX2 6GG, UK
 ({\tt christoph.reisinger@maths.ox.ac.uk, 
wolfgang.stockinger@maths.ox.ac.uk,
yufei.zhang@maths.ox.ac.uk})}
\and
Wolfgang Stockinger\footnotemark[1]
\and
Yufei Zhang\footnotemark[1]
}
\date{}

\maketitle


\noindent\textbf{Abstract.} 
 Fully coupled
 McKean--Vlasov
  forward-backward 
 stochastic differential equations 
 (MV-FBSDEs)
  arise naturally from 
   large population optimization problems.
 Judging the quality of  given numerical solutions
 for MV-FBSDEs,
 which 
 usually require Picard iterations and approximations of  nested conditional expectations,
   is typically difficult.
This paper proposes an {a posteriori} error  estimator 
to quantify the $L^2$-approximation error of an arbitrarily generated 
approximation  on a time grid.
We establish that  
the error estimator   
is equivalent to 
the global approximation error between
 the given numerical solution 
 and the solution of
a forward Euler discretized MV-FBSDE.
A crucial and challenging step in the analysis is the proof of stability of this Euler approximation to the MV-FBSDE, which is of independent interest.
We further demonstrate that,
for  sufficiently fine time grids,
the  accuracy
of numerical solutions
 for solving  the continuous MV-FBSDE
 can also be measured by the  error estimator.
The  error  estimates  
 justify the use of residual-based algorithms
 for solving 
MV-FBSDEs.
Numerical experiments 
for   MV-FBSDEs arising from 
  mean field control and games   
confirm the effectiveness and practical applicability of the 
error estimator.


\ms

\medskip

\section{Introduction}\l{sec:intro}
In this article, we 
propose an {a posteriori} error estimator to quantify the approximation accuracy 
of  given numerical solutions to
the following 
  MV-FBSDEs:
for all $t\in [0,T]$,
\begin{subequations}\l{eq:fbsde_conts_intro}
\begin{align}
 X_t&=\xi_0+\int_0^t b(s,X_s,Y_s,Z_s,\sP_{(X_s,Y_s,Z_s)})\,ds  +
 \int_0^t \sigma (s,X_s,Y_s,Z_s,\sP_{(X_s,Y_s,Z_s)})\, d W_s, 
 \l{eq:fbsde_conts_fwd_intro}
\\
 Y_t&=g(X_T,\sP_{X_T})+\int_t^Tf(s,X_s,Y_s,Z_s,\sP_{(X_s,Y_s,Z_s)})\,ds-\int_t^TZ_s\,d W_s,
  \l{eq:fbsde_conts_bwd_intro}
\end{align}
\end{subequations}
where 
$X,Y,Z$ are unknown solution processes taking values in $\sR^n,\sR^m, \sR^{m\t d}$, respectively,
$T>0$ is an arbitrary given finite number, 
$\xi_0$ is a given $n$-dimensional  random variable, 
$W$ is a $d$-dimensional standard Brownian motion,
$\sP_{(X_t,Y_t,Z_t)}$ is the marginal law of the process $(X,Y,Z)$
at time $t\in [0,T)$, 
$\sP_{X_T}$  is the marginal law of the process $X$
at the terminal time $T$,
and
$b,\sigma, g,h$ are given functions with appropriate dimensions,
which will be called the generator of \eqref{eq:fbsde_conts_intro}
as in \cite{yong2010}.

Such equations extend the classical FBSDEs without McKean--Vlasov interaction,  
i.e., the generator $(b,\sigma, g,h)$ is independent of the distribution of the solution triple $(X,Y,Z)$,
and  play an important role in 
large population optimization problems
(see e.g.~\cite{peng1999,carmona2013,bensoussan2015,carmona2015} and the references therein).
In particular, by applying the stochastic maximum principle, one can  
construct 
both the  equilibria
 of  the mean field games
and the solution to optimal mean field control problems
based on the solution triple $(X,Y,Z)$
of the fully-coupled MV-FBSDE \eqref{eq:fbsde_conts_intro}.
Moreover, the  Feynman-Kac representation formula
 for  partial differential equations  (PDEs)
can be generalized 
to certain nonlinear nonlocal PDEs 
defined on 
 the Wasserstein space
 (also known as  ``master equations") 
by using MV-FBSDE \eqref{eq:fbsde_conts_intro},
where 
the processes $Y$ and $Z$
give a stochastic representation of
the solutions  to master equations
and the gradient of the solutions, respectively
(see e.g.~\cite{chassagneux2014,buckdahn2017,chassagneux2019}).

\paragraph{Numerical challenges in solving MV-FBSDEs.}
As the solution to \eqref{eq:fbsde_conts_intro} is in general not known analytically, 
many numerical schemes have been proposed to solve these  nonlinear equations
in various special cases,
which  typically involve two steps.
First, 
a  time-stepping scheme,
such as the Euler-type discretizations in \cite{bouchard2004,zhang2004, bender2008,lionnet2015},
is employed to discretized  
the continuous-time dynamics
\eqref{eq:fbsde_conts_intro}  
into a discrete-time MV-FBSDE,
whose solution can be expressed 
in terms of  nested conditional expectations defined on the time grid.
Second,
 a suitable numerical procedure  is introduced to solve the  discrete-time MV-FBSDE,
which usually consists of  projecting the nested conditional expectations 
onto some trial spaces by least-squares regression
(see e.g.~\cite{delarue2006,gobet2006, bender2008,chaudru2015,e2017,andrea2019,carmona2019,
chassagneux2019,fouque2019, germain2019, ji2020,robert2019,hure2020}).

However, in the absence of an analytic solution,
 it is typically difficult to judge the quality of   a  numerical approximation, 
 especially 
  in the practically relevant pre-limit situation
  (i.e., for a given choice of discretization parameters)
or in high-dimensional settings.
This is mainly 
 due to the following reasons:
(1) The available computational resources constrain us to adopt  
a trial space with  limited approximation capacity 
   in the simulation, 
such as polynomials  of fixed degrees  (see e.g.~\cite{bender2008}) or  neural networks of fixed sizes
(see e.g.~\cite{e2017,fouque2019,germain2019}).
Hence, it is unclear whether the chosen trial space is 
rich enough to approximate the required conditional expectations
up to the desired accuracy.
(2)
 It is well-known that choosing a trial space with better approximation capacity
in the computation of conditional expectations 
may not lead to more accurate numerical solutions.
For example, a high-order  polynomial ansatz may lead to oscillatory solutions 
that blow up quickly for large spatial values,
and  
 neural networks with more complex structures 
in general result in more challenging optimization problems
in the regression steps
(see e.g.~\cite{gnoatto2020,ito2020}).
(3) Most existing numerical schemes for solving coupled (MV-)FBSDEs \eqref{eq:fbsde_conts_intro}
involve the Picard method, which solves for the backward components $(Y,Z)$ with a given  proxy of the forward component $X$
and then iterates
 (see \cite{delarue2006,bender2008,andrea2019,chassagneux2019}).
Unfortunately, 
sharp criteria for convergence of the Picard method are difficult to establish
since, on one hand, it is well-known that  the 
 Picard theorem only applies to the fully coupled system \eqref{eq:fbsde_conts_intro}
 with a sufficiently small maturity $T$ (see e.g.~\cite{andrea2019,chassagneux2019}),
while on the other hand,
empirical studies show that  the theoretical bound on the maturity  to ensure convergence is usually far too pessimistic 
\cite{germain2019}.

\paragraph{Our work.}
This paper consists of three parts.

\begin{itemize}[leftmargin=*]
\item
We propose   an {a posteriori} error estimator to quantify the  accuracy 
of  given numerical solutions
to  \eqref{eq:fbsde_conts_intro}.
These solutions  can be 
produced from an arbitrary time-stepping scheme,
an arbitrary numerical procedure for approximating conditional expectations
and an arbitrary discrete
approximation of  Brownian increments.
For 
a given  approximation  
 $(\hat{X}_{t_i},\hat{Y}_{t_i},\hat{Z}_{t_i})_{t_i\in \pi}$ 
on the grid $\pi=\{0=t_0<\ldots <t_N=T\}$ (generated by some algorithm), 
the   error estimator 
determines its accuracy by 
checking how well the given approximation  satisfies   \eqref{eq:fbsde_conts_intro}
running forward in time on the grid $\pi$:
\begin{align}\l{eq:error_estimator_intro}
\begin{split}
&\cE_\pi(\hat{X},\hat{Y},\hat{Z})
\\
&\coloneqq
\sE[|\hat{X}_0-\xi_0|^2]
+
\sE[|\hat{Y}_T-g(\hat{X}_T,\sP_{\hat{X}_T})|^2]
\\
&\quad +
\max_{0\le i\le N-1}
\sE
\bigg[
\bigg|
\hat{X}_{t_{i+1}}
-\hat{X}_{0}
-\sum_{j=0}^i
\left(
b(t_j,\hat{\Theta}_{t_j},\sP_{\hat{\Theta}_{t_j}})\Delta_j  +
\sigma (t_j,\hat{\Theta}_{t_j},\sP_{\hat{\Theta}_{t_j}})\, \Delta W_j
\right)
\bigg|^2
\bigg]\\
&\quad +
\max_{0\le i\le N-1}
\sE\bigg[
\bigg|
\hat{Y}_{t_{i+1}}-\hat{Y}_0
+\sum_{j=0}^{i}
\left(
f(t_{j},\hat{\Theta}_{t_j}, \sP_{\hat{\Theta}_{t_j}})\Delta_j- \hat{Z}_j\,\Delta W_j
\right)
\bigg|^2
\bigg],
\end{split}
\end{align}
where
$\hat{\Theta}_{t_i}=(\hat{X}_{t_i},\hat{Y}_{t_i},\hat{Z}_{t_i})$,
 $\Delta_i=t_{i+1}-t_i$ and $\Delta W_i=W_{t_{i+1}}-W_{t_i}$ for all $i=0,\ldots,N-1$.
The error estimator 
\eqref{eq:error_estimator_intro}
 naturally extends  the  {a posteriori} error estimator for 
 standard (decoupled)
 BSDEs in \cite{bender2013}
to  systems of fully coupled FBSDEs with mean field interaction,
and can be accurately evaluated by
plain Monte Carlo simulation;
see Section \ref{sec:numerical} for a detailed discussion
on the implementation. 
 
\item 
We prove --  under the standard monotonicity assumption  -- that
\eqref{eq:error_estimator_intro} 
yields upper and lower bounds of
 the  squared $L^2$-error between a given discrete approximation
 and the solution to an explicit forward Euler 
discretization of \eqref{eq:fbsde_conts_intro},
up to a  constant independent of the time stepsize and the given approximation
(see Theorems \ref{thm:error_efficient_discrete_exp} and \ref{thm:error_reliable_discrete_exp}).
We then show that 
the squared $L^2$-error 
between 
a  discrete approximation
and the  continuous-time solution $(X,Y,Z)$
to \eqref{eq:fbsde_conts_intro}
can be measured by \eqref{eq:error_estimator_intro}  along with 
the path regularity of $(X,Y,Z)$ 
 (see Theorem \ref{thm:a_posterior_conts}).
The path regularity term  vanishes as the time stepsize tends to zero,
and admits a first-order convergence rate under 
 certain structural  conditions.
These results 
indicate that 
numerical
solutions  with smaller residuals \eqref{eq:error_estimator_intro}
 are   more accurate,
and hence justify  the use of
 residual minimization algorithms 
(e.g., the   deep BSDE solvers in
\cite{carmona2019,  fouque2019,germain2019})
 for solving 
\eqref{eq:fbsde_conts_intro}
(see Corollary \ref{cor:deep_bsde}).

 \item
 We  finally verify the theoretical properties of the   {a posteriori}  estimator
through several numerical experiments.
Section \ref{sec:one_dimensional} studies a one-dimensional 
 coupled MV-FBSDE arising from 
a mean field game, 
for which  a hybrid scheme  consisting of 
the Markovian  iteration in \cite{bender2008}
and the least-squares Monte Carlo methods  in \cite{gobet2006}
is implemented 
to generate  numerical solutions.
We show that the  estimator  accurately  predicts  the  squared approximation errors 
for different choices of model parameters and discretization parameters,
no matter whether the hybrid scheme converges.
The error estimator \eqref{eq:error_estimator_intro}
also leads to more efficient algorithms with tailored hyper-parameters,
such as  the number of time steps, the number of simulation paths,
{and the number of Picard iterations}.
{
Section \ref{sec:multi-dimensional} studies multidimensional coupled MV-FBSDEs arising from the  optimal control of  
Cucker--Smale models,
whose 
 numerical solutions are computed using   neural network based BSDE solvers.
The results  show that the estimator  effectively predicts the true  
approximation error   and is robust with respect to  
 model parameters. 
}
\end{itemize}

\paragraph{Our approach and related works.}

{A posteriori} error analysis 
has  been performed in \cite{bender2013,bender2017}
for decoupled BSDEs 
(where  \eqref{eq:fbsde_conts_fwd_intro}
is independent of  $Y,Z$)
and  in \cite{han2018}
for weakly coupled FBSDEs 
(where  \eqref{eq:fbsde_conts_fwd_intro}
is independent of $Z$).
To the best of our knowledge,
 this is the first 
 {a posteriori} error estimator
 with rigorous error estimates 
 for  fully coupled (MV-)FBSDEs.
Moreover,
instead of 
merely estimating
the  accuracy 
at $t=0$
 as in \cite{bender2017}, 
 the estimator 
\eqref{eq:error_estimator_intro}  
yields upper and lower bounds for the global $L^2$-error of a given discrete approximation
 $(\hat{X},\hat{Y},\hat{Z})$
over the grid.
This subsequently allows for  measuring
 the accuracy of the  numerical Nash equilibria and optimal control strategies
 (see e.g.~\cite{andrea2019,carmona2019,chassagneux2019,germain2019})
 or the dynamic risk measures  \cite{gnoatto2020} computed over the whole interval.

A crucial step in analyzing  \eqref{eq:error_estimator_intro}  
 is to establish the well-posedness and stability of  
  a family of  coupled discrete-time MV-FBSDEs 
  (referred to as MV-FBS$\Delta$E) arising from 
discretizing   \eqref{eq:fbsde_conts_intro}  with a  forward Euler scheme. 
There are two main challenges in analyzing these discrete-time equations 
beyond those encountered in a continuous-time setting  \cite{peng1999,bensoussan2015}:
\begin{itemize}[leftmargin=*]
\item 
Adapting   the method of continuation  to coupled  MV-FBS$\Delta$Es
involves estimating  the   product
of   forward and backward processes on $[0,T]$,
which subsequently requires  controlling 
  the product of drift coefficients on each subinterval. 
  Note that such a  term   only appears in the  discrete-time setting,
  and  cannot be   controlled by the  monotonicity condition as in \cite{peng1999,bensoussan2015}.
 Here, we exploit
 a precise a priori
  estimate of the MV-FBS$\Delta$Es,
 and prove  that
the additional term  is of  magnitude $\cO(\max_i \Delta_i)$.
This allows for implementing  the  continuation method 
and subsequently concluding the desired well-posedness and stability of the  MV-FBS$\Delta$E
for all sufficiently fine grids
(see Section \ref{sec:fbsde}).

\item 
The error estimates allow
for numerical solutions generated from 
an arbitrary
 discrete approximation of Brownian increments
 and an arbitrary time-stepping scheme.
 This requires  establishing
the well-posedness  and stability of the forward Euler scheme 
  in a general setting by  allowing the driving noise to be a general discrete-time martingale,
  and by allowing the perturbation to be a general square-integrable process. 
 As   discrete-time martingales 
in general do not enjoy the predictable representation property,
the associated  MV-FBSDEs are not well-posed in terms of a solution triple $(X,Y,Z)$
(cf.~\eqref{eq:fbsde_conts_intro}).
Here we augment the solution with an additional martingale process that is 
strongly
orthogonal to the given discrete-time martingale,
and
construct  adapted solutions to MV-FBSDEs
based on the  Kunita--Watanabe decomposition.

\end{itemize}


\paragraph{Notation.}

Let  $T>0$ be  a given  terminal time
and  $(\Om, \cF,  \sP)$ be
 a given complete probability space
equipped with a complete and right-continuous filtration $\sF=\{\cF_t\}_{t\in [0, T]}$.
The filtration $\sF$ 
 is in general larger than 
the augmented filtration generated by the   driving noise of the system (i.e., the martingale $W$ in \eqref{eq:MV-fbsde_exp}),
and contains the information of all independently simulated sample paths of the  driving noise
that are used to obtain the numerical solutions. 
 All equalities and inequalities 
on a vector/matrix quantity 
are understood
componentwise in $\sP$-almost surely sense.

 For each
$N\in \sN$,
let 
$\cN=\{0,1,\ldots,N \}$
and
$\cN_{<N}=\{0,1,\ldots,N-1 \}$.
We  denote by
$\pi_N=\{t_i \}_{i\in \cN}$
 a uniform  partition of $[0,T]$
such that for all $i\in \cN$,
 $t_i=i\tau_N$
with the time stepsize $\tau_N=T/N$,\footnotemark
 by
$\sE_i[\cdot]$ the conditional expectation 
$\sE[\cdot\mid \cF_{t_i}]$
for $i\in\cN$,
and by $\Delta$ the difference operator 
such that
$\Delta U_{i} = U_{t_{i+1}}- U_{t_i}$
  for all $i\in \cN_{<N}$ and  
 processes $(U_t)_{0\le t\le T}$.
For simplicity, 
for each   $i\in \cN$
and  process $(U_t)_{0\le t\le T}$,
we   write $U_i=U_{t_i}$
if no confusion  occurs.

\footnotetext{In this paper, we work with a
uniform partition of $[0,T]$ to simplify the notation and to keep the focus on the main issues,
but similar results are valid for  nonuniform time-steps as well.}

For each  $n\in \sN$, we  denote by 
  $\sI_n$  the $n\t n$ identity matrix.
We  denote by $\la \cdot,\cdot\ra$
the usual inner product in a given Euclidean space
and by   $|\cdot|$ the norm induced by $\la \cdot,\cdot\ra$,
which in particular satisfy  for all 
$n,m,d\in \sN$
and
$\theta_1=(x_1,y_1,z_1),\theta_2=(x_2,y_2,z_2)\in \sR^n\t \sR^m\t \sR^{m\t d}$,
$\la z_1,z_2\ra =\tr(z^*_1z_2)$
and 
$\la \theta_1,\theta_2\ra =\la x_1,x_2\ra+\la y_1,y_2\ra+\la z_1,z_2\ra$,
where 
$\tr(\cdot)$ 
and
$(\cdot)^*$ denote the     trace and the transposition  of a matrix, respectively.

For each 
 $n,n'\in \sN$
 and $\sigma$-algebra $\cG\subset \cF$,
we 
 introduce the following spaces:
$L^2(\cG; \sR^n)$ is the space  of 
all $\cG$-measurable
$\sR^n$-valued 
square integrable random variables;
 $\cM^2(0, T; \sR^{n})$ is the space of all $\sF$-adapted 
 $\sR^n$-valued 
square integrable process;
$\cP_2(\sR^{n})$ is the set of 
square integrable
probability measures on $\sR^{n}$ 
endowed with the  2-Wasserstein distance defined by 
$$
\cW_2(\mu_1,\mu_2)
\coloneqq \inf_{\nu\in \Pi(\mu_1,\mu_2)} \left(\int_{\sR^{n}\t \sR^{n}}|x-y|^2\nu(dx,dy)\right)^{1/2},
\q \mu_1,\mu_2\in \cP_2(\sR^{n}),
$$
where $\Pi(\mu_1,\mu_2)$ is the set of all couplings of $\mu_1$ and $\mu_2$, i.e.,
$\nu\in \Pi(\mu_1,\mu_2)$ is a probability measure on $\sR^{n}\t \sR^{n}$ such that $\nu(\cdot\t \sR^{n})=\mu_1$ 
and $\nu(\sR^{n}\t \cdot)=\mu_2$.
Note that 
for all $n\in \sN, \mu_1,\mu_2\in \cP_2(\sR^{n})$,
$\cW^2_2(\mu_1,\mu_2)\le \sE[|X_1-X_2|^2]$,
 where $X_1$ and $X_2$ are $n$-dimensional random vectors having the distributions 
 $\mu_1$ and $\mu_2$, respectively.

\section{Well-posedness and stability of discrete MV-FBSDEs}\l{sec:fbsde}

This section  studies the  MV-FBS$\Delta$E
associated with the {a posteriori} error estimator 
\eqref{eq:error_estimator_intro}.
We  prove that the MV-\fbsde admits a unique adapted solution and 
establish an {a priori} stability estimate of  its solution with respect to the perturbation of coefficients.

For each $N\in\sN$,   consider the following MV-\fbsde
 on the time grid $\pi_N$: 
for all $i\in \cN_{<N}$,
\begin{subequations}\label{eq:MV-fbsde_exp}
\begin{align}
\Delta X^\pi_i&=b(t_{i},X^\pi_{i},Y^\pi_{i},Z^\pi_{i},\sP_{(X^\pi_{i},Y^\pi_{i},Z^\pi_{i})})\tau_N  
+\sigma (t_i,X^\pi_i,Y^\pi_i,Z^\pi_i,\sP_{(X^\pi_i,Y^\pi_i,Z^\pi_i)})\, \Delta W_i, 
\l{eq:fbsde_fwd_exp}\\
\Delta Y^\pi_i&=-f(t_{i},X^\pi_{i},Y^\pi_{i},Z^\pi_{i}, \sP_{(X^\pi_{i},Y^\pi_{i},Z^\pi_{i})})\tau_N+Z^\pi_i\,\Delta W_i+\Delta M^\pi_i,
\l{eq:fbsde_bwd_exp}\\
X^\pi_0&=\xi_0,\q Y_N=g(X^\pi_N,\sP_{X^\pi_N}),
\l{eq:fbsde_terminal_exp}
\end{align}
\end{subequations}
 where
 $\xi_0\in L^2(\cF_0;\sR^n)$,
 the solution processes 
 $X^\pi$, $Y^\pi$, $Z^\pi$ and $M^\pi$ take  values  in $\sR^n$, $\sR^m$, $\sR^{m\t d}$ and $\sR^m$, respectively,  
 the coefficients $(b, \sigma, f, g)$, referred as the \textit{generator} of the MV-\fbsde \eqref{eq:MV-fbsde_exp}, are (possibly random) functions with appropriate dimensions (see (H.\ref{assum:fbsde_discrete_exp}) for the precise conditions),
 and 
  $W=(W_t)_{t\in [0,T]}\in \cM^2(0, T; \sR^{d})$ is a given 
  (possibly piecewise-constant)
 martingale
process
satisfying for all 
$i\in \cN_{<N}$,
$\sE_i[\Delta W_i (\Delta W_i)^*]=\tau_N\sI_d$.
Above and hereafter, when there is no ambiguity, we will omit the dependence of $(b, \sigma, f, g)$ on $\om\in \Om$ for notational simplicity.

\begin{Remark}\l{rmk:predictable}
Both the $Z^\pi$ and  $M^\pi$ processes in \eqref{eq:MV-fbsde_exp} arise from applying the martingale representation theorem 
to obtain an $\sF$-adapted solution to
 \eqref{eq:MV-fbsde_exp}. 
 Note that
 we allow \eqref{eq:MV-fbsde_exp} to be driven by a general discrete martingale $W$,
 which represents the discrete
approximation of  Brownian increments
that are used to generate numerical solutions
(such as those  based on Gauss-Hermite quadrature formula
as  in \cite{picarelli2020}).
It is well-known that 
   martingale processes with jumps, in particular 
the discrete-time martingale $( W_i)_{i\in \cN}$,
in general do not enjoy the predictable representation property,
i.e.,  
for 
a given  martingale  $U\in \cM^2(0,T;\sR^m)$, there may not exist a  process $Z$ satisfying
$\Delta U_i=Z_i\Delta W_i$ and $Z_i\in L^2(\cF_{t_i};\sR^m)$ for all $i\in \cN_{<N}$.
Hence we augment the solution with another martingale process $M$ (see Definition \ref{def:MV-fbsde})
and  apply  Kunita--Watanabe decomposition (\cite[Theorem 10.18]{follmer2004})
   to construct  adapted solutions to
    \eqref{eq:MV-fbsde_exp}; 
    see Lemma \ref{lem:linear_existence} and  also \cite{bender2013,bielecki2015}.

In the case that 
 $W$ has the predictable representation property,
 such as
 Bernoulli processes with independent increments,
and $\sF$ is
 the augmented filtration
generated by $(W_t)_{t\in [0,T]}$ and an independent initial $\sigma$-field $\cF_0$,
then 
$M=0 $ on $[0,T]$ due to  the uniqueness of  the Kunita--Watanabe  decomposition.

\end{Remark}

Throughout this work,
we shall perform the analysis under the following  assumptions on the generator $(b,\sigma,f,g)$.

\begin{Assumption}\l{assum:fbsde_discrete_exp}
Let $n,m,d\in \sN$,
$T>0$,
and let
$b:\Om\t [0,T]\t \sR^n\t \sR^m\t \sR^{m\t d}\t \cP_2(\sR^{n+m+md})\to \sR^n $,
$\sigma:\Om\t [0,T] \t \sR^n\t \sR^m\t \sR^{m\t d}\t \cP_2(\sR^{n+m+md})\to \sR^{n\t d} $,
$f:\Om\t [0,T]\t \sR^n\t \sR^m\t \sR^{m\t d}\t \cP_2(\sR^{n+m+md})\to \sR^{m} $
and 
$g:\Om\t \sR^n\t \cP_2(\sR^{n})\to \sR^{m}$
be measurable functions.
\begin{enumerate}[(1)]
\item{(Monotonicity.)}\l{item:monotone_exp}
There exists 
a full-rank matrix $G\in \sR^{m\t n}$
and
constants
 $\a\ge 0$,
  $\beta_1,\beta_2\ge 0$ 
with 
$\a+\beta_1>0$ and
$\beta_1+\beta_2>0$
such that
$\beta_1>0$ (resp.~$\a>0,\beta_2>0$) 
 when $m<n$ (resp.~$m>n$),
and
it holds
{
for $\sP$-a.s.~$\omega\in \Om$, }
all $t\in [0,T]$, $i\in \{1,2\}$, 
 $\Theta_i\coloneqq (X_i,Y_i,Z_i)\in L^2(\cF; \sR^n\t \sR^m\t \sR^{m\t d})$,
\begin{align}
&\sE[\la b(t,\Theta_1,\sP_{\Theta_1})-b(t,\Theta_2,\sP_{\Theta_2}), G^*(\delta Y)\ra]
+\sE[\la \sigma(t,\Theta_1,\sP_{\Theta_1})-\sigma(t,\Theta_2,\sP_{\Theta_2}), G^*(\delta  Z)\ra]
\nb
\\
&\quad + \sE[\la -f(t,\Theta_1,\sP_{\Theta_1})+f(t,\Theta_2,\sP_{\Theta_2}), G(\delta X)\ra]
\nb
\\
&\quad \le -\beta_1(\sE[|G^*(\delta Y)|^2]+\sE[|G^*(\delta Z)|^2]) -\beta_2\sE[|G(\delta X)|^2],
\label{eq:monotone_b_sigma_f}
\\
&\sE[\la g(X_1,\sP_{X_1})-g(X_2,\sP_{X_2}), G(\delta X)\ra]
\ge  \a\,\sE[|G(\delta X)|^2],
\nb
\end{align}
with $(\delta X,\delta Y,\delta Z) \coloneqq (X_1-X_2,Y_1-Y_2,Z_1-Z_2)$.

%

 \item{(Lipschitz continuity.)}\l{item:lipschitz_exp}
There exists a constant $L\ge 0$ such that
{
 for $\sP$-a.s.~$\omega\in \Om$, }
 all $t\in [0,T]$,
$i\in \{1,2\}$, $\theta_i\coloneqq (x_i,y_i,z_i)\in  \sR^n\t \sR^m\t \sR^{m\t d}$,
 $\mu_i\in \cP_2(\sR^{n+m+md})$
and  $\nu_i\in \cP_2(\sR^n)$,
\begin{align*}
|\phi(t,\theta_1,\mu_1)-\phi(t,\theta_2,\mu_2)|
&\le L(|\theta_1-\theta_2|+\cW_2(\mu_1,\mu_2))
\q \forall \phi=b,\sigma,f,
\\
|g(x_1,\nu_1)-g(x_2,\nu_2)|
&\le L(|x_1-x_2|+\cW_2(\nu_1,\nu_2)).
\end{align*}

 \item {(Integrability.)}\l{item:integrable_exp}
 $\xi_0\in L^2(\cF_0;\sR^n)$,
 $b(\cdot,\cdot,0,\delta_0)\in \cM^2(0, T; \sR^{n})$,
  $\sigma(\cdot,\cdot,0,\delta_0)\in \cM^2(0, T; \sR^{n\t d})$,
 $f(\cdot,\cdot,0,\delta_0)\in \cM^2(0, T; \sR^{m})$
 and $g(\cdot,0,\delta_0)\in L^2(\cF_T;\sR^m)$,
 where 
 $\delta_0\in \cP_2(\sR^n\t \sR^m\t \sR^{m\t d})$
 is
  the Dirac measure supported at $0$.
 \end{enumerate}
\end{Assumption}

\begin{Remark}

Assumption (H.\ref{assum:fbsde_discrete_exp}) is the same as Assumption (A.1) in \cite{bensoussan2015},
which 
has also been imposed in   \cite{peng1999} for  coupled FBSDEs
without mean-field interaction.
It allows for proving the  \emph{stability}   of 
(MV-)FBSDEs
with respect to perturbations in   coefficients (see Proposition \ref{thm:stab_exp}),
which subsequently yields the 
well-posedness of  
fully coupled (MV-)FBSDEs with an arbitrary terminal time $T$.
This assumption 
can be naturally satisfied by linear MV-FBSDEs which arise from 
applying 
the stochastic maximum principle approach
to solve linear-quadratic stochastic control problems and mean field games,
where the monotonicity of the generator  is inherited  from the concavity of the Hamiltonian 
(see e.g.~\cite{peng1999,bensoussan2015} for more details).
The matrix $G\in \sR^{m\t n}$ in  (H.\ref{assum:fbsde_discrete_exp}\ref{item:monotone_exp})
not only 
matches the dimensions of the processes $X$ and $Y$
in the monotonicity condition,
but also helps to handle the indefiniteness of Hamiltonian systems
arising from zero-sum differential games
(see e.g.~Example 3.4 in \cite{peng1999}).

{
It is worth noting that the  stability and well-posedness of continuous-time MV-FBSDE 
\eqref{eq:fbsde_conts_intro}
 can be established by relaxing  Assumption (H.\ref{assum:fbsde_discrete_exp}\ref{item:monotone_exp})
 with a generalised monotonicity condition.
 This condition  
  replaces  the term $\sE[|G^*(\delta Y)|^2]+\sE[|G^*(\delta Z)|^2]$ 
in \eqref{eq:monotone_b_sigma_f}
  by a term  $\phi(t, \Theta_1,\Theta_2)\in [0,\infty)$.
The generalised monotonicity condition has been verified for nonlinear (MV-)FBSDEs arising from linear-convex control problems in \cite[Lemma 2.3]{guo2023reinforcement},  and \cite[Proposition 3.3]{reisinger2020optimal}
(see also \cite{carmona2015}). 
We anticipate that under this   condition, one can establish the stability of the discrete-time FBSDE  \eqref{eq:MV-fbsde_exp}
 and carry out a similar a-posterior error analysis. A complete analysis in this direction is left for future research. 
}

\end{Remark}
%

We now   state the precise definition of a solution to  MV-\fbsde \eqref{eq:MV-fbsde_exp}. 
\begin{Definition}\l{def:MV-fbsde}
For  each $N\in\sN$,
let $\cS_N$ be the space  of 
all $4$-tuples 
$(X,Y, Z, M)\in \cM^2(0, T ;\sR^n\t \sR^m\t \sR^{m\t d}\t \sR^m)$
defined on $\pi_N$,
which 
are  constant on the intervals $[t_i,t_{i+1})$ for $i\in \cN_{<N}$,
and 
satisfy
the conditions that 
 $M_0=0$ and  $M$ is a martingale process strongly
orthogonal to $W$,\footnotemark
and let $\cS^0_N$ be the subspace of 
$(X,Y, Z, M)\in \cS_N$ for which $M\equiv 0$.

Then
for each $N\in \sN$,
 we say a $4$-tuple 
$(X,Y, Z, M)\in 
\cS_N$
is a solution to MV-\fbsde \eqref{eq:MV-fbsde_exp} 
defined on  $\pi_N$
if it satisfies  the system \eqref{eq:MV-fbsde_exp}.
We say a triple 
$(X,Y, Z)\in  \cS^0_N$
is a solution to MV-\fbsde \eqref{eq:MV-fbsde_exp} 
defined on  $\pi_N$
if 
$(X,Y, Z, 0)\in  \cS_N$
is a solution.
\footnotetext{
We say that a 
$\sR^m$-valued
martingale process $M$ is strongly orthogonal to 
$W$ if 
the process $(M_tW^*_t)_{0\le t\le T}$ is a martingale.
}
\end{Definition}

To establish that \eqref{eq:MV-fbsde_exp} admits a unique solution in $\cS_N$,
we  adapt   the  continuation argument in \cite{peng1999,bensoussan2015} to the present discrete-time setting.
To this end, 
we 
consider a family of MV-FBS$\Delta$Es on the grid $\pi_N$ parameterized by $\lambda\in [0,1]$: 
for all $i\in \cN_{<N}$,
\begin{align}\label{eq:MV-fbsde_exp_lambda}
\begin{split}
\Delta X_i&=
[(1-\lambda)\b_1 (-G^*Y_{i})+\lambda b(t_{i},\Theta_{i},\sP_{\Theta_{i}})+\phi_{i}]\tau_N  
\\
&\quad
+[(1-\lambda)\b_1 (-G^*Z_{i})+\lambda \sigma(t_{i},\Theta_{i},\sP_{\Theta_{i}})+\psi_{i}] \, \Delta W_i, 
\\
\Delta Y_i&=-[(1-\lambda)\b_2GX_i+\lambda f(t_{i},\Theta_{i}, \sP_{\Theta_{i}})+\gamma_i]\tau_N+Z_i\,\Delta W_i+\Delta M_i,
\\
X_0&=\xi_0,\q Y_N=(1-\lambda)GX_N+\lambda g(X_N,\sP_{X_N})+\eta,
\end{split}
\end{align}
where
 $G\in \sR^{m\t n}, \b_1,\b_2\ge 0$ are given in (H.\ref{assum:fbsde_discrete_exp}),
 $\Theta_i=(X_i,Y_i,Z_i)$ for all $i\in \cN$, 
$(\phi,\psi,\gamma)\in \cM^2(0,T; \sR^n\t \sR^{n\t d}\t \sR^m)$ are given processes,
and $\eta\in L^2(\cF_T;\sR^m)$ is a given random variable.
It is clear that the well-posedness of \eqref{eq:MV-fbsde_exp_lambda} with $\lambda=1$ 
implies that of  \eqref{eq:MV-fbsde_exp}.

We first establish a stability result of solutions to  
\eqref{eq:MV-fbsde_exp_lambda}
under  (H.\ref{assum:fbsde_discrete_exp}),
which
 extends \cite[Theorem 5]{bensoussan2015} 
to the present  setting
with  a general discrete-time martingale $W$.
Applying the following proposition  with different choices of 
$\lambda$, $ (\phi,\psi,\gamma,\eta)$,
  $(\bar{b},\bar{\sigma},\bar{f},\bar{g})$
 and $(\bar{\phi},\bar{\psi},\bar{\gamma},\bar{\eta},\bar{\xi})$
 allows us to 
establish
 the well-posedness of \eqref{eq:MV-fbsde_exp_lambda}
 via the method of continuation 
 and to prove the desired {a posteriori} error estimate for \eqref{eq:MV-fbsde_exp}
 in Section \ref{sec:a_posteriori_discrete}.

For the sake of readability,
 the detailed proof of Proposition \ref{thm:stab_exp}
is  given 
in Appendix \ref{appendix:proof_section_discretebsde},
as it involves several technical and lengthy calculations.

\begin{Proposition}\l{thm:stab_exp}
Suppose the generator $(b,\sigma,f,g)$ satisfies (H.\ref{assum:fbsde_discrete_exp}), 
and let $\b_1,\b_2$ and $G$ be the constants in  (H.\ref{assum:fbsde_discrete_exp}\ref{item:monotone_exp}).
Then there exists $N_0\in \sN$  and $C>0$ such that,
 for all $N\in \sN\cap[ N_0,\infty)$, $\lambda_0\in [0,1]$,
all 4-tuples
$(X,Y, Z, M)\in 
\cS_N$
satisfying \eqref{eq:MV-fbsde_exp_lambda}
defined on  $\pi_N$ 
with 
$\lambda=\lambda_0$,
 generator $(b,\sigma,f,g)$ and some 
$(\phi,\psi,\gamma)\in \cM^2(0,T; \sR^n\t \sR^{n\t d}\t \sR^m)$, 
 $\eta\in L^2(\cF_T;\sR^m)$,
 $\xi_0\in L^2(\cF_0;\sR^n)$,
and 
all 4-tuples
$(\bar{X},\bar{Y}, \bar{Z}, \bar{M})\in 
\cS_N$
satisfying  \eqref{eq:MV-fbsde_exp_lambda}
defined on  $\pi_N$ 
with 
$\lambda=\lambda_0$,
another generator $(\bar{b},\bar{\sigma},\bar{f},\bar{g})$ 
satisfying 
(H.\ref{assum:fbsde_discrete_exp}\ref{item:integrable_exp}),
 and some 
$(\bar{\phi},\bar{\psi},\bar{\gamma})\in \cM^2(0,T; \sR^n\t \sR^{n\t d}\t \sR^m)$, 
 $\bar{\eta}\in L^2(\cF_T;\sR^m)$,
  $\bar{\xi}_0\in L^2(\cF_0;\sR^m)$,
\begin{align*}
\begin{split}
&\max_{i\in \cN}
\left(
\sE[|{X}_{i}
-\bar{X}_{i}|^2]
+
\sE[|{Y}_{i}
-\bar{Y}_{i}|^2]
\right)
+
\sum_{i=0}^{N-1}
\sE
[
|{Z}_{i}
-\bar{Z}_{i}|^2]\tau_N
+
\sE[|{M}_{N}
-\bar{M}_{N}|^2]
\\
&\le
C\bigg\{
\sE[| \xi_{0}-\bar{\xi}_0|^2]
+
\sE[ |\lambda_0(g(\bar{X}_N,\sP_{\bar{X}_N})-\bar{g}(\bar{X}_N,\sP_{\bar{X}_N}))
+ \eta-\bar{\eta}|^2]
\\
&\quad
+\sum_{i=0}^{N-1}
\bigg(
\sE[|\lambda_0({f}(t_{i},\bar{\Theta}_i,\sP_{\bar{\Theta}_i})-\bar{f}(t_{i},\bar{\Theta}_i,\sP_{\bar{\Theta}_i}))+\gamma_i-\bar{\gamma}_i|^2]\tau_N
\\
&\quad
+
\sE[|\lambda_0({b}(t_{i},\bar{\Theta}_{i},\sP_{\bar{\Theta}_{i}})-\bar{b}(t_{i},\bar{\Theta}_{i},\sP_{\bar{\Theta}_{i}})+ \phi_{i}- \bar{\phi}_{i}|^2]\tau_N
\\
&
\quad 
+
\sE [|\lambda_0({\sigma}(t_i,\bar{\Theta}_i,\sP_{\bar{\Theta}_i})-\bar{\sigma}(t_i,\bar{\Theta}_i,\sP_{\bar{\Theta}_i}))+\psi_{i}-\bar{\psi}_{i}|^2]\tau_N
\bigg)
\bigg\},
\end{split}
\end{align*}
where $\bar{\Theta}_i\coloneqq (\bar{X}_i,\bar{Y}_i, \bar{Z}_i)$ for all $i\in \cN_{<N}$.

\end{Proposition}

A direct consequence of Proposition \ref{thm:stab_exp}
is the uniqueness of solutions to \eqref{eq:MV-fbsde_exp_lambda},
which can be shown by setting 
$(\bar{\phi},\bar{\psi},\bar{\gamma},\bar{\eta},\bar{\xi}_0)=(\phi,\psi,\gamma,\eta,{\xi}_0)$
and $(\bar{b},\bar{\sigma},\bar{f},\bar{g})=({b},{\sigma},{f},{g})$ 
in the statement of Proposition \ref{thm:stab_exp}.

\begin{Corollary}\l{cor:uniqueness_bwd_imp}
Suppose  (H.\ref{assum:fbsde_discrete_exp}) holds.
Then
there exists $N_0\in \sN$ such that
it holds for all  $N\in \sN\cap[N_0,\infty)$, $\lambda_0\in [0,1]$,
 $(\phi,\psi,\gamma)\in \cM^2(0,T; \sR^n\t \sR^{n\t d}\t \sR^m)$, 
 $\eta\in L^2(\cF_T;\sR^m)$,
  $\xi_0\in L^2(\cF_0;\sR^n)$
that 
 \eqref{eq:MV-fbsde_exp_lambda} with $\lambda=\lambda_0$ admits at most one solution in $\cS_N$.
\end{Corollary}

We proceed to prove the existence of solutions to \eqref{eq:MV-fbsde_exp}.
The following lemma
 constructs solutions to the linear MV-\fbsde \eqref{eq:MV-fbsde_exp_lambda} with $\lambda=0$,

\begin{Lemma}\l{lem:linear_existence}
Let $\b_1,\b_2\ge 0$, $G\in \sR^{m\t n}$ be a full-rank matrix
and 
  $\xi_0\in L^2(\cF_0;\sR^n)$.
Then  it holds for all  $N\in \sN$,
$(\phi,\psi,\gamma)\in \cM^2(0,T; \sR^n\t \sR^{n\t d}\t \sR^m)$, 
 $\eta\in L^2(\cF_T;\sR^m)$
  that 
 \eqref{eq:MV-fbsde_exp_lambda} with $\lambda=0$ admits a solution in $\cS_N$.
\end{Lemma}

The proof is given in 
Appendix \ref{appendix:proof_section_discretebsde}.
Compared to  \cite[Lemma 2.5]{peng1999},
 the analysis of discrete-time equations 
has two main   difficulties: 
(1) In contrast to  linear FBSDEs,
 solutions to \fbsde are  constant on each subinterval,
and hence cannot be obtained based on
  differential Riccati
equations.
Here we reduce  the linear MV-\fbsde
into a class of semi-implicit time-discretized  Riccati equations,
and 
prove 
these equations have 
symmetric positive definite solutions via induction; 
(2) 
Due to the lack of predictable representation property of 
the discrete-time martingale $W$ 
(see  Remark \ref{rmk:predictable}),
it is essential  
to augment the  solution with an additional  martingale process $M$
as in Definition \ref{def:MV-fbsde},
whose  existence is achieved by the   Kunita--Watanabe decomposition.

The following proposition extends  the well-posedness of 
\eqref{eq:MV-fbsde_exp_lambda} with $\lambda=\lambda_0$
to 
that of 
\eqref{eq:MV-fbsde_exp_lambda} with $\lambda\in [\lambda_0,\lambda_0+c]$,
for some $c>0$, independent of $\lambda_0$.

\begin{Proposition}\l{prop:moc_bwd_imp}
Suppose  (H.\ref{assum:fbsde_discrete_exp}) holds,
 let $\b_1,\b_2$ and $G$ be the constants in  (H.\ref{assum:fbsde_discrete_exp}\ref{item:monotone_exp}),
  $N_0\in \sN$ be  the  natural number in Proposition \ref{thm:stab_exp}
  and $N\in\sN\cap[ N_0,\infty)$.
Assume further that there exists $\lambda_0\in [0,1)$ satisfying
for any given 
$(\bar{\phi},\bar{\psi},\bar{\gamma})\in \cM^2(0,T; \sR^n\t \sR^{n\t d}\t \sR^m)$ and  $\bar{\eta}\in L^2(\cF_T;\sR^m)$
that \eqref{eq:MV-fbsde_exp_lambda} with $\lambda=\lambda_0$
and $(\phi,\psi,\gamma,\eta)=(\bar{\phi},\bar{\psi},\bar{\gamma},\bar{\eta})$
 admits a unique solution in $\cS_N$.
Then there exists $c_0\in (0,1)$, depending only on the constants 
$T,L,G,\a,\b_1,\b_2$
in (H.\ref{assum:fbsde_discrete_exp}),
such that
it holds 
 for all $\tilde{\lambda}\in [\lambda_0,\lambda_0+c_0]\cap [0,1]$,
$(\bar{\phi},\bar{\psi},\bar{\gamma})\in \cM^2(0,T; \sR^n\t \sR^{n\t d}\t \sR^m)$ and  $\bar{\eta}\in L^2(\cF_T;\sR^m)$
that
 \eqref{eq:MV-fbsde_exp_lambda} 
 with $\lambda=\tilde{\lambda}$
and $(\phi,\psi,\gamma,\eta)=(\bar{\phi},\bar{\psi},\bar{\gamma},\bar{\eta})$
 admits a unique solution in $\cS_N$.
\end{Proposition}

\begin{proof}
Throughout this proof,
let $(\bar{\phi},\bar{\psi},\bar{\gamma})\in \cM^2(0,T; \sR^n\t \sR^{n\t d}\t \sR^m)$ and  $\bar{\eta}\in L^2(\cF_T;\sR^m)$
 be fixed, 
and let $\cS_N$ be the space of  piecewise-constant processes on $\pi_N$ defined as in 
Definition \ref{def:MV-fbsde},
which is a Banach space equipped with  the norm $\|\cdot\|_{\cS_N}$ defined as
$$
\|(x,y,z,m)\|_{\cS_N}\coloneqq
\bigg(\max_{i\in \cN}\big(\sE[|x_i|^2]+\sE[|y_i|^2]\big)+\sum_{i=0}^{N-1}\sE[|z_i|^2]\tau_N+\sE[|m_N|^2]\bigg)^{1/2},
\q (x,y,z,m)\in \cS_N.
$$
For each $c\in (0,1)$, let $\cI_{\lambda_0+c}:\cS_N\to \cS_N$ 
be the mapping 
such that for all $(x,y,z,m)\in \cS_N$,
$\cI_{\lambda_0+c}(x,y,z,m)=(X,Y,Z,M)\in \cS_N$ is the unique solution to the following MV-\fbsde defined on $\pi_N$:
for all $i\in \cN_{<N}$,
\begin{align}
\begin{split}
\Delta X_i&=
[(1-\lambda_0)\b_1 (-G^*Y_{i})+\lambda_0 b(t_{i},\Theta_{i},\sP_{\Theta_{i}})
+\phi^c_{i}]\tau_N  
\\
&\quad
+[
(1-\lambda_0)\b_1 (-G^*Z_{i})+\lambda_0 \sigma(t_{i},\Theta_{i},\sP_{\Theta_{i}})
+\psi^c_i] \, \Delta W_i, 
\\
\Delta Y_i&=-[
(1-\lambda_0)\b_2GX_i+\lambda_0 f(t_{i},\Theta_{i}, \sP_{\Theta_{i}})
+\gamma^c_i]\tau_N+Z_i\,\Delta W_i+\Delta M_i,
\\
X_0&=\xi_0,\q Y_N= (1-\lambda_0)G X_N + \lambda_0g(X_N,\sP_{X_N})+\eta^c,
\end{split}
\end{align}
where  $\Theta=(X,Y,Z)$, 
and for each $\theta=(x,y,z)$,
$\phi^c_{i}\coloneqq c (\b_1 G^*y_{i}+ b(t_{i},\theta_{i},\sP_{\theta_{i}}))+\bar{\phi}_{i}$,
$\psi^c_i \coloneqq c (\b_1 G^*z_{i}+ \sigma (t_{i},\theta_{i},\sP_{\theta_{i}}))+\bar{\psi}_{i}$,
$\gamma^c_i\coloneqq c (-\b_2 G z_{i}+ f (t_{i},\theta_{i},\sP_{\theta_{i}}))+\bar{\gamma}_i$
and $\eta^c \coloneqq c (-Gx_N+g(x_N,\sP_{x_N}))+\bar{\eta}$.
The well-posedness assumption of 
\eqref{eq:MV-fbsde_exp_lambda} with $\lambda=\lambda_0$
and (H.\ref{assum:fbsde_discrete_exp})
 ensure that the mapping $\cI_{\lambda_0+c}$ is well-defined for all $c>0$.

We now show that there exists a constant $c_0\in (0,1)$, 
depending only on the constants in  (H.\ref{assum:fbsde_discrete_exp}),
such that $\cI_{\lambda_0+c}:\cS_N\to \cS_N$ is a contraction for  
all $c\in (0, c_0]$.
Let $(\hat{x},\hat{y},\hat{z},\hat{m}), (\tilde{x}, \tilde{y}, \tilde{z}, \tilde{m})\in \cS_N$ be  given, 
 $(\hat{X},\hat{Y},\hat{Z},\hat{M})=\cI_{\lambda_0+c}(\hat{x},\hat{y},\hat{z},\hat{m})$
 and
$(\tilde{X}, \tilde{Y}, \tilde{Z}, \tilde{M})=\cI_{\lambda_0+c}(\tilde{x}, \tilde{y}, \tilde{z}, \tilde{m})$.
By applying Proposition \ref{thm:stab_exp} 
with $\lambda=\lambda_0$, $(\phi,\psi,\gamma,\eta)=(\hat{\phi}^c,\hat{\psi}^c,\hat{\gamma}^c,\hat{\eta}^c)$,
$(\bar{\phi},\bar{\psi},\bar{\gamma},\bar{\eta})=(\tilde{\phi}^c,\tilde{\psi}^c,\tilde{\gamma}^c,\tilde{\eta}^c)$,
$\bar{\xi}_0=\xi_0$
and $(\bar{b},\bar{\sigma},\bar{f},\bar{g})=({b},{\sigma},{f},{g})$,
there exists   $C>0$, depending only on  constants in  (H.\ref{assum:fbsde_discrete_exp}),
such that
\begin{align*}
\begin{split}
&
\|(\hat{X}-\tilde{X},
\hat{Y}-\tilde{Y},\hat{Z}-\tilde{Z},\hat{M}-\tilde{M})\|^2_{\cS_N}
\\
&\le
C\bigg\{
\sE[ | \hat{\eta}^c-\tilde{\eta}^c|^2]
+\sum_{i=0}^{N-1}
\bigg(
\sE[| \hat{\gamma}^c_i-\tilde{\gamma}^c_i|^2]\tau_N
+\sE[|\hat{\phi}^c_{i}-\tilde{\phi}^c_{i}|^2]\tau_N
+
\sE[|\hat{\psi}^c_i-\tilde{\psi}^c_i|^2]\tau_N
\bigg)
\bigg\}
\\
&\le c^2C\|(\hat{x}-\tilde{x},
\hat{y}-\tilde{y},\hat{z}-\tilde{z},\hat{m}-\tilde{m})\|^2_{\cS_N}.
\end{split}
\end{align*}
Hence we see for $c_0=1/{\sqrt{2C}}>0$ and $c\in (0,c_0]$ that
$\cI_{\lambda_0+c}:\cS_N\to \cS_N$ is a contraction,
which together with the  Banach fixed point theorem  implies that \eqref{eq:MV-fbsde_exp_lambda} with $\lambda\in [\lambda_0,\lambda_0+c]\cap [0,1]$
and $(\phi,\psi,\gamma,\eta)=(\bar{\phi},\bar{\psi},\bar{\gamma},\bar{\eta})$
  admits a unique solution.
\end{proof}

Combining Corollary \ref{cor:uniqueness_bwd_imp}, Lemma \ref{lem:linear_existence},
 and Proposition \ref{prop:moc_bwd_imp}
implies the well-posedness of  \eqref{eq:MV-fbsde_exp}.

\begin{Theorem}
\l{thm:wp_bwd_imp}
Suppose  (H.\ref{assum:fbsde_discrete_exp}) holds.
Then for  all sufficiently large $N\in \sN$, 
 \eqref{eq:MV-fbsde_exp} admits a unique solution in $\cS_N$.
\end{Theorem}

%
%

\section{A posteriori  estimates for  discrete FBSDEs}\l{sec:a_posteriori_discrete}

In this section,
we  carry out the {a posteriori} error analysis 
in a discrete-time setting.
In particular, 
for any  given 
4-tuple 
$(\hat{X},\hat{Y},\hat{Z}, \hat{M})\in\cS_N$
generated by an arbitrary numerical scheme on the grid $\pi_N$,
we  derive  a computable  bound on the $L^2$-error between the approximation $(\hat{X},\hat{Y},\hat{Z}, \hat{M})$ 
 and the solution $(X^\pi,Y^\pi, Z^\pi, M^\pi)$ to \eqref{eq:MV-fbsde_exp}, 
 which requires only
knowledge of the given approximation and the data $(b, \sigma, f, g)$.
We  also demonstrate the  reliability and efficiency of
the proposed {a posteriori} error estimator.

More precisely, for any given
time grid $\pi_N$
and
numerical approximation  $(\hat{X},\hat{Y},\hat{Z}, \hat{M})\in\cS_N$,
 we  consider the following error estimator  on the grid $\pi_N$:
\begin{align}\l{eq:error_law_exp}
\begin{split}
&\cE_\pi(\hat{X},\hat{Y},\hat{Z}, \hat{M})
\\
&\coloneqq
\sE[|\hat{X}_0-\xi_0|^2]
+
\sE[|\hat{Y}_N-g(\hat{X}_N,\sP_{\hat{X}_N})|^2]
\\
&\quad +
\max_{i\in \cN_{<N}}
\sE
\bigg[
\bigg|
\hat{X}_{i+1}
-\hat{X}_{0}
-\sum_{j=0}^i
\left(
b(t_{j},\hat{\Theta}_{j},\sP_{\hat{\Theta}_{j}})\tau_N  +
\sigma (t_j,\hat{\Theta}_j,\sP_{\hat{\Theta}_j})\, \Delta W_j
\right)
\bigg|^2
\bigg]\\
&\quad +
\max_{i\in \cN_{<N}}
\sE\bigg[
\bigg|
\hat{Y}_{i+1}-\hat{Y}_0
+\sum_{j=0}^{i}
\left(
f(t_{j},\hat{\Theta}_{j}, \sP_{\hat{\Theta}_{j}})\tau_N- \hat{Z}_j\,\Delta W_j
\right)
- \hat{M}_{i+1}
\bigg|^2
\bigg]
\end{split}
\end{align}
with $\hat{\Theta}=(\hat{X},\hat{Y},\hat{Z})$.
Observe that 
 \eqref{eq:error_law_exp}
 takes a  more general form than  \eqref{eq:error_estimator_intro}, 
and takes into account   numerical approximations of the orthogonal martingale $M$.
It reduces to \eqref{eq:error_estimator_intro}
for numerical solution $(\hat{X},\hat{Y},\hat{Z})\in\cS^0_N$
 (with $\hat{M}\equiv 0$).

The estimator \eqref{eq:error_law_exp} extends the error criterion 
proposed for classical  BS$\Delta$Es in \cite{bender2013}
to 
fully coupled FBS$\Delta$Es \eqref{eq:MV-fbsde_exp}
with random initial data and mean field interaction.
Intuitively, the first term in \eqref{eq:error_law_exp} quantifies  the squared $L^2$-error of the $\hat{X}$-component at the initial time $t=0$,
the second term quantifies the  squared $L^2$-error of the $\hat{Y}$-component at the terminal time $t=T$,
and the last two terms measure 
the consistency of the approximation to the difference equations
\eqref{eq:fbsde_fwd_exp} and \eqref{eq:fbsde_bwd_exp}
defined on the time grid $\pi_N$.
In practice, 
\eqref{eq:error_law_exp}
 can be accurately evaluated
by 
approximating the expectations 
via Monte Carlo simulation
and 
by estimating the law of $({\hat{\Theta}_i})_{i\in \cN_{<N}}$ via  particle approximations;
see Section \ref{sec:numerical} for more details on the practical implementation 
of the {a posteriori} error estimator.

\color{black}

The remaining part of the section is  devoted to
proving the
 efficiency (see Theorem \ref{thm:error_efficient_discrete_exp})
and reliability  
(see Theorem \ref{thm:error_reliable_discrete_exp})
of 
  \eqref{eq:error_law_exp} 
for \eqref{eq:MV-fbsde_exp}. 
Recall that  an
{a posteriori} error estimator is   said to be efficient 
if an inequality of
the form ``error estimator $\ge$ tolerance" implies that the true error is
also greater than the tolerance possibly up to an multiplicative
constant,
while an
 {a posteriori} error estimator is said to be reliable if 
 an inequality of the form ``error
estimator $\le$ tolerance" implies that the true error is also less than the
tolerance up to another multiplicative constant.
Hence, as an efficient and reliable  error estimator,
the quantity \eqref{eq:error_law_exp} is equivalent to the squared $L^2$-error between $(\hat{X},\hat{Y},\hat{Z}, \hat{M})$ and 
the solution $(X^\pi,Y^\pi, Z^\pi, M^\pi)$ to \eqref{eq:MV-fbsde_exp}.

We start by showing that the  error estimator \eqref{eq:error_law_exp} is  efficient.
Note that the following theorem in fact holds for any  time grid $\pi_N$,
as long as the MV-\fbsde \eqref{eq:MV-fbsde_exp} admits a solution in $\cS_N$.

\begin{Theorem}\l{thm:error_efficient_discrete_exp}
Suppose (H.\ref{assum:fbsde_discrete_exp}\ref{item:lipschitz_exp}) 
holds. Then there exists a constant $C>0$,
{
depending only on $T$ and $L$ in (H.\ref{assum:fbsde_discrete_exp}\ref{item:lipschitz_exp}),}
such that
 for all  
$N\in\sN$ and every
 4-tuple of processes 
$(\hat{X},\hat{Y},\hat{Z}, \hat{M})\in \cS_N$,
\begin{align*}
\begin{split}
&\max_{i\in \cN}
\left(
\sE[|\hat{X}_{i}
-{X}^\pi_{i}|^2]
+
\sE[|\hat{Y}_{i}
-{Y}^\pi_{i}|^2]
\right)
+
\sum_{i=0}^{N-1}
\sE
[
|\hat{Z}_{i}
-{Z}^\pi_{i}|^2]\tau_N
+
\sE[|\hat{M}_{N}
-{M}^\pi_{N}|^2]
\\
&\ge
\cE_\pi(\hat{X},\hat{Y},\hat{Z}, \hat{M})/C,
\end{split}
\end{align*}
where $(X^\pi,Y^\pi, Z^\pi, M^\pi)\in \cS_N$
is a solution to MV-\fbsde \eqref{eq:MV-fbsde_exp}
defined on $\pi_N$. 

\end{Theorem}

\begin{proof}
Throughout this proof, 
let 
 $N\in\sN$ 
 and
 $(\hat{X},\hat{Y},\hat{Z}, \hat{M})\in \cS_N$
 be fixed.
We shall omit the superscript $\pi$ of  $(X^\pi,Y^\pi, Z^\pi, M^\pi)$ 
 for notational simplicity.
Let
$\hat{\Theta}=(\hat{X}, \hat{Y}, \hat{Z}),\Theta=(X, Y, Z)$,
 $(\delta \Theta, \delta X,\delta Y, \delta Z, \delta M)= ( \hat{ \Theta}-\Theta, \hat{X}-X, \hat{Y}-Y, \hat{Z}-Z, \hat{M}-M)$
and for each $t\in [0,T]$, $\phi=b, \sigma, f$
let $\delta \phi(t)= \phi(t, \hat{\Theta}_t,\sP_{\hat{\Theta}_t})-\phi(t, \Theta_t,\sP_{\Theta_t})$.
We also denote by $C$ a generic  positive constant, 
which
depends on $T$,
 $L$  in (H.\ref{assum:fbsde_discrete_exp}\ref{item:lipschitz_exp}),
and  may take a different value at each occurrence.

By  summation of \eqref{eq:MV-fbsde_exp}
over the index $i$
and insertion in \eqref{eq:error_law_exp},
\begin{align}\l{eq:error_delta_bwd_imp}
\begin{split}
\cE_\pi(\hat{X},\hat{Y},\hat{Z}, \hat{M})
&=
\sE[|\delta {X}_0|^2]
+
\sE[|\delta {Y}_N-(g(\hat{X}_N,\sP_{\hat{X}_N})-g({X}_N,\sP_{{X}_N}))|^2]
\\
&\quad +
\max_{i\in \cN_{<N}}
\sE
\bigg[
\bigg|
\underbrace{
\delta {X}_{i+1}
-\delta {X}_{0}
-\sum_{j=0}^i
\big(
\delta b(t_{j})\tau_N  +
\delta \sigma (t_j)\, \Delta W_j
\big)
}_{\coloneqq A_{i}}
\bigg|^2
\bigg]\\
&\quad +
\max_{i\in \cN_{<N}}
\bigg[
\bigg|
\underbrace{
\delta {Y}_{i+1}-\delta {Y}_0
+\sum_{j=0}^{i}
\big(
\delta f(t_{j})\tau_N- \delta{Z}_j\,\Delta W_j-\Delta (\delta {M})_j
\big)
}_{\coloneqq B_{i}}
\bigg|^2
\bigg],
\end{split}
\end{align}
where   the last  term   used   $\delta M_0=0$.
The Lipschitz continuity of  $g$ and the Cauchy-Schwarz inequality imply  that
\begin{align*}
&\sE[|\delta {Y}_N-(g(\hat{X}_N,\sP_{\hat{X}_N})-g({X}_N,\sP_{{X}_N}))|^2]
\le 
2\sE[|\delta {Y}_N|^2]+2L^2\sE[(|\delta {X}_N|+\cW_2(\sP_{\hat{X}_N},\sP_{{X}_N}))^2]
\\
&\le 
C\big(\sE[|\delta {Y}_N|^2]+\sE[(|\delta {X}_N|^2+\cW^2_2(\sP_{\hat{X}_N},\sP_{{X}_N}))]
\big)
\le
C(\sE[|\delta {Y}_N|^2]+\sE[|\delta {X}_N|^2]),
\end{align*}
which together with \eqref{eq:error_delta_bwd_imp} leads  to the estimate that
\begin{align}\l{eq:efficient_error1_bwd_imp}
\begin{split}
\cE_\pi(\hat{X},\hat{Y},\hat{Z}, \hat{M})
&\le
C\bigg(
\max_{i\in \cN}\sE[|\delta {X}_i|^2]
+\max_{i\in \cN}\sE[|\delta {Y}_i|^2]
 +
\max_{i\in \cN_{<N}}
(\sE[ |A_{i}|^2] +\sE[ |B_{i}|^2])
\bigg),
\end{split}
\end{align}
where the quantities $(A_i,B_i)_{i\in \cN_{<N}}$ are defined as in \eqref{eq:error_delta_bwd_imp}.

We first estimate $A_i$ for  $i\in \cN_{<N}$.
The Cauchy-Schwarz inequality,
the adaptedness of  coefficients and
the fact that $W$ is a martingale with $\sE_j[\Delta W_j(\Delta W_j)^*]=\tau_N I_d$ 
for  $j\in \cN_{<N}$
yield 
\begin{align*}
\sE[|A_i|^2]
&\le
C\bigg(
\sE[|\delta {X}_{i+1}|^2]
+\sE[|\delta {X}_{0}|^2]
+\sE\bigg[ \bigg|
\sum_{j=0}^i
\delta b(t_{j})\tau_N
\bigg|^2\bigg]  
+
\sE\bigg[ \bigg|\sum_{j=0}^i
\delta \sigma (t_j)\, \Delta W_j
\bigg|^2\bigg]
\bigg)
\\
&\le 
C\bigg(
\max_{i\in\cN}\sE[|\delta {X}_{i}|^2]
+\sE\bigg[ 
\sum_{j=0}^i
|\delta b(t_{j})|^2\tau_N\bigg] 
T
+\sum_{j=0}^i\sE\left[ \left|
\delta \sigma (t_j)\, \Delta W_j
\right|^2\right]
\bigg)
\\
&\le 
C\bigg(
\max_{i\in\cN}\sE[|\delta {X}_{i}|^2]
+T\sum_{j=0}^{N-1}\sE[ 
|\delta b(t_{j})|^2\tau_N] 
+\sum_{j=0}^{N-1}\sE\left[ 
|\delta \sigma (t_j)|^2\, \tau_N
\right]
\bigg).
\end{align*}
Note that the definitions of $\delta b$, $\delta\sigma$ and 
the Lipschitz continuity of $b$, $\sigma$
in (H.\ref{assum:fbsde_discrete_exp}\ref{item:lipschitz_exp}) show that
 for all $j\in \cN_{<N}$
and $\phi=b,\sigma$,
\begin{align*}
\sE[ |\delta \phi(t_j)|^2] 
&\le C\sE[ (|\delta \Theta_j|+\cW_2(\sP_{\Theta_j},\sP_{\hat{\Theta}_j}))^2] 
\le  C(\sE[ |\delta X_j|^2]+\sE[ |\delta Y_j|^2]+\sE[ |\delta Z_j|^2]).
\end{align*}
Hence,   for all $i\in \cN_{<N}$,
\begin{align}\l{eq:efficient_error2_bwd_imp}
\begin{split}
\sE[|A_i|^2]
&\le 
C\bigg(
\max_{i\in\cN}\sE[|\delta {X}_{i}|^2]
+\sum_{j=0}^{N-1}
\big( 
\sE[ |\delta X_j|^2]+\sE[ |\delta Y_j|^2]+\sE[ |\delta Z_j|^2]
\big) 
\tau_N
\bigg)
\\
&\le
C\bigg(
\max_{i\in\cN}\sE[|\delta {X}_{i}|^2+ |\delta Y_i|^2]
+
\sum_{j=0}^{N-1}\sE[ |\delta Z_j|^2]\tau_N
\bigg).
\end{split}
\end{align}

We proceed to derive an upper bound of $B_i$ for all $i\in \cN_{<N}$.
The Cauchy-Schwarz inequality
and  
the fact that 
the martingale $\delta M$ is strongly orthogonal to  $W$
imply that
\begin{align*}
\begin{split}
\sE[|B_i|^2]
&\le
4\bigg(
\sE[|\delta Y_{i+1}|^2]
+\sE[|\delta Y_{0}|^2]
+\sE\bigg[ \bigg|
\sum_{j=0}^i
\delta f(t_{j})\tau_N
\bigg|^2\bigg]  
+\sE\bigg[ \bigg|
\sum_{j=0}^i
\delta{Z}_j\,\Delta W_j+\Delta (\delta {M})_j
\bigg|^2\bigg]  
\bigg)
\\
&\le
4\bigg(
2\max_{i\in \cN}\sE[|\delta Y_{i}|^2]
+\sE\bigg[ \bigg|
\sum_{j=0}^i
\delta f(t_{j})\tau_N
\bigg|^2\bigg]  
+
\sE\bigg[ \bigg|
\sum_{j=0}^i
\delta{Z}_j\,\Delta W_j\bigg|^2\bigg]
+\sE\bigg[ \bigg|
\sum_{j=0}^i\Delta (\delta {M})_j
\bigg|^2\bigg]  
\bigg)
\\
&\le 
4\bigg(
2\max_{i\in \cN}\sE[|\delta Y_{i}|^2]
+\sE\bigg[ \bigg|
\sum_{j=0}^i
\delta f(t_{j})\tau_N
\bigg|^2\bigg]  
 +\sE\bigg[ 
\sum_{j=0}^i
|\delta{Z}_j|^2\,\tau_N\bigg]
+\sE\bigg[ 
\sum_{j=0}^i|\Delta (\delta {M})_j|^2
\bigg]
\bigg)
\\
&\le
4\bigg(
2\max_{i\in \cN}\sE[|\delta Y_{i}|^2]
+\sE\bigg[ \bigg|
\sum_{j=0}^i
\delta f(t_{j})\tau_N
\bigg|^2\bigg]  
 +\sE\bigg[ 
\sum_{j=0}^{N-1}
|\delta{Z}_j|^2\,\tau_N\bigg]
+\sE[| \delta {M}_N|^2]
\bigg),
\end{split}
\end{align*}
where 
 the last inequality
 used    $\delta M_0=0$.
Moreover, 
by  using the Cauchy-Schwarz inequality
and the Lipschitz continuity of $f$,
 for all $i\in \cN_{<N}$,
\begin{align*}
\begin{split}
&\sE\bigg[ \bigg|
\sum_{j=0}^i
\delta f(t_{j})\tau_N
\bigg|^2\bigg]  
\le
\sE\bigg[ \bigg(
\sum_{j=0}^{N-1}
|\delta f(t_{j})|\tau_N
\bigg)^2\bigg]  
\le 
T\sE\bigg[ 
\sum_{j=0}^{N-1}
|\delta f(t_{j})|^2\tau_N
\bigg]   
\\ 
&\le 
C\bigg( 
\sum_{j=0}^{N-1}
\sE\bigg[\big(
|\delta \Theta_{j}|
+\cW_2(\sP_{\Theta_{j}},\sP_{\hat{\Theta}_{j}})
\big)^2
\bigg]
\tau_N\bigg) 
\le 
C\bigg( 
\sum_{j=0}^{N-1}
\sE[
|\delta X_{j}|^2+|\delta Y_{j}|^2+
|\delta Z_{j}|^2]
\tau_N
\bigg)   
\\
&\le 
C
\bigg( 
\max_{i\in\cN}
\sE[
|\delta X_i|^2+|\delta Y_i|^2]
 +
\sum_{j=0}^{N-1}
\sE[
|\delta Z_{j}|^2]
\tau_N\bigg).  
\end{split}
\end{align*}
Hence, for all $i\in \cN_{<N}$,
\begin{align}\l{eq:efficient_error3_bwd_imp}
\begin{split}
\sE[|B_i|^2]
&\le C\bigg(
\max_{i\in \cN}\sE[|\delta X_{i}|^2+|\delta Y_{i}|^2]
 +
\sum_{j=0}^{N-1}
\sE[ |\delta{Z}_j|^2]\,\tau_N
 +\sE[| \delta {M}_N|^2]
\bigg).
\end{split}
\end{align}
The desired estimate then follows from 
\eqref{eq:efficient_error1_bwd_imp}, \eqref{eq:efficient_error2_bwd_imp} and \eqref{eq:efficient_error3_bwd_imp}.
%
%
%
\end{proof}

We then proceed to  establish the reliability of the {a posteriori} error estimator \eqref{eq:error_law_exp}
by first
introducing the following auxiliary processes.
Suppose that (H.\ref{assum:fbsde_discrete_exp}) holds,
and  $(\hat{X},\hat{Y},\hat{Z}, \hat{M})\in\cS_N$
is a given approximation on a time grid $\pi_N$.
We  introduce the   processes $(\bar{X},\bar{Y},\bar{Z}, \bar{M})\in\cS_N$
such that $\bar{Z}\equiv \hat{Z}$,  $\bar{M}\equiv \hat{M}$, 
$\bar{X}_0=\hat{X}_0$, $\bar{Y}_0=\hat{Y}_0$
and  for all $i\in \cN_{<N}$,
\begin{align}\l{eq:Theta_bar_exp}
\begin{split}
\Delta\bar{X}_{i}
&\coloneqq
b(t_{i},\hat{\Theta}_{i},\sP_{\hat{\Theta}_{i}})\tau_N  +
\sigma (t_i,\hat{\Theta}_i,\sP_{\hat{\Theta}_i})\, \Delta W_i,
\\
\Delta \bar{Y}_{i}
&\coloneqq
-f(t_{i},\hat{\Theta}_{i}, \sP_{\hat{\Theta}_{i}})\tau_N+ \hat{Z}_i\,\Delta W_i+\Delta \hat{M}_i
\end{split}
\end{align}
with $\hat{\Theta}= (\hat{X},\hat{Y},\hat{Z})$.
Then it is clear that the error estimator \eqref{eq:error_law_exp} can be equivalently written as
\begin{align}\l{eq:error_law_Theta_bar_bwd_imp}
\begin{split}
\cE_\pi(\hat{X},\hat{Y},\hat{Z}, \hat{M})
&=
\sE[|\hat{X}_0-\xi_0|^2]
+
\sE[|\hat{Y}_N-g(\hat{X}_N,\sP_{\hat{X}_N})|^2]
\\
&\quad +
\max_{i\in \cN_{<N}}
\sE[|\hat{X}_{i+1}
-\bar{X}_{i+1}|^2]
+
\max_{i\in \cN_{<N}}
[|\hat{Y}_{i+1}-\bar{Y}_{i+1}|^2].
\end{split}
\end{align}

With the above processes $(\bar{X},\bar{Y},\bar{Z}, \bar{M})\in\cS_N$ at hand,
we now show the  error estimator \eqref{eq:error_law_exp} is  reliable for all sufficiently fine time grids $\pi_N$.

\begin{Theorem}\l{thm:error_reliable_discrete_exp}
Suppose (H.\ref{assum:fbsde_discrete_exp}) 
holds.
Then there exists a constant $C>0$,
such that
for  
all 
sufficiently large $N$
and
for every 
4-tuple of processes $(\hat{X},\hat{Y},\hat{Z}, \hat{M})\in \cS_N$,
\begin{align*}
\begin{split}
&\max_{i\in \cN}
\left(
\sE[|\hat{X}_{i}
-{X}^\pi_{i}|^2]
+
\sE[|\hat{Y}_{i}
-{Y}^\pi_{i}|^2]
\right)
+
\sum_{i=0}^{N-1}
\sE
[
|\hat{Z}_{i}
-{Z}^\pi_{i}|^2]\tau_N
+
\sE[|\hat{M}_{N}
-{M}^\pi_{N}|^2]
\\
&\le
C\cE_\pi(\hat{X},\hat{Y},\hat{Z}, \hat{M}),
\end{split}
\end{align*}
where $(X^\pi,Y^\pi, Z^\pi, M^\pi)\in \cS_N$
is the solution to MV-\fbsde \eqref{eq:MV-fbsde_exp}
defined on $\pi_N$. 

\end{Theorem}

{
\begin{Remark}
\label{rmk:constant_dependence}
The constant $C$ in Theorem \ref{thm:error_reliable_discrete_exp} depends    on  the constants
$T,  \alpha, \b_1,\b_2, L$    in (H.\ref{assum:fbsde_discrete_exp}), 
the spectral norm of $G$ in (H.\ref{assum:fbsde_discrete_exp}\ref{item:monotone_exp}), 
the spectral norm of $(G^*G)^{-1}G^*$ if $m\ge n$,
and 
the spectral norm of $(GG^*)^{-1}G$ if $n\ge m$.
This can be seen by examining 
the proofs  of  Proposition \ref{thm:stab_exp} and Theorem \ref{thm:error_reliable_discrete_exp} carefully.
In particular, the constant $C$
   does not depend explicitly on the dimensions $m, n ,d $. 
   Similar remarks also apply to the constant $C$ in the statements of Proposition  \ref{thm:time_error}, 
   Theorem \ref{thm:a_posterior_conts} and      Corollary \ref{cor:deep_bsde}.
 \end{Remark}
}

\begin{proof}
Throughout this proof, 
let $N_0\in \sN$ be  the  natural number  in Proposition \ref{thm:stab_exp},
 $N\in \sN\cap [N_0,\infty)$
 and
 $(\hat{X},\hat{Y},\hat{Z}, \hat{M})\in \cS_N$
 be fixed.
 Let $(X^\pi,Y^\pi, Z^\pi, M^\pi)$ be  a solution to  \eqref{eq:MV-fbsde_exp}
 on $\pi_N$,
and  $C$ be a generic  positive constant, 
which
depends only on the constants   in (H.\ref{assum:fbsde_discrete_exp})
and  may take a different value at each occurrence.

Let $(\bar{X},\bar{Y},\bar{Z}, \bar{M})\in \cS_N$
be the auxiliary processes defined as in \eqref{eq:Theta_bar_exp}
and $\bar{\Theta}= (\bar{X},\bar{Y}, \bar{Z})$.
 We  first derive
  an $L^2$-estimate 
of
the difference between
 $(\bar{X},\bar{Y},\bar{Z}, \bar{M})$ and 
the solution $(X^\pi,Y^\pi, Z^\pi, M^\pi)$ to \eqref{eq:MV-fbsde_exp}.
Observe that  $(\bar{X},\bar{Y},\bar{Z}, \bar{M})\in \cS_N$
is a solution to \eqref{eq:MV-fbsde_exp_lambda} with 
$\lambda=1$, generator $({b},{\sigma},{f},{g})=0$,
${\xi}_0=\bar{X}_0$, ${\eta}=\bar{Y}_N$
and 
$({\phi},{\psi},{\gamma})\in \cM^2(0,T; \sR^n\t \sR^{n\t d}\t \sR^m)$
satisfying for all $i\in \cN_{<N}$ that
${\phi}_{i}=b(t_{i},\hat{\Theta}_{i},\sP_{\hat{\Theta}_{i}})$,
${\psi}_i=\sigma (t_i,\hat{\Theta}_i,\sP_{\hat{\Theta}_i})$
and ${\gamma}_i=f(t_{i},\hat{\Theta}_{i}, \sP_{\hat{\Theta}_{i}})$.
Hence by
 Proposition \ref{thm:stab_exp} 
(with $\lambda_0=1$
and $({\phi},{\psi},{\gamma},\eta)=0$), 
there exists a constant $C>0$ such  that
\begin{align*}
\begin{split}
&\max_{i\in \cN}
\left(
\sE[|{X}^\pi_{i}
-\bar{X}_{i}|^2]
+
\sE[|{Y}^\pi_{i}
-\bar{Y}_{i}|^2]
\right)
+
\sum_{i=0}^{N-1}
\sE
[
|{Z}^\pi_{i}
-\bar{Z}_{i}|^2]\tau_N
+
\sE[|{M}^\pi_{N}
-\bar{M}_{N}|^2]
\\
&\le
C\bigg\{
\sE[| \xi_{0}-\bar{X}_0|^2]
+
\sE[ |g(\bar{X}_N,\sP_{\bar{X}_N})
-\bar{Y}_N|^2]
+\sum_{i=0}^{N-1}
\bigg(
\sE[|{f}(t_{i},\bar{\Theta}_i,\sP_{\bar{\Theta}_i})-f(t_{i},\hat{\Theta}_{i}, \sP_{\hat{\Theta}_{i}})|^2]\tau_N
\\
&\quad
+
\sE[|{b}(t_{i},\bar{\Theta}_{i},\sP_{\bar{\Theta}_{i}})- b(t_{i},\hat{\Theta}_{i},\sP_{\hat{\Theta}_{i}})|^2]\tau_N
+
\sE|{\sigma}(t_i,\bar{\Theta}_i,\sP_{\bar{\Theta}_i})-\sigma (t_i,\hat{\Theta}_i,\sP_{\hat{\Theta}_i})|^2]\tau_N
\bigg)
\bigg\},
\end{split}
\end{align*}
which together with the Lipschitz continuity of the generator and the fact that $\bar{Z}\equiv \hat{Z}$, $\bar{X}_0=\hat{X}_0$ and  $\bar{Y}_0=\hat{Y}_0$,
yields   that
\begin{align*}
\begin{split}
&\max_{i\in \cN}
\left(
\sE[|{X}^\pi_{i}
-\bar{X}_{i}|^2]
+
\sE[|{Y}^\pi_{i}
-\bar{Y}_{i}|^2]
\right)
+
\sum_{i=0}^{N-1}
\sE
[
|{Z}^\pi_{i}
-\bar{Z}_{i}|^2]\tau_N
+
\sE[|{M}^\pi_{N}
-\bar{M}_{N}|^2]
\\
&\le
C\bigg(
\sE[| \xi_{0}-\hat{X}_0|^2]
+
\sE[ |g(\bar{X}_N,\sP_{\bar{X}_N})
-\bar{Y}_N|^2]
+\sup_{i\in \cN_{<N}}
\big(
\sE[|\bar{X}_i-\hat{X}_i|^2]
+\sE[|\bar{Y}_i-\hat{Y}_i|^2]
\big)
\bigg)
\\
&\le
C\bigg(
\sE[| \xi_{0}-\hat{X}_0|^2]
+
\sE[ |g(\bar{X}_N,\sP_{\bar{X}_N})-g(\hat{X}_N,\sP_{\hat{X}_N})|^2
+|g(\hat{X}_N,\sP_{\hat{X}_N})-\hat{Y}_N|^2+|\hat{Y}_N-\bar{Y}_N|^2]
\\
&\quad 
+\sup_{i\in \cN_{<N}}
\big(
\sE[|\bar{X}_{i+1}-\hat{X}_{i+1}|^2]
+\sE[|\bar{Y}_{i+1}-\hat{Y}_{i+1}|^2]
\big)
\bigg)
\\
&\le
C\cE_\pi(\hat{X},\hat{Y},\hat{Z}, \hat{M}),
\end{split}
\end{align*}
where  the last line  used 
the equivalent definition \eqref{eq:error_law_Theta_bar_bwd_imp} of 
the  estimator \eqref{eq:error_law_exp}.
Consequently, by using the triangle inequality 
and the fact that 
$\bar{X}_0=\hat{X}_0$,  $\bar{Y}_0=\hat{Y}_0$,
$\bar{M}\equiv \hat{M}$
and  $\bar{Z}\equiv \hat{Z}$,
\begin{align*}
\begin{split}
&\max_{i\in \cN}
\left(
\sE[|\hat{X}_{i}
-{X}^\pi_{i}|^2]
+
\sE[|\hat{Y}_{i}
-{Y}^\pi_{i}|^2]
\right)
+
\sum_{i=0}^{N-1}
\sE
[
|\hat{Z}_{i}
-{Z}^\pi_{i}|^2]\tau_N
+
\sE[|\hat{M}_{N}
-{M}^\pi_{N}|^2]
\\
&\le
2\max_{i\in \cN}
\left(
{
\mathbb{E}[|\hat{X}_{i}
-\bar{X}_{i}|^2] }
+\sE[|\bar{X}_{i}
-{X}^\pi_{i}|^2]
+
\sE[
|\hat{Y}_{i}
-\bar{Y}_{i}|^2
+|\bar{Y}_{i}
-{Y}^\pi_{i}|^2]
\right)
\\
&\quad
+
\sum_{i=0}^{N-1}
\sE
[
|\bar{Z}_{i}
-{Z}^\pi_{i}|^2]\tau_N
+
\sE[|\bar{M}_{N}
-{M}^\pi_{N}|^2]
\le C\cE_\pi(\hat{X},\hat{Y},\hat{Z}, \hat{M}).
\end{split}
\end{align*}
This proves the desired estimate.
\end{proof}

\section{A posteriori  estimates for continuous MV-FBSDEs}\l{sec:a_posteriori_conts}

Based on  Theorems
\ref{thm:wp_bwd_imp} and 
 \ref{thm:error_reliable_discrete_exp},
we  
  prove that 
 the  approximation error between a given numerical approximation and the solutions to  
\eqref{eq:fbsde_conts_intro}
  can also be measured by the {a posteriori} error estimator \eqref{eq:error_law_exp}
  together with
  a measure of 
    the time regularity of the exact solution, which vanishes as the  stepsize $\tau_N$ tends to zero.
{We shall also provide a theoretical justification for the convergence of a commonly used machine learning-based algorithm for solving MV-FBSDEs
based on  the {a posteriori} error estimates.}

In the sequel, 
we  assume that
 $W=(W_t)_{t\in [0,T]}$  is
 a $d$-dimensional Brownian motion,
 $\sF= \{ \cF_t \}_{t\in [0,T]}$ is the augmented filtration generated by $W$ and an independent initial $\sigma$-algebra $\cF_0$,
and assume the generator
$(b, \sigma, f, g)$ of 
  the   MV-FBSDE
\eqref{eq:fbsde_conts_intro}
satisfies 
 (H.\ref{assum:fbsde_discrete_exp}).
Since
 every $\sF$ local martingale can be represented as a stochastic integral with respect to $W$
(see \cite[Theorem  4.33 on p.~176]{jacod1987}),
 extending 
Theorem 2 in \cite{bensoussan2015} to the present case with random initial condition $\xi_0$
shows that 
 \eqref{eq:fbsde_conts_intro}
 admits  a unique triple $(X,Y,Z)\in \cM^2(0,T;\sR^n\t \sR^m\t \sR^{m\t d })$.
%
%
To analyze the time discretization error, 
we further assume the following time regularity of the coefficients of \eqref{eq:fbsde_conts_intro}:

\begin{Assumption}\l{assum:time_reg}
There exists an increasing function $\ol{\om}: [0,\infty] \to [0, \infty]$,
vanishing at $0$ and continuous at $0$,
such that 
it holds 
  for $\sP$-a.s.~$\omega\in \Om$,  
  all $t,s\in [0,T]$,
 $(x,y,z)\in  \sR^n\t \sR^m\t \sR^{m\t d}$,
 $\mu\in \cP_2(\sR^{n+m+md})$,
 $\phi=b,\sigma,f$ that
 $ |\phi(t,x,y,z,\mu)-\phi(s,x,y,z,\mu)|\le \ol{\om}(|t-s|)$.
\end{Assumption}

%

To quantify the performance of \eqref{eq:error_estimator_intro},
for any numerical solution  $(\hat{X},\hat{Y},\hat{Z})\in\cS^0_N$  to \eqref{eq:fbsde_conts_intro},
 we consider 
 the squared approximation error
 of $(\hat{X},\hat{Y},\hat{Z})$
 on the interval $[0,T]$ defined by
\begin{align}\l{eq:error_T}
\begin{split}
&\err(\hat{X},\hat{Y},\hat{Z})
\\
&\coloneqq \max_{i\in \cN_{<N}}
\max_{t\in [t_i,t_{i+1}]}
\left(
\sE[|
X_t
-\hat{X}_{i}|^2]
+
\sE[|
{Y}_t
-\hat{Y}_{i}|^2]
\right)
+
\sum_{i=0}^{N-1}
\sE\bigg[
\int_{t_i}^{t_{i+1}}
|{Z}_{t}
-\hat{Z}_{i}|^2
]
\,dt\bigg],
\end{split}
\end{align}
and  
 the squared approximation error
of $(\hat{X},\hat{Y},\hat{Z})$
on the grid $\pi_N$  defined as follows (see \cite{zhang2004,lionnet2015}):
\begin{align}\l{eq:error_T_N}
\begin{split}
\err_\pi(\hat{X},\hat{Y},\hat{Z})
&\coloneqq \max_{i\in \cN}
\left(
\sE[|
X_i
-\hat{X}_{i}|^2]
+
\sE[|
{Y}_i
-\hat{Y}_{i}|^2]
\right)
+
\sum_{i=0}^{N-1}
\sE[
|\bar{Z}_{i}
-\hat{Z}_{i}|^2
]
]\tau_N,
\end{split}
\end{align}
 where 
$\bar{Z}_i\coloneqq \frac{1}{\tau_N}\sE_i\big[\int_{t_i}^{t_{i+1}} Z_s\, ds\big]$
for all $i\in \cN_{<N}$.
%
%
In the following,
we shall demonstrate that 
both $\err_\pi(\hat{X},\hat{Y},\hat{Z})$ and
$\err(\hat{X},\hat{Y},\hat{Z})$ can be effectively estimated  by the modulus of continuity $\ol{\om}$ in (H.\ref{assum:time_reg}), 
the {a posteriori} error estimator 
$\cE_\pi(\hat{X},\hat{Y},\hat{Z})$ defined as in \eqref{eq:error_estimator_intro}
and 
a measure of 
the time regularity of the solution $(X,Y,Z)$ defined 
 as follows: for any given grid $\pi_N$,
\begin{align}\l{eq:L2_Regularity_proj}
\begin{split}
&\cR_\pi({X},{Y},{Z})
\\
&\coloneqq 
\max_{i\in \cN_{<N}}
\max_{t\in [t_i,t_{i+1}]}
\left(
\sE[|X_t
-{X}_{i}|^2
]
+
\sE[|
{Y}_t
-{Y}_{i}|^2]
\right)
+
\sum_{i=0}^{N-1}
\sE\bigg[
\int_{t_i}^{t_{i+1}}
|{Z}_{t}
-\bar{Z}_{i}|^2
\,dt
\bigg].
\end{split}
\end{align}

\begin{Remark}\l{rmk:R}
The term  $\cR_\pi(X,Y,Z)$ 
is often referred to as the path regularity of 
$(X,Y,Z)$, 
and 
is essential for error estimates of numerical schemes for BSDEs 
(see 
\cite{zhang2004,lionnet2015}).
The fact that $(X,Y, Z)\in \cM^2(0,T;\sR^n\t \sR^m\t \sR^{m\t d })$
and the dominated convergence theorem show that 
 $\cR_\pi(X,Y,Z)$  tends to zero as the stepsize $\tau_N$ vanishes.
A rate of convergence of $\cR_\pi(X,Y,Z)$ can be obtained under further   structural assumptions.
In the case where $b$ and $\sigma$ are independent of $Z$ and $\sP_Z$, 
 \cite{reisinger2021} proves
 under (H.\ref{assum:fbsde_discrete_exp})-(H.\ref{assum:time_reg}) that 
 $\cR_\pi({X},{Y},{Z})=\cO(\tau_N)$ via    Malliavin calculus.
Alternatively, 
 suppose that 
there exists   $\cU : [0,T]\t \sR^n\t \cP_2(\sR^n)\to \sR^m$ 
and $\cV : [0,T]\t \sR^n\t \cP_2(\sR^n)\to \sR^{m\t d}$ 
satisfying the following properties
(see \cite{chassagneux2019}): 
\begin{itemize}
\item $Y_t=\cU(t,X_t,\sP_{X_t})$
and $Z_t=\cV(t,X_t,\sP_{X_t})$ for all $t\in [0,T]$,
\item 
 $\cU$ and $\cV$ are $1/2$-H\"{o}lder continuous in the time variable,
and are Lipschitz continuous in the spatial and measure variables.
\end{itemize}
The functions $U$ and $V$ are known as the decoupling fields for $Y$ and $Z$, respectively,
and allow rewriting \eqref{eq:fbsde_conts_fwd_intro}   as a McKean--Vlasov SDE with Lipschitz coefficients.
Then
  standard regularity estimates of MV-SDEs and the regularity of $U$ and $V$
give  $\cR_\pi({X},{Y},{Z})=\cO(\tau_N)$.

\end{Remark}

Now we   perform the   {a posteriori} error analysis  for  \eqref{eq:fbsde_conts_intro}.
The next proposition quantifies the time discretization error between \eqref{eq:MV-fbsde_exp} and \eqref{eq:fbsde_conts_intro},
whose  proof is  given in Appendix \ref{appendix:proof_continuous_error}.

\begin{Proposition}\l{thm:time_error}
Suppose (H.\ref{assum:fbsde_discrete_exp})-(H.\ref{assum:time_reg}) hold.
Let
$(X,Y, Z)\in \cM^2(0,T;\sR^n\t \sR^m\t \sR^{m\t d })$ be the solution to MV-FBSDE \eqref{eq:fbsde_conts_intro},
and 
for each $N\in \sN$, $i\in \cN_{<N}$
let $\bar{Z}_i=\frac{1}{\tau_N}\sE_i\big[\int_{t_i}^{t_{i+1}} Z_s\, ds\big]$.
Then 
there exists a constant $C>0$,\footnotemark such that 
 for  all sufficiently large $N\in \sN$,
\begin{align*}
&\max_{i\in \cN}\left(
\sE[|
X_i
-{X}^\pi_{i}|^2]
+
\sE[|
{Y}_i
-{Y}^\pi_{i}|^2]
\right)
+
\sum_{i=0}^{N-1}
\sE[
|\bar{Z}_{i}
-{Z}^\pi_{i}|^2
]\tau_N +\sE[|M^\pi_N|^2]
\\
&\le 
C\big(\ol{\om}(\tau_N)^2
+
\cR_\pi(X,Y,Z)
\big),
\end{align*}
where $(X^\pi,Y^\pi, Z^\pi, M^\pi)\in \cS_N$
is the solution to  \eqref{eq:MV-fbsde_exp} defined on $\pi_N$
(cf.~Theorem \ref{thm:wp_bwd_imp}),
$\ol{\om}$ is the modulus of continuity in (H.\ref{assum:time_reg}),
and $\cR_\pi(X,Y,Z)$ is defined as  in \eqref{eq:L2_Regularity_proj}.

\footnotetext{
{
See Remark \ref{rmk:constant_dependence}
for the dependence of the constant $C$ in the statements
of Proposition  \ref{thm:time_error}, 
   Theorem \ref{thm:a_posterior_conts}
   and Corollary \ref{cor:deep_bsde}.

}
}
\end{Proposition}

Based on Proposition \ref{thm:time_error},
we prove the efficiency and reliability of \eqref{eq:error_estimator_intro}
for \eqref{eq:fbsde_conts_intro}.

\begin{Theorem}\l{thm:a_posterior_conts}
Suppose (H.\ref{assum:fbsde_discrete_exp})-(H.\ref{assum:time_reg}) hold.
Let
$(X,Y, Z)\in \cM^2(0,T;\sR^n\t \sR^m\t \sR^{m\t d })$ be the solution to MV-FBSDE \eqref{eq:fbsde_conts_intro}.
Then 
there exists a constant $C>0$, such that 
for  all sufficiently large $N\in \sN$
and 
for every 
triple  $(\hat{X},\hat{Y},\hat{Z})\in \cS^0_N$, 
\begin{align}
\err(\hat{X},\hat{Y},\hat{Z})
&\le 
C\big(\ol{\om}(\tau_N)^2
+
\cR_\pi(X,Y,Z)
+\cE_\pi(\hat{X},\hat{Y},\hat{Z})
\big),
\l{eq:efficient_T}
\\
\cE_\pi(\hat{X},\hat{Y},\hat{Z})
&\le 
C\big(\ol{\om}(\tau_N)^2
+
\cR_\pi(X,Y,Z)
+\err(\hat{X},\hat{Y},\hat{Z})
\big),
\l{eq:reliable_T}
\end{align}
where 
$\err(\hat{X},\hat{Y},\hat{Z})$ is 
defined as in \eqref{eq:error_T},
$\ol{\om}$ is the modulus of continuity in (H.\ref{assum:time_reg}),
 $\cR_\pi(X,Y,Z)$ is 
 defined as in \eqref{eq:L2_Regularity_proj},
 and $\cE_\pi(\hat{X},\hat{Y},\hat{Z})$ is 
 defined as in \eqref{eq:error_estimator_intro}.
Moreover, the same error estimates \eqref{eq:efficient_T} and \eqref{eq:reliable_T} also hold 
by replacing $\err(\hat{X},\hat{Y},\hat{Z})$
with $\err_\pi(\hat{X},\hat{Y},\hat{Z})$
 defined as in \eqref{eq:error_T_N}.
\end{Theorem}

\begin{proof}
Throughout this proof, 
let $\pi_N$ be an arbitrary  fixed partition of $[0,T]$ 
with a sufficiently large $N$,
let $(X^\pi,Y^\pi, Z^\pi, M^\pi
)\in \cS_N$
be the solution to  \eqref{eq:MV-fbsde_exp} defined on $\pi_N$,
let  $(\delta \Theta, \delta X,\delta Y, \delta Z)= ( \hat{ \Theta}-\Theta, \hat{X}-X, \hat{Y}-Y, \hat{Z}-Z)$,
and for each $t\in [0,T]$, $\phi=b, \sigma, f$
let $ \hat{\phi}(t)= \phi(t, \hat{\Theta}_t,\sP_{\hat{\Theta}_t})$,
$\phi(t)=\phi(t, \Theta_t,\sP_{\Theta_t})$.
We also denote by $C$ a generic constant, 
which
depends only on the constants appearing in (H.\ref{assum:fbsde_discrete_exp}),
and  may take a different value at each occurrence.

Observe from the triangle inequality  that
$\err(\hat{X},\hat{Y},\hat{Z})\le 2(\err_\pi(\hat{X},\hat{Y},\hat{Z})+\cR_\pi(X,Y,Z))$.
Hence
 it suffices to prove \eqref{eq:efficient_T} for $\err_\pi(\hat{X},\hat{Y},\hat{Z})$ 
and \eqref{eq:reliable_T} for $\err(\hat{X},\hat{Y},\hat{Z})$.
The estimate \eqref{eq:efficient_T} for $\err_\pi(\hat{X},\hat{Y},\hat{Z})$  essentially follows by combining 
 Theorem \ref{thm:error_reliable_discrete_exp}
and Proposition \ref{thm:time_error}.
In fact, 
for all  sufficiently large $N\in \sN$,
\begin{align*}
\err_\pi(\hat{X},\hat{Y},\hat{Z})
&\le 
2\bigg[
\max_{i\in \cN}\left(
\sE[|
X_i
-{X}^\pi_{i}|^2]
+
\sE[|
{Y}_i
-{Y}^\pi_{i}|^2]
\right)
+
\sum_{i=0}^{N-1}
\sE[
|\bar{Z}_{i}
-{Z}^\pi_{i}|^2
]
]\tau_N
\\
&\quad +
\max_{i\in \cN}
\left(
\sE[|\hat{X}_{i}
-{X}^\pi_{i}|^2]
+
\sE[|\hat{Y}_{i}
-{Y}^\pi_{i}|^2]
\right)
+
\sum_{i=0}^{N-1}
\sE
[
|\hat{Z}_{i}
-{Z}^\pi_{i}|^2]\tau_N
\bigg]
\\
&
\le
C\big(\ol{\om}(\tau_N)^2
+
\cR_\pi(X,Y,Z)
+\cE_\pi(\hat{X},\hat{Y},\hat{Z})
\big).
\end{align*}
This proves the estimate \eqref{eq:efficient_T}.

We then establish the estimate \eqref{eq:reliable_T} for $\err(\hat{X},\hat{Y},\hat{Z})$
by following a similar argument as that for Theorem \ref{thm:error_efficient_discrete_exp}.
By using \eqref{eq:fbsde_conts_intro},
\begin{align*}
\begin{split}
\cE_\pi(\hat{X},\hat{Y},\hat{Z})
&=
\sE[|\delta {X}_0|^2]
+
\sE[|\delta {Y}_N-(g(\hat{X}_N,\sP_{\hat{X}_N})-g({X}_N,\sP_{{X}_N}))|^2]
\\
&\quad +
\max_{i\in \cN_{<N}}
\sE
\bigg[
\bigg|
\delta {X}_{i+1}
-\delta {X}_{0}
-\sum_{j=0}^i
\int_{t_j}^{t_{j+1}}
\bigg(
 \big(
\hat{b}(t_{j})-b(t)
 \big)\,dt  +
 \big(
 \hat{\sigma}(t_{j})-\sigma(t)
  \big)\,dW_t
  \bigg)
\bigg|^2
\bigg]\\
&\quad +
\max_{i\in \cN_{<N}}
\bigg[
\bigg|
\delta {Y}_{i+1}-\delta {Y}_0
+\sum_{j=0}^{i}
\int_{t_j}^{t_{j+1}}
\bigg(
\big(
 \hat{f}(t_{j})
 -{f}(t)
 \big)\,dt- 
 (\hat{Z}_j-Z_t)\,dW_t
 \bigg)
\bigg|^2
\bigg],
\end{split}
\end{align*}
which together with the Lipschitz continuity of $g$ implies that
\begin{align}
\begin{split}
\cE_\pi(\hat{X},\hat{Y},\hat{Z})
&\le
C\bigg(
\max_{i\in \cN}\sE[|\delta {X}_i|^2]
+\max_{i\in \cN}\sE[|\delta {Y}_i|^2]
 +
\max_{i\in \cN_{<N}}
(\sE[ |A_{i}|^2] +\sE[ |B_{i}|^2])
\bigg),
\end{split}
\end{align}
with the quantities $(A_i,B_i)_{i\in \cN_{<N}}$  defined by
\begin{align*}
A_i
&\coloneqq 
\sum_{j=0}^i
\int_{t_j}^{t_{j+1}}
\bigg(
 \big(
\hat{b}(t_{j})-b(t)
 \big)\,dt  +
 \big(
 \hat{\sigma}(t_{j})-\sigma(t)
  \big)\,dW_t
  \bigg),
\\
B_i
&\coloneqq 
\sum_{j=0}^{i}
\int_{t_j}^{t_{j+1}}
\bigg(
\big(
 \hat{f}(t_{j})
 -{f}(t)
 \big)\,dt- 
 (\hat{Z}_j-Z_t)\,dW_t
 \bigg).
\end{align*}
Then, by applying the Cauchy-Schwarz inequality, the It\^{o} isometry
and the Lipschitz continuity of the coefficients, we have for all $i\in \cN_{<N}$ that
\begin{align*}
\sE[|A_i|^2]
&\le
C
\sum_{j=0}^i
\int_{t_j}^{t_{j+1}}
\bigg(
 \sE[
|\hat{b}(t_{j})-b(t)|^2]+\sE[|\hat{\sigma}(t_{j})-\sigma(t)|^2]
  \bigg)\,dt 
  \\
 &\le
C
\sum_{j=0}^{N-1}
\int_{t_j}^{t_{j+1}}
\bigg(
\ol{\om}(\tau_N)^2+
\sE[|\Theta_t-\hat{\Theta}_{j}|^2]
  \bigg)\,dt 
 \\
 &\le  
 C\big(\ol{\om}(\tau_N)^2
+
\cR_\pi(X,Y,Z)
+\err(\hat{X},\hat{Y},\hat{Z})
\big).
\end{align*}
Similarly, for all $i\in \cN_{<N}$,
\begin{align*}
\sE[|B_i|^2]
&\le
C
\sum_{j=0}^i
\int_{t_j}^{t_{j+1}}
\bigg(
 \sE[
|\hat{f}(t_{j})-f(t)|^2]+\sE[|\hat{Z}_j-Z_t|^2]
  \bigg)\,dt 
  \\
 &\le
C
\sum_{j=0}^{N-1}
\int_{t_j}^{t_{j+1}}
\bigg(
\ol{\om}(\tau_N)^2+
 \sE[
|\Theta_t-\hat{\Theta}_{j}|^2]
  \bigg)\,dt 
 \\
 & \le
 C\big(\ol{\om}(\tau_N)^2
+
\cR_\pi(X,Y,Z)
+\err(\hat{X},\hat{Y},\hat{Z})
\big).
\end{align*}
Summarizing all the above estimates 
gives the desired upper bound  \eqref{eq:reliable_T}. 
\end{proof}

\begin{Remark}\l{eq:err_vanishing_R}
As already mentioned above, both
$\ol{\om}(\tau_N)$
and
$\cR_\pi(X,Y,Z)$
will vanish as the stepsize $\tau_N$ tends to zero,
and admit a first-order convergence rate under suitable structural conditions.
Hence 
the estimates \eqref{eq:efficient_T} and 
\eqref{eq:reliable_T} suggest that 
the  error estimator 
  \eqref{eq:error_estimator_intro}
effectively measures 
the   accuracy of given numerical solutions to \eqref{eq:fbsde_conts_intro},
including the performance of the chosen 
numerical procedure for approximating 
the conditional expectations, 
for all sufficiently small  stepsizes.
\end{Remark}

We end this section 
by applying Theorem \ref{thm:a_posterior_conts}
to study
the Deep BSDE Solver
proposed in \cite{carmona2019,  fouque2019,germain2019}
for solving  coupled MV-FBSDEs,
extended from the original algorithm for BSDEs in \cite{e2017}.
Roughly speaking,
for a given time grid $\pi_N$  of $[0,T]$ with stepsize $\tau_N=T/N$
and any given 
 measurable functions $y_0:\sR^n\to \sR^m$, $z_i:\sR^n\to \sR^{m\t d}$, $i\in \cN_{<N}$,
the Deep BSDE Solver 
generates the discrete approximation
$(\hat{X}_i,\hat{Y}_i,\hat{Z}_i)_{i\in \cN_{<N}}$ 
 by
following an explicit forward Euler scheme:
\begin{align}\label{eq:deep_fbsde_exp}
\begin{split}
\hat{X}_0&\coloneqq \xi_0, \q \hat{Y}_0\coloneqq y_0(\hat{X}_0), \q \hat{Z}_i\coloneqq z_i(\hat{X}_i) \q  \fa i\in \cN_{<N},
\\
\Delta \hat{X}_i&\coloneqq b(t_{i},\hat{X}_{i},\hat{Y}_{i}, \hat{Z}_{i},\sP_{(\hat{X}_{i},\hat{Y}_{i}, \hat{Z}_{i})})\tau_N  
+\sigma (t_{i},\hat{X}_{i},\hat{Y}_{i}, \hat{Z}_{i},\sP_{(\hat{X}_{i},\hat{Y}_{i}, \hat{Z}_{i})})\, \Delta W_i
\q \fa  i\in \cN_{<N},
\\
\Delta \hat{Y}_i&\coloneqq -f(t_{i},\hat{X}_{i},\hat{Y}_{i}, \hat{Z}_{i}, \sP_{(\hat{X}_{i},\hat{Y}_{i}, \hat{Z}_{i})})\tau_N+\hat{Z}_i\,\Delta W_i
\q \fa i\in \cN_{<N}.
\end{split}
\end{align}
The algorithm then seeks the optimal $(\hat{y}_0,\{\hat{z}_i\}_{i})$ by minimizing the following terminal loss:
$$
(\hat{y}_0,\{\hat{z}_i\}_{i})
\in \argmin_{(y_0,\{z_i\}_{i})\in \cC} \sE[|\hat{Y}_N-g(\hat{X}_N,\sP_{\hat{X}_N}))|^2]
\q \textnormal{with $\cC=\cY \t \bigtimes_{i=0}^{N-1}\cZ_i$},
$$
where
$\cY$ 
 is a  parametric family of measurable functions 
from $\sR^n$ to $\sR^m$
and
 $(\cZ_i)_{i\in \cN_{<N}}$
are parametric families of measurable functions 
from $\sR^n$ to $\sR^{m\t d}$.
Note that  for simplicity
we consider the exact law $\sP_{(\hat{X}_{i},\hat{Y}_{i}, \hat{Z}_{i})}$ in  \eqref{eq:deep_fbsde_exp}, 
which 
in practice will be estimated by
particle approximations (see, e.g.,~\cite{germain2019}).
In the subsequent analysis, 
we shall denote by 
$(\hat{X}^{y,z},\hat{Y}^{y,z},\hat{Z}^{y,z})$
the numerical solution generated by \eqref{eq:deep_fbsde_exp}
to emphasize the dependence on 
$(y_0,\{z_i\}_{i})\in \cC$.

The following corollary 
shows that the approximation accuracy  of the  Deep BSDE Solver
can be measured by the terminal loss,
which extends Theorem 1  in \cite{han2018} to fully coupled MV-FBSDEs.
\begin{Corollary}\l{cor:deep_bsde}
Suppose (H.\ref{assum:fbsde_discrete_exp})-(H.\ref{assum:time_reg}) hold,
and the functions in $\cC$ are of linear growth.
Let
$(X,Y, Z)\in \cM^2(0,T;\sR^n\t \sR^m\t \sR^{m\t d })$ be the solution to  \eqref{eq:fbsde_conts_intro}.
Then 
there exists a constant $C>0$ such that 
it holds for  all sufficiently large $N\in \sN$
and 
for all 
$(y_0,\{z_i\}_{i})\in \cC$
that
\begin{align}\l{eq:efficient_T_deep}
\err(\hat{X}^{y,z},\hat{Y}^{y,z},\hat{Z}^{y,z})
\le 
C\big(\ol{\om}(\tau_N)^2
+
\cR_\pi(X,Y,Z)
+\sE[|\hat{Y}^{y,z}_N-g(\hat{X}^{y,z}_N,\sP_{\hat{X}^{y,z}_N})|^2]
\big),
\end{align}
where 
$\err(\hat{X}^{y,z},\hat{Y}^{y,z},\hat{Z}^{y,z})$ is 
defined as in \eqref{eq:error_T},
$\ol{\om}$ is the modulus of continuity in (H.\ref{assum:time_reg})
and $\cR_\pi(X,Y,Z)$ is 
 defined as in \eqref{eq:L2_Regularity_proj}.
\end{Corollary}
\begin{proof}
Note that 
 the linear growth of $(y_0,\{z_i\}_i)\in \cC$,
 (H.\ref{assum:fbsde_discrete_exp}\ref{item:lipschitz_exp}),
 (H.\ref{assum:fbsde_discrete_exp}\ref{item:integrable_exp}), 
and \eqref{eq:deep_fbsde_exp} 
imply that 
$(\hat{X}^{y,z},\hat{Y}^{y,z},\hat{Z}^{y,z})\in \cS_N^0$
and $\sE[|\hat{Y}^{y,z}_N-g(\hat{X}^{y,z}_N,\sP_{\hat{X}^{y,z}_N})|^2]=\cE_\pi(\hat{X}^{y,z},\hat{Y}^{y,z},\hat{Z}^{y,z})$,
which enable us to conclude \eqref{eq:efficient_T_deep} from
 \eqref{eq:efficient_T} in Theorem \ref{thm:a_posterior_conts}.
\end{proof}

\begin{Remark}
One can further control the terminal loss by using the approximation accuracy of 
$(y_0,\{z_i\}_i)\in \cC$,
which is important for the convergence analysis of the Deep BSDE Solver.
In fact, 
for any given $(y_0,\{z_i\}_i)\in \cC$, by viewing  \eqref{eq:MV-fbsde_exp} and   \eqref{eq:deep_fbsde_exp} as  explicit forward Euler schemes, 
we can deduce from  (H.\ref{assum:fbsde_discrete_exp}\ref{item:lipschitz_exp}), 
 Gronwall's inequality  {
 (see  Lemma \ref{lemma:gronwall})}
 and $\hat{X}^{y,z}_{0}
={X}^\pi_{0}=\xi_0$ 
that
\begin{align*}
\begin{split}
&
\max_{i\in \cN}
\left(
\sE[|\hat{X}^{y,z}_{i}
-{X}^\pi_{i}|^2]
+
\sE[|\hat{Y}^{y,z}_{i}
-{Y}^\pi_{i}|^2]
\right)
\le 
C
\bigg\{
\sE[|\hat{Y}^{y,z}_{0}
-{Y}^\pi_{0}|^2]
+
\sum_{i=0}^{N-1}
\sE
[
|\hat{Z}^{y,z}_{i}
-{Z}^\pi_{i}|^2]\tau_N +\sE[|M^\pi_N|^2]
\bigg\},
\end{split}
\end{align*}
where $(X^\pi,Y^\pi, Z^\pi, M^\pi)\in \cS_N$
solves   \eqref{eq:MV-fbsde_exp}
 on $\pi_N$. 
Hence,  by   Theorem \ref{thm:error_efficient_discrete_exp} and Proposition \ref{thm:time_error},
\begin{align*}
\begin{split}
&
\inf_{(y_0,\{z_i\}_i)\in \cC}
\sE[|\hat{Y}^{y,z}_N-g(\hat{X}^{y,z}_N,\sP_{\hat{X}^{y,z}_N})|^2]
=
\inf_{(y_0,\{z_i\}_i)\in \cC}
\cE_\pi(\hat{X}^{y,z},\hat{Y}^{y,z},\hat{Z}^{y,z})
\\
&\le 
C
\bigg(
\ol{\om}(\tau_N)^2
+
\cR_\pi(X,Y,Z)
+
\inf_{(y_0,\{z_i\}_i)\in \cC}
\bigg\{\sE[|\hat{Y}^{y,z}_{0}
-{Y}_{0}|^2]
+
\sum_{i=0}^{N-1}
\sE
[
|\hat{Z}^{y,z}_{i}
-\bar{Z}_{i}|^2]\tau_N
\bigg\}
\bigg)
\end{split}
\end{align*}
with $\bar{Z}_i=\frac{1}{\tau_N}\sE_i\big[\int_{t_i}^{t_{i+1}} Z_s\, ds\big]$
for all $i\in \cN_{<N}$.
The above estimate and Corollary  \ref{cor:deep_bsde} suggest that
 to show the convergence of the Deep BSDE Solver, 
it remains to show the trial space $\cC$ is large enough 
such that 
$\{(y_0(\hat{X}^{y,z}_0), \{z_i(\hat{X}^{y,z}_{i})\}_i)\mid (y_0,\{z_i\}_i)\in\cC\}$
can approximate $({Y}_{0},\{\bar{Z}_{i}\}_i)$ arbitrarily well  
(up to a time discretization error).
A complete  analysis of  this issue
for the coupled MV-FBSDE \eqref{eq:fbsde_conts_intro}
  requires a careful analysis of the decoupling fields  
and the nonlinear mapping $(y_0,\{z_i\}_i)\mapsto \hat{X}^{y,z}_{i}$,
and 
 is 
  left to future
research.

\end{Remark}



\section{Numerical experiments}\l{sec:numerical}

In this section, we illustrate the theoretical findings and demonstrate the effectiveness of the 
{a posteriori} error estimator 
 through numerical experiments. {
 We  present a one-dimensional linear MV-FBSDE example in Section \ref{sec:one_dimensional} and     multidimensional 
linear and nonlinear MV-FBSDE 
 examples in Section  \ref{sec:multi-dimensional}.
 }
 
\subsection{One-dimensional linear MV-FBSDE} 
\label{sec:one_dimensional}
 
We shall study the following linear coupled MV-FBSDE as in \cite{carmona2018a,andrea2019}:
\begin{subequations}\l{eq:fbsde_lq}
\begin{alignat}{2}
&d X_t = -\frac{1}{c_{\alpha}} Y_t \, dt + \sigma \, dW_t, 
\q t\in [0,T]; 
\qq && X_0=x_0,
\l{eq:fbsde_lq_fwd}
\\
& dY_t = -\bigg( c_x X_t + \frac{\bar{h}}{c_{\alpha}}\mathbb{E}[Y_t] \bigg) \, dt + Z_t \, d W_t,  
\q t\in [0,T];
\l{eq:fbsde_lq_bwd}
\qq 
&& Y_T = c_g X_T,
\end{alignat}
\end{subequations}
where  $x_0, T, c_{\alpha}, \sigma,  c_x , \bar{h}>0$ are some given constants
and $W=(W_t)_{t\in [0,T]}$ is a one-dimensional Brownian motion. 
This equation 
 arises from applying the Pontryagin approach
 to   a  linear-quadratic mean field game,
in which the representative agent interacts with the law of the control  instead of the law of their state.
Such a model has been used in studying  optimal execution problems for high frequency trading,
where $\bar{h}$ represents the impact of a trading strategy on the market price
and $c_\a$ represents the  cost of trading
(see Section 4.4.2 of  \cite{andrea2019}
for  interpretations of the remaining parameters).
One can easily check by using Young's inequality that 
 if the  parameters in \eqref{eq:fbsde_lq} satisfy the relation that 
$-c_x+\bar{h}^2/(4c_\a)<0$, 
then
 \eqref{eq:fbsde_lq} satisfies
 (H.\ref{assum:fbsde_discrete_exp})
with 
 $G=1$, 
 $\a=c_g$,
 $\b_1=0$
and
$\b_2=c_x-\bar{h}^2/(4c_\a)$
in (H.\ref{assum:fbsde_discrete_exp}\ref{item:monotone_exp}). {
The condition  (H.\ref{assum:time_reg}) is clearly satisfied as all coefficients are constant in the time variable.}

The linearity  of the equation implies that the decoupling field of the process $Y$ is affine,
in the sense that 
there exist deterministic functions
$(\eta_t)_{0 \leq t \leq T}$ and $(\xi_t)_{0 \leq t \leq T}$ such that 
$Y_t = \eta_t X_t + \xi_t$ for all $t\in [0,T]$.
Choosing this ansatz for the decoupling field of $Y$  and solving a system of ODEs for $\mathbb{E}[X_t]$ and $\mathbb{E}[Y_t]$ (see pages 310--312 in \cite{carmona2018a} for details on these computations), 
we can obtain for all $t\in [0,T]$ that  
\begin{align}\l{eq:decouple_lq}
\begin{split}
& \eta_t = - c_{\alpha} \sqrt{c_x/ c_{\alpha}} \frac{c_{\alpha} \sqrt{c_x/ c_{\alpha}} - c_g - (c_{\alpha} \sqrt{c_x/ c_{\alpha}} + c_g)e^{2 \sqrt{c_x/ c_{\alpha}}(T-t)}}{c_{\alpha} \sqrt{c_x/ c_{\alpha}} - c_g + (c_{\alpha} \sqrt{c_x/ c_{\alpha}} + c_g)e^{2 \sqrt{c_x/ c_{\alpha}}(T-t)}}, \\
& \xi_t = \frac{\bar{h}}{c_{\alpha}} \int_{t}^{T} \mathbb{E}[Y_s] e^{-\frac{1}{c_{\alpha}} \int_{t}^{s} \eta_u \, {d}u} \, {d}s.
\end{split}
\end{align}
The mean of $Y_t$ can  also be explicitly  expressed as
$\mathbb{E}[Y_t] = x_0 \bar{\eta}_t e^{- \frac{1}{c_{\alpha}}\int_{0}^{t} \bar{\eta}_u \, {d}u}$
for all $t \in [0,T]$,
where
\begin{align*}
& \bar{\eta}_t = \frac{-C \left( e^{(\delta^{+} - \delta^{-})(T-t)} -1 \right) - c_g \left( \delta^{+} e^{(\delta^{+} - \delta^{-})(T-t)} - \delta^{-} \right) }{\left( \delta^{-} e^{(\delta^{+} - \delta^{-})(T-t)} - \delta^{+} \right) - c_g B\left(e^{(\delta^{+} - \delta^{-})(T-t)} - 1 \right)}
\end{align*}
for $t\in [0,T]$,
$B={1}/{c_{\alpha}}$, $C=C_x$ and $\delta^{\pm}= -D \pm \sqrt{D^2+BC}$
with $D=-{\bar{h}}/{(2 c_{\alpha})}$.
Applying  It\^{o}'s formula to the decoupling field of $Y$ 
further implies that the process $Z$ is  deterministic and can be expressed as 
$Z_t = \sigma \eta_t$ for all $t\in [0,T]$. 
These  explicit expressions of the decoupling fields for $Y$ and $Z$ allow us  to compare the exact squared $L^2$-error of  a given numerical solution with the {a posteriori} error estimator, both qualitatively and quantitatively. 

To obtain numerical approximations of the solution triple $(X,Y,Z)$, we shall employ
a hybrid scheme consisting of  
 Picard iterations for the decoupling field of $Y$,
an explicit forward Euler discretization of \eqref{eq:fbsde_lq_fwd},
an explicit backward Euler discretization of \eqref{eq:fbsde_lq_bwd}
 and the least-squares Monte Carlo approximation of conditional expectations (see e.g.~\cite{gobet2006}),
 which is similar to  the Markovian iteration scheme proposed  in \cite{bender2008}
 for solving 
weakly coupled FBSDEs without mean field interaction.

We now briefly outline the main steps of the numerical procedure 
for the reader's convenience.
Let $N\in \sN$, $K\in \sN$,
$\pi_N=\{t_i \}_{i=0}^N$ be
 the uniform  partition of $[0,T]$
 with stepsize $\tau_N=T/N$,
 $\gamma = \{\gamma_k\}_{k=1}^K$
 be a set of basis functions on $\sR$
 and $P\in \sN$ be the number of Picard iterations.
We shall seek the following approximate decoupling fields  on $\pi_N$:
 $$
 \hat{y}^P_i\coloneqq ( \hat{\a}^P_i)^*\gamma :\sR\to \sR,
 \q 
 \hat{z}^P_i\coloneqq (\hat{\b}^P_i)^*\gamma: \sR\to \sR,
\q i=0,\ldots, N-1,
$$
where for each $i$,
$\hat{\a}^P_i, \hat{\b}^P_i\in \sR^K$ 
are some unknown deterministic weights to be determined.
After determining the weights
$(\hat{\a}^P_i,\hat{\b}^P_i)_{i=0}^{N-1}$,
we  define the  approximation
$(\hat{X}^P,\hat{Y}^P,\hat{Z}^P)$
of the solution triple 
$(X, Y,Z)$ as follows: $\hat{X}^P_0=x_0$,
\bb\l{eq:X^P}
\hat{X}^P_{i+1} \coloneqq \hat{X}^P_i -\frac{1}{c_{\alpha}} \hat{y}^P_i(\hat{X}^P_i)\tau_N + \sigma \Delta W_i,
\q i=0,\ldots, N-1,
\ee
$\hat{Y}^P_N\coloneqq c_g\hat{X}^P_N$
and 
for each $i=0,\ldots, N-1$, 
$\hat{Y}^P_i\coloneqq \hat{y}^P_i(\hat{X}^P_i)$
and 
$\hat{Z}^P_i\coloneqq \hat{z}^P_i(\hat{X}^P_i)$.
For the present one-dimensional case, we choose for simplicity  a set of local basis functions
which are indicators  of disjoint   partitions of a chosen computational domain
$[x_{\min}, x_{\max}]$:
we set for each $K \geq 3$ that
\begin{align*}
& \gamma_1(x) = \mathbf{1}_{( -\infty, x_{\min} )}(x), \qquad \gamma_K(x) = \mathbf{1}_{ [x_{\max},\infty) }(x) \\
& \gamma_{k+1}(x) = \mathbf{1}_{  [x_{\min}+(x_{\max}-x_{\min})\frac{k-1}{K-2}, x_{\min}+(x_{\max}-x_{\min})\frac{k}{K-2}) }(x), \qquad k=1, \ldots, K-2.
\end{align*}

To compute the   weights
$(\hat{\a}^P_i,\hat{\b}^P_i)_{i=0}^{N-1}$, we shall employ  Picard iterations with  least-squares Monte Carlo regression,
 starting with an initial guess 
 $(\hat{\a}^{0}_i)_{i=0}^{N-1}\in \sR^{KN}$
 of the weights  for $Y$.
Let  $p\in \{1,\ldots, P\}$.
We 
assume the 
approximate decoupling field of $Y$ for the $(p-1)$-th Picard iteration 
has been determined by $\hat{y}^{p-1}=( \hat{\a}^{p-1})^*\gamma$
and
 consider the $p$-th Picard iteration.  
For each $i\in \{0,\ldots, N-1\}$,
let 
$(\Delta W_{i}^{\lambda})_{\lam=1}^\Lambda$
be a  family of independent copies of the Brownian increment $\Delta W_{i}$.
We shall first generate 
$\hat{X}^{p-1,\Lambda}$ by 
following \eqref{eq:X^P}
with 
the decoupling fields
$(\hat{y}^{p-1}_i)_{i=0}^{N-1}$ and 
the increments
$(\Delta W_{i}^{\lambda})_{i,\lam}$:
$\hat{X}^{p-1,\lambda}_0=x_0$,
\bb\l{eq:X^P-1}
\hat{X}^{p-1,\lambda}_{i+1} \coloneqq \hat{X}^{p-1,\lambda}_i -\frac{1}{c_{\alpha}} \hat{y}^{p-1}_i(\hat{X}^{p-1,\lambda}_i)\tau_N + \sigma \Delta W^\lambda_i,
\q i=0,\ldots, N-1,\,
\lambda=1,\ldots, \Lambda,
\ee 
and then employ  a backward pass to update the weights 
$(\hat{\a}^{p}_i,\hat{\b}^{p}_i)_{i=0}^{N-1}$:
set $ \hat{y}^p_N(x)=c_gx$ for all $x\in \sR$, and for all $i=N-1,\ldots, 0$, let
\begin{align*}
& \hat{\b}^p_i \in \argmin_{ \b \in \mathbb{R}^K} \sum_{\lambda = 1}^{\Lambda} \left | \frac{\Delta W^{\lambda}_{i}}{\tau_N} \hat{y}^{p}_{i+1}(\hat{X}^{p-1,\lambda}_{i+1}) - \b^*\gamma(\hat{X}^{p-1,\lambda}_{i})  \right |^2, \nonumber \\
&\hat{ \alpha}^p_i \in  
\argmin_{ \alpha \in \mathbb{R}^K} \sum_{\lambda = 1}^{\Lambda} \left |\hat{y}^{p}_{i+1}(\hat{X}^{p-1,\lambda}_{i+1}) 
+ \tau_N
\bigg(
c_x \hat{X}^{p-1,\lambda}_{i}
+\frac{\bar{h}}{c_{\alpha}} 
\frac{1}{\Lambda}\sum_{\lambda=1}^\Lambda
\hat{y}^{p}_{i+1}(\hat{X}^{p-1,\lambda}_{i+1})
\bigg)
%
 - \alpha^*\gamma(\hat{X}^{p-1,\lambda}_{i})  \right |^2, \nonumber \\
& \hat{y}^{p}_i \coloneqq (\hat{ \alpha}^p_i)^*\gamma,\q 
 \hat{z}^{p}_{i}\coloneqq (\hat{\b}^p_i)^*\gamma, 
\end{align*}
 where 
we have taken  
a backward implicit discretization for  $X_t$
and also
replaced $\sE[Y_t]$ in \eqref{eq:fbsde_lq_bwd}
by the empirical mean
 in the updating scheme for $ \hat{ \alpha}^p_i$.
This procedure is repeated until the last Picard step (with $p=P$),
which determines the numerical solution 
$(\hat{X}^P,\hat{Y}^P,\hat{Z}^P)$ as  in \eqref{eq:X^P}.

%
 The error of the above hybrid scheme depends on the number of Picard iterations $P$, the number of time steps $N$, the number of basis functions $K$ and the sample size $\Lambda$. We are not aware of any published \emph{a priori} error estimates for solutions to \eqref{eq:fbsde_lq}, and even if they were available, they would almost certainly not be able capture the complicated dependence on these numerical parameters in a sharp enough way so as to give a complete, practically useful guide on choosing computationally efficient parameter combinations. 
 In contrast,  as we shall see shortly, 
 the proposed error estimator \eqref{eq:error_estimator_intro} gives a very accurate prediction of the true approximation error of a given numerical solution,
 which provides a guidance on the choices of these discretization parameters.
Note that, thanks to the explicit expressions of the true decoupling fields \eqref{eq:decouple_lq}, 
we can express the squared approximation error of a given numerical solution $(\hat{X}^P,\hat{Y}^P,\hat{Z}^P)$
on the grid as
\begin{align}
& \max_{0 \leq i \leq N} 
\big( \mathbb{E}[|X_{i}-\hat{X}^{P}_i|^2] + \mathbb{E}[|Y_{i}-\hat{Y}^{P}_i|^2]
\big) + \sum_{i=0}^{N-1}   \mathbb{E}[|Z_{i}-\hat{Z}^{P}_i|^2] \tau_N 
\nb\\
& = \max_{0 \leq i \leq N} \big( \mathbb{E}[|X^{\textrm{MS}}_{i}-\hat{X}^{P}_i|^2] + \mathbb{E}[|\eta_{t_i}X^{\textrm{MS}}_{i} + \xi_{t_i} -\hat{Y}^{P}_i|^2]\big) + \sum_{i=0}^{N-1}  \mathbb{E}[|\sigma \eta_i-\hat{Z}^{P}_i|^2]\tau_N + \mathcal{O}(N^{-2}),
\label{eq:TrueError}
\end{align} 
where $X^{\textrm{MS}}$ is an approximation of $X$  obtained by using an explicit Euler scheme 
of \eqref{eq:fbsde_lq_fwd}
(which coincides  with the Milstein scheme here)
with the drift term $ -\frac{1}{c_{\alpha}} Y_t=-\frac{1}{c_{\alpha}} (\eta_tX_t + \xi_t)$, $t\in [0,T]$.  
On the other hand, 
for a numerical solution $(\hat{X}^P,\hat{Y}^P,\hat{Z}^P)$
generated by the above hybrid scheme on a grid $\pi_N$,
the {a posteriori} error estimator \eqref{eq:error_estimator_intro}
will simplify to 
\begin{align}\label{eq:errorEstimate} 
\cE_\pi(\hat{X}^{P},\hat{Y}^{P},\hat{Z}^{P}) = \max_{0\le i\le N-1}
\sE\bigg[
\bigg|
\hat{Y}^{P}_{i+1}-\hat{Y}^{P}_0
+\sum_{j=0}^{i}
\left(
\bigg( c_x \hat{X}^P_j+ \frac{\bar{h}}{c_{\alpha}}\mathbb{E}[\hat{Y}^{P}_j] \bigg)\tau_N- { \hat{Z}^{P}_j} \,\Delta W_j
\right)
\bigg|^2
\bigg],
\end{align}   
which will
 be used to examine the approximation accuracy without using explicit knowledge of the exact decoupling fields of $Y$ and $Z$.

For our numerical experiments,
we set 
the model parameters as 
$x_0=1$, $T=1$,
$c_{\alpha} = 10/3$, $\sigma = 0.7$, $c_x = 2$, $\bar{h}=2$, $x_0=1$
and $c_g=0.3$ as in \cite{carmona2018a}
(note that  these parameters satisfy 
$-c_x+\bar{h}^2/(4c_\a)<0$
and hence 
 (H.\ref{assum:fbsde_discrete_exp}) holds).
We will also
examine the robustness of the estimator  \eqref{eq:errorEstimate} 
by 
fixing the   parameters 
$(x_0, T, \sigma,c_x, \bar{h},c_g)$
and
 increasing the coupling parameter $1/c_\a$,
 whose values will be specified later.
Since the forward equation starts with $x_0=1>0$, we shall implement the above hybrid  scheme
with
the computational domain 
$[x_{\min}, x_{\max}]=[0,2]$
and
 the following choices of  
$N,K,\Lambda$
 as suggested in
\cite{bender2013}:
\begin{align}\l{eq:NKLambda}
N = \left[2 \sqrt{2}^{j-1} \right], \qquad K = \max \left \lbrace \left \lceil \sqrt{2}^{j-1} \right  \rceil, 3 \right \rbrace, \qquad \Lambda = \left[2 \sqrt{2}^{l(j-1)} \right]
\end{align} 
for $j=2, \ldots, 9$ and $l=3,4,5$,
where $[x]$ is the nearest integer to $x\in \sR$ and $\lceil x \rceil$ is the  smallest integer not less than $x\in \sR$.
We choose for simplicity
the initial guess 
 $(\hat{\a}^{0}_i)_{i=0}^{N-1}$
 of the decoupling field
 to be the constant matrix $1/K$ for each $K$,
 and specify the number of Picard iterations  $P$ later,
 which will
  depend on the value of the coupling parameter $1/c_\a$.

To evaluate \eqref{eq:TrueError} and \eqref{eq:errorEstimate}
for a given numerical solution $(\hat{X}^P,\hat{Y}^P,\hat{Z}^P)$,
represented by the approximate decoupling fields, 
 we shall 
 simultaneously 
generate $10^4$ independent sample paths of $X^{\textrm{MS}}$ and $\hat{X}^P$,
and replace the expectations in  \eqref{eq:TrueError} and \eqref{eq:errorEstimate}
by  empirical  means over these sample paths.\footnotemark
\footnotetext{
Note that the mean field term 
$\mathbb{E}[\hat{Y}^{P}_j]$
appearing in the estimator \eqref{eq:errorEstimate} 
will  also be replaced by an empirical mean based on these forward simulations of $\hat{X}^P$,
which, strictly speaking, implies that 
 \eqref{eq:errorEstimate} is estimated based on
  $10^4$   identically distributed but  non-independent realizations 
  (also known as  an interacting  particle system of size $10^4$).
It is possible to recover the independence assumption of the  law of large numbers,
by further simulating multiple  independent realizations of such particle systems (each of size $10^4$) and 
then estimating the outer expectation in  \eqref{eq:errorEstimate}
 via an empirical average over these  independent realizations (see, e.g.,~\cite{hajiali2018}).
 However,
our experiments show that
 for such a large number of sample paths,
one  realization of the particle system is sufficient
to   evaluate 
\eqref{eq:errorEstimate} accurately,
since different independent realizations of the particle estimators
 usually lead to 
 negligible variances compared to other discretization errors,
 which can be explained by the well-known
``propagation of chaos" phenomenon (see, e.g.,~\cite{carmona2018a}).
For example, for the numerical solution obtained with
$c_a=10/3$,
 $j=9$ and $l=5$,
64  independent realizations of the particle estimator   (each of size $10^4$)
estimate the squared $L^2$-error to be 0.07
with a variance  of  magnitude $10^{-6}$.
}
We  remark that 
on the basis of our experiments, $10^4$ sample paths  seem to be sufficiently large 
for an accurate evaluation of \eqref{eq:TrueError} and \eqref{eq:errorEstimate}, 
since  further increasing the number of sample paths  results in negligible differences in the estimated values.
All computations are performed using \textsc{Matlab} R2019b 
on a 2.30GHz Intel Xeon Gold 6140  processor.

\begin{figure}[!h]
\centering
\includegraphics[height=7cm,keepaspectratio]{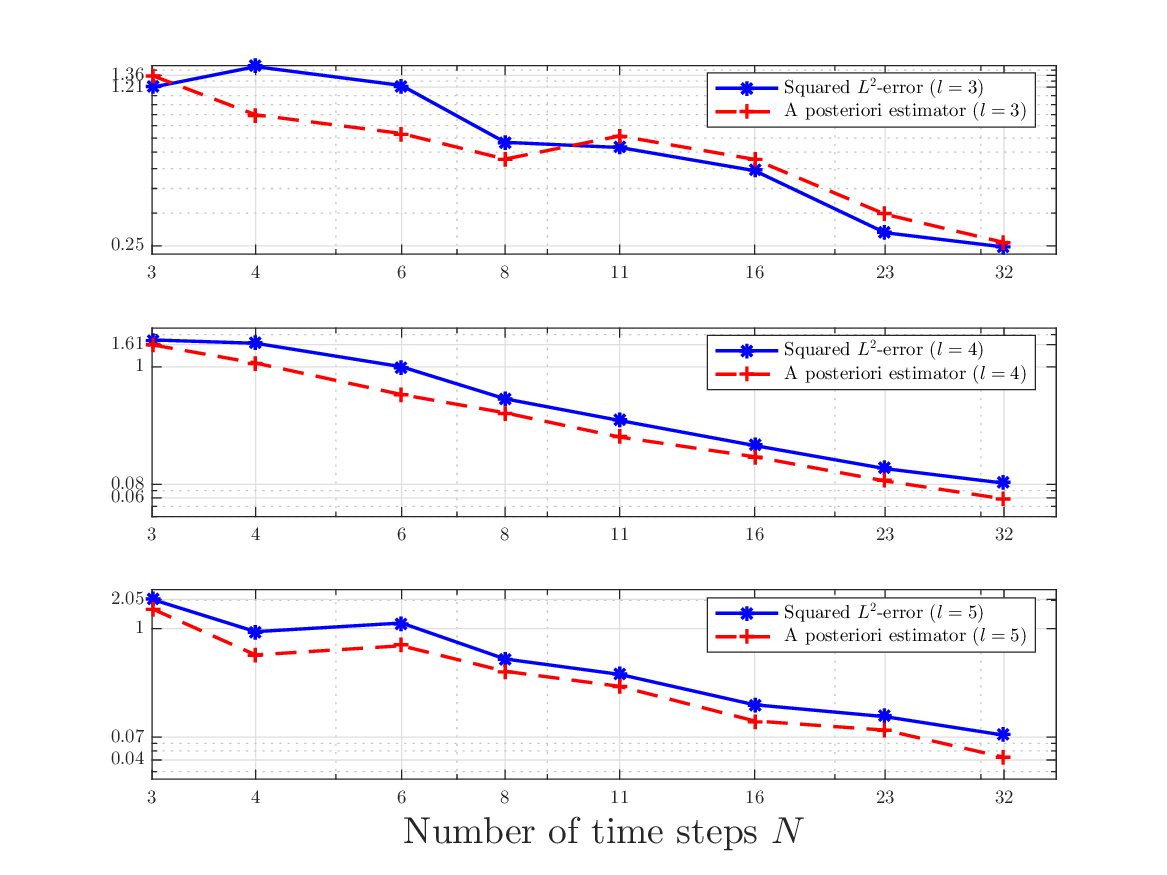}
\caption{
Comparison between the  squared $L^2$-error 
and  the {a posteriori} error estimator with different 
time steps and sample sizes
(plotted in a log-log scale).}
\label{fig:TurePost}
\end{figure}

Figure \ref{fig:TurePost} compares  the  squared $L^2$-errors 
and  
the estimated squared errors (by using \eqref{eq:errorEstimate})
 for
numerical solutions obtained with 5 Picard iterations (i.e., $P=5$),
and
different  time steps $N$ and sample sizes $\Lambda$
as listed in \eqref{eq:NKLambda}.
We clearly observe that, for all choices of sample sizes,
the convergence behavior of  the estimated error 
and the true  error
are almost identical as  
the time stepsize tends to zero,
which confirms the theoretical results in Theorem \ref{thm:a_posterior_conts}.
Moreover, 
the ratio of the estimated error to the true error
suggests that,
for  this set of model parameters,
 the generic equivalence constant in Theorem \ref{thm:a_posterior_conts}
lies within the range of $0.7-1.2$,
which indicates that  the error estimator predicts  the  squared approximation error very accurately.
By performing linear regression of the estimated values (the dashed line) against the number of time steps,
we can infer 
without using the analytic solution of \eqref{eq:fbsde_lq}
that
 the approximation error 
 (in the $L^2$-norm)
 converges to zero at a rate of $N^{-0.7}$
 for the cases $l=4, 5$,
while for $l=3$, 
the approximation error also converges to zero but with a much slower rate.

Note that for general decoupled FBSDEs,
Corollary 1 in \cite{gobet2006} suggests  choosing the sample size $\Lambda$ corresponding to $l=5$  
in the least-squares Monte Carlo method
to achieve a half-order  $L^2$-convergence with respect to the number of time steps $N$.
Our numerical results indicate that,
for the present example,
the convergence behaviour is  much better    than this theoretical error estimate,
possibly  due to a better time regularity of the process $Z$.
This suggests that one can design more efficient algorithms 
with tailored hyper-parameters
based on the error estimator \eqref{eq:errorEstimate}.
In particular, 
 \eqref{eq:errorEstimate}
 shows that 
 $l=4$ leads to  the most efficient  algorithm  among the three choices of $l\in \{3,4,5\}$.
The cheaper algorithm with $l=3$ in general 
results in  significantly larger errors,
while
the choice $l=5$
not only requires a tremendously higher computational cost,
but also
achieves  almost the same  
accuracy as   the choice  $l=4$
 for  sufficiently fine grids;
for instance, with $N=32$ time steps, the error estimator predicts increasing $l$ from $4$ to $5$
will only reduce the squared error 
from 0.0586 to 0.0427,
and in fact  the true squared error only reduces
from 0.0822 to 0.0734.
To illustrate the computational efforts  for the two choices $l=4,5$,
we present the 
corresponding
 sample size $\Lambda$
and  computational time
  with different numbers of time steps
in Table \ref{table:compare}.

\begin{table}[H]
 \renewcommand{\arraystretch}{1.05} 
\centering
\caption{ Sample size $\Lambda$ and computational time with different $N$ and $l$}
\label{table:compare}
\begin{tabular}[t]{@{}c c cc c cc@{}}
\toprule
 & \phantom{a}&  
 \multicolumn{2}{c}{$N=23$} &  \phantom{a}&  
 \multicolumn{2}{c}{$N=32$} \\
 \cmidrule{3-4} \cmidrule{6-7}
 $l$ & & Sample size  &   Run time & & Sample size &   Run time\\ \midrule
4 & & 32\,768  & 533s & & 	131\,072	& 3\,908s \\ 
5 & &  370\,728  & 5\,715s & & 	2\,097\,152	& 59\,338s \\ 
\bottomrule
\end{tabular}
\end{table}%

\begin{figure}[!ht]
\centering
\hspace{-12mm}
\includegraphics[trim=43 10 3 30, clip, width=0.54\columnwidth,keepaspectratio]{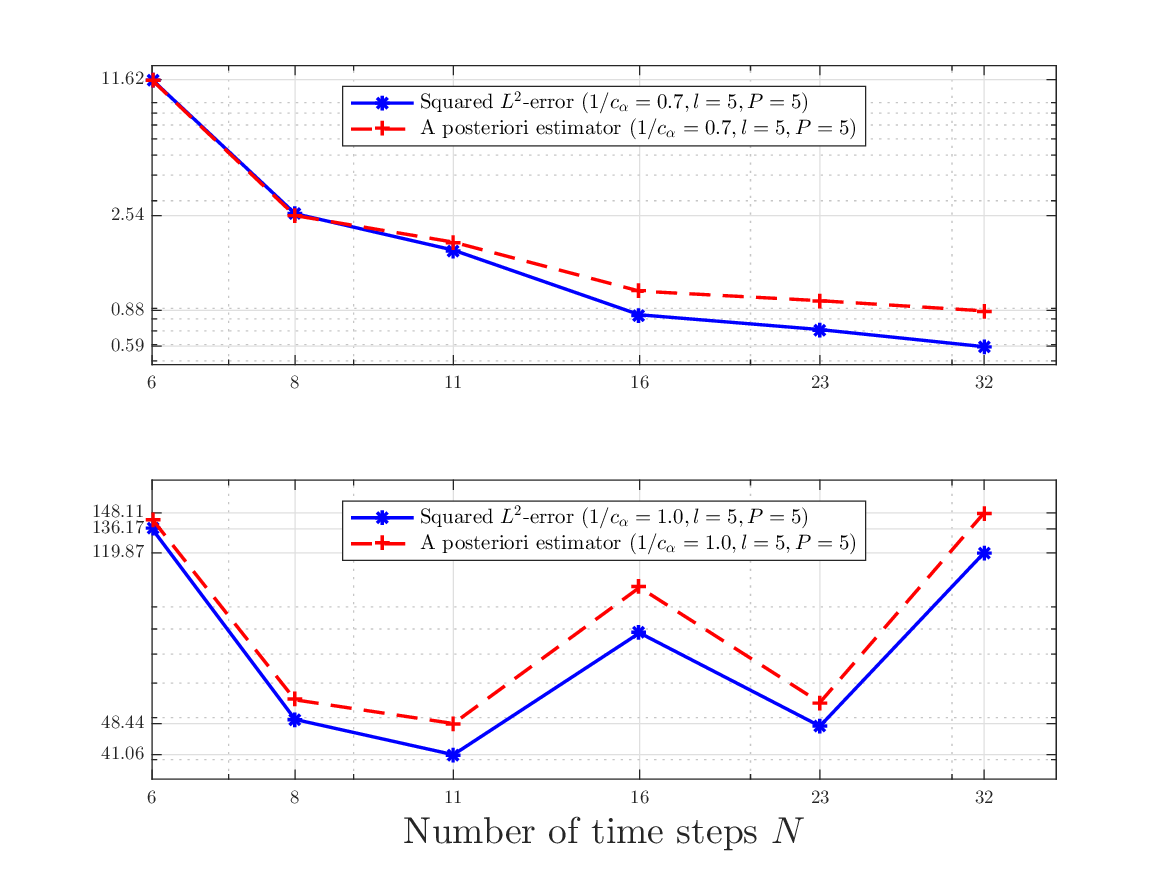}
\hspace{-7mm}
    \includegraphics[trim=43 10 3 30, clip, width=0.54\columnwidth,keepaspectratio]{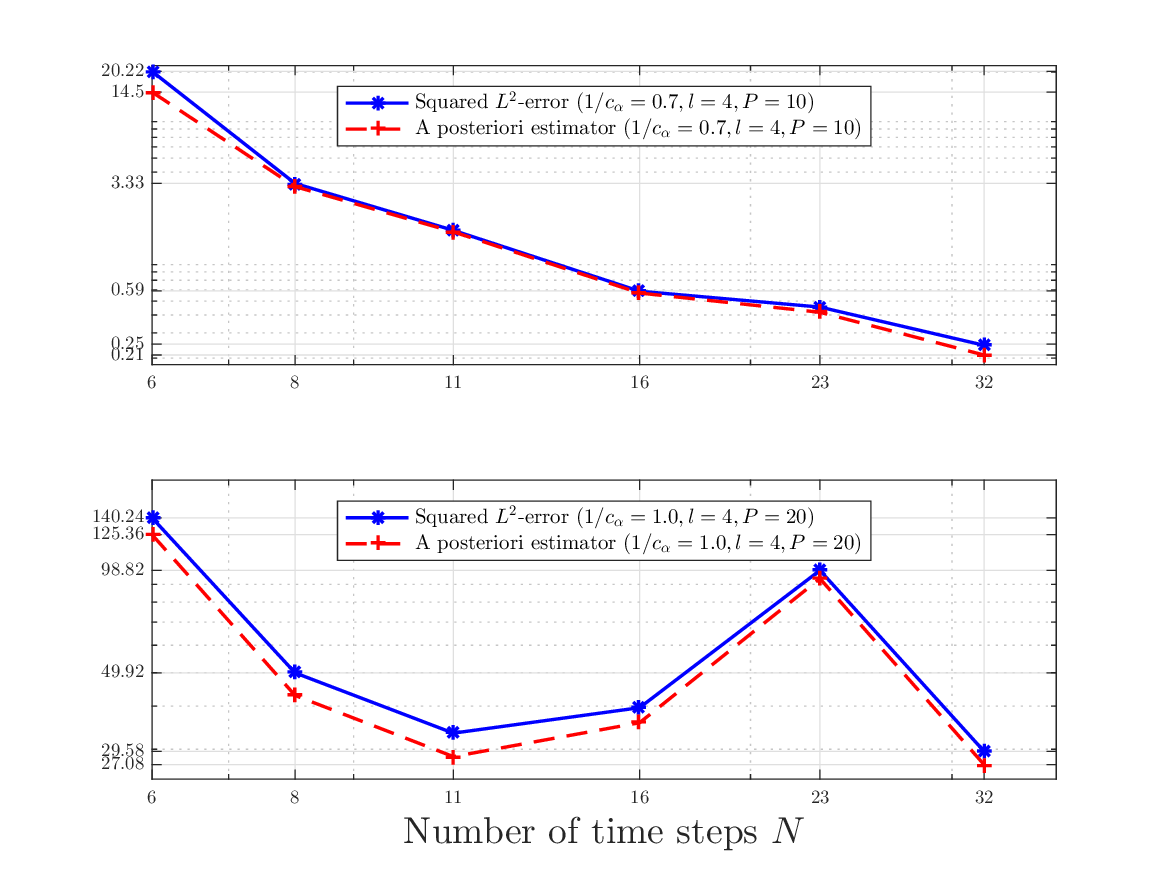}
\hspace{-20mm}    
\caption{
Robustness of  the {a posteriori} error estimator for different
coupling parameters  $c_\alpha$
(plotted in a log-log scale);
from top to bottom: numerical results with $1/c_\a=0.7$ and  $1/c_\a=1.0$;
from left to right: numerical results with larger sample size ($l=5$) but fewer Picard iterations ($P=5$),
and numerical results with smaller sample size ($l=4$) but more Picard iterations.
}
\label{fig:coupling}
\end{figure}

We then proceed to examine the performance of the error estimator 
for MV-FBSDEs with  stronger coupling,
by varying the coefficient  $1/c_\alpha \in \{ 0.7, 1\}$
and keeping the other model parameters  as above.
Figure \ref{fig:coupling} (left) presents
the numerical results obtained by 
the hybrid algorithm  with
5 Picard iterations (i.e., $P=5$)
and
the discretization parameters $N, K,\Lambda$ as defined in \eqref{eq:NKLambda}
for $j=4,\ldots, 9$, $l=5$.
By comparison with 
the numerical results for $1/c_a=0.3$
(see Figure \ref{fig:TurePost}, bottom),
we can clearly observe  that
 as the coupling parameter $1/c_a$ increases,
 the  same choice of  discretization parameters
 leads to larger approximation errors.
The $L^2$-approximation error 
decays slowly for  the case with $1/c_{\alpha}=0.7$
as the number of time steps $N$ tends to infinity, 
while for  the case with $1/c_{\alpha}=1$,
the approximation errors  oscillate around
the value  $10^2$ 
and do not show   convergence 
for sufficiently large $N$.
Similar phenomena have been observed 
in \cite{andrea2019,chassagneux2019,germain2019},
where the authors found that  a stronger coupling between the forward and backward equations
can 
pose significant numerical challenges 
such as 
 slow convergence or even divergence of  Picard iterations.

More importantly, we see that the performance of the {a posteriori} error estimator is very robust 
even for a large coupling parameter.
Regardless of the convergence of the hybrid algorithm,
the proposed error estimator captures the precise convergence behaviour of the true error
starting from a fairly small number of time steps,
and 
the ratio of the estimated error to the true error
generally stays in  the  range of $0.7-1$.
This enables us to judge the success of a given choice of discretization parameters 
without knowing the analytic solution to the problem.
In particular,  the error estimator suggests 
that for the case with $1/c_\alpha\in \{0.7,1\}$ and $P=5$,  the dominating error stems from other  sources 
(such as the Picard iteration)
instead of the time discretization or the  Monte Carlo regression.
Hence 
we cannot expect to significantly improve the approximation accuracy 
by keeping the number of Picard iterations fixed and  
only by
further refining the time grid or enlarging the sample size.

Motivated by the above observation,
we carry out the hybrid algorithm with 
 more Picard iterations 
($P=10$ for $1/c_\a=0.7$ and $P=20$ for $1/c_\a=1$)
but less simulation samples
($l=4$).
Figure \ref{fig:coupling} (right)
presents 
the numerical results for 
the discretization parameters $N, K,\Lambda$ as defined in \eqref{eq:NKLambda}
with $j=4,\ldots, 9$.
One can observe 
a significant improvement in the algorithm's efficiency
for the case with $1/c_\a=0.7$ (see Figure \ref{fig:coupling}, top-right),
where
the hybrid algorithm converges with a rate of $N^{-0.8}$
for the whole range of time steps,
and results in more accurate numerical solutions with less computational time
than the original choice of $P=5$, $l=5$ (see Figure \ref{fig:coupling}, top-left).
The situation is less clear for  the case with $1/c_\a=1$ (see Figure \ref{fig:coupling}, bottom-right).
Although the error is reduced by half
as compared to the choice of $P=5$ and $l=5$,
the error estimator does not decrease significantly starting from $N=11$,
which suggests that
more Picard iterations or a better scheme need to be employed for further improvements.\footnotemark
\footnotetext{
{
Alternative approaches to decouple \eqref{eq:MV-fbsde_exp}
include the   fictitious play approach in  \cite{han2022} 
and the gradient descent approach in \cite{reisinger2021b}. 
Rather than
replacing the former iterate by the new one as in  Picard iteration,
these methods update the approximate solutions with a smaller rate 
to ensure the convergence of algorithms. 
}
}

\subsection{Multidimensional linear and nonlinear MV-FBSDEs}
\label{sec:multi-dimensional}

{
In this section, we demonstrate the effectiveness of the  a posteriori   estimator \eqref{eq:error_estimator_intro} 
for the following multidimensional coupled MV-FBSDEs:
   for all $t\in [0,T]$, 
\begin{align}
\label{eq:fbsde_cs}
\begin{split}
 d X_t& =V_t\,d t, \quad  
 d V_t=
 \bigg(
 \mathbb{E}[\kappa(x,v,X_t,V_t)]\big\vert_{(x,v)=(X_t,V_t)}
 -\frac{1}{2\gamma}Y^{2}_t \bigg)\,d t+\sigma \, d W_t, 
\\
dY^1_t &=- \bigg( \mathbb{E}[\p_x \kappa(x,v,X_t,V_t)]\big\vert_{(x,v)=(X_t,V_t)}Y^2_t
+\mathbb{E}[\p_{x'} \kappa(X_t,V_t,x,v)Y^2_t]\big\vert_{(x,v)=(X_t,V_t)}\bigg)\, dt 
+Z^1_t \, d W_t, 
\\
dY^2_t &=- \bigg( Y^1_t+ \mathbb{E}[\p_v \kappa(x,v,X_t,V_t)]\big\vert_{(x,v)=(X_t,V_t)}Y^2_t
+\mathbb{E}[\p_{v'} \kappa(X_t,V_t,x,v)Y^2_t]\big\vert_{(x,v)=(X_t,V_t)}
\\
&\quad + 2(V_t-\mathbb{E}[V_t]) \bigg)\, dt 
+Z^2_t \, d W_t, 
\\
X_0&= x_0, \quad V_0=v_0, \quad 
Y^1_T=0, \quad Y^2_T=2(V_T-\mathbb{E}[V_T]),
\end{split}
\end{align}
where 
  $T>0$, $n\in \mathbb{N}$,  $\gamma>0$ and $\sigma\in \mathbb{R}^{n\times n}$ are given constants, 
  $ x_0,v_0$
are    given $\mathbb{R}^{n}$-valued square integrable random variables,
$W$ is an $n$-dimensional standard Brownian motion, 
 $\kappa:\mathbb{R}^n \times \mathbb{R}^n \times \mathbb{R}^n \times \mathbb{R}^n \to \mathbb{R}^n$ is 
 the interaction  kernel 
  given by
\bb\label{eq:commmunication_kernel}
\kappa(x,v,x',v')=\frac{v'-v }{(1+|x-x'|^2)^\beta},
\quad \textnormal{with some given  $\beta\ge 0$,}
\ee
and 
$X,V,Y^1,Y^2,Z^1,Z^2$ are unknown $n$-dimensional  solution processes.  
 The  equation \eqref{eq:fbsde_cs} arises from applying the Pontryagin maximum principle to 
 an optimal control problem of  multidimensional stochastic  mean-field Cucker--Smale   dynamics, 
where the controller  applies an external force to induce a
consensus of the velocity process $V$     
  (i.e., all trajectories of the velocity process 
tend to the same value as $t$ increases).
We refer the reader to \cite[Section 5.2]{reisinger2021b} for details of the control problem
and to  \cite{bailo2018,nourian2011,carrillo2010,han2022} for similar control problems.

 If $\beta =0$ in \eqref{eq:commmunication_kernel},
 then  
\eqref{eq:fbsde_cs} reduces to a linear MV-FBSDE, which satisfies (H.\ref{assum:fbsde_discrete_exp}) and (H.\ref{assum:time_reg}). 
For   
$\beta>0$, the coefficients of \eqref{eq:fbsde_cs}  
exhibit a more complicated interaction through the nonlinear kernel \eqref{eq:commmunication_kernel}.
The coupling  strength of the forward and backward dynamics in \eqref{eq:fbsde_cs}
is determined by the parameter $\gamma$, 
i.e., the smaller the parameter $\gamma$, the stronger the coupling. 
In the sequel, we 
fix $T=1$, $\sigma= 0.1\mathbb{I}_n$
and $(x_0,v_0)\sim \textrm{Unif}([0,1]^{2n})$, and 
examine the performance of \eqref{eq:error_estimator_intro}
for different choices of $\gamma, \beta>0$ and $n\in \mathbb{N}$.
As we shall see soon, 
although       (H.\ref{assum:fbsde_discrete_exp}) may not hold for general $\beta,\gamma>0$,   
the error estimator \eqref{eq:error_estimator_intro} still quantifies the approximation errors very well.

\paragraph{Two-dimensional nonlinear examples. }

We first carry out the experiments with 
  $n=1$, $\gamma=0.3$ and $\beta\in \{1,10\}$. 
For any given $\beta$, we compute approximate solutions to the two-dimensional MV-FBSDE \eqref{eq:fbsde_cs}
by the deep BSDE method introduced in     \eqref{eq:deep_fbsde_exp}. 
More precisely, 
we use 
a    neural  network 
$f_\theta:\mathbb{R}^2\to \mathbb{R}^2$ 
(with one hidden layer of width 20 and 
 the sigmoid  activation function)
to approximate the decoupling field of $(Y_0^1,Y_0^2)$,
and 
use 
 a    neural  network 
 $g_\vartheta:\mathbb{R}^3\to \mathbb{R}^2$
(with 
one hidden layer of width 110 and 
the sigmoid  activation function)
to approximate the decoupling fields of $(Z^1,Z^2) $ at all times. 
We then take a uniform grid $\pi$ of  $[0,1]$ with stepsize $1/32$,
and  compute the discrete solution $\hat{\Theta}=(\hat{X},\hat{V}, \hat{Y}^1, \hat{Y}^2, \hat{Z}^1, \hat{Z}^2)$  of  \eqref{eq:fbsde_cs} 
using the explicit forward Euler scheme  \eqref{eq:deep_fbsde_exp}
on the grid $\pi$.
Note that the discrete solution
$\hat{\Theta}$ 
   depends on the network parameters 
$(\theta,\vartheta)$, which we update  
 iteratively by minimising the following terminal loss: 
 \bb\label{eq:cs_terminal}
\mathcal{E}_\pi(\hat{\Theta})=\sE[|\hat{Y}^1_N|^2]+
\sE[|\hat{Y}^2_N-2(\hat{V}_N-\mathbb{E}[\hat{V}_N])|^2].
\ee
 For each   iteration, we consider   a particle approximation of size $500$ to \eqref{eq:fbsde_cs},   estimate   the law   $\mathbb{P}_{\hat{\Theta}}$ via its empirical distribution, 
 and update the network  parameters with  the Adam algorithm. 
 This yields a sequence of parameters $(\theta_\ell,\vartheta_\ell)_{\ell \in \mathbb{N}}$,
 which in turn   yields a sequence of approximate solutions  $(\hat{\Theta}^\ell)_{\ell \in \mathbb{N}}$ to \eqref{eq:fbsde_cs}. 

To assess the accuracy of $(\hat{\Theta}^\ell)_{\ell \in \mathbb{N}}$, 
we obtain a  reference solution using  the iterative  PDE  method   introduced in \cite[Section 5.2]{reisinger2021b},
as  the exact solution to \eqref{eq:fbsde_cs}  is  not known.\footnotemark
\footnotetext{For the PDE method, we choose the computational domain $[-1,3]^2$, time stepsize $1/64$ and mesh size $1/100$, which   lead to 
 negligible discretization errors on the basis of our experiments.} 
This allows for 
computing the   squared approximation errors of $(\hat{\Theta}^\ell)_{\ell\in \mathbb{N}}$ defined in \eqref{eq:error_T_N}. 
We shall compare the approximation error of $ \hat{\Theta}^\ell$ with
the predicted error   given by 
 the error estimator   \eqref{eq:error_estimator_intro}, 
which 
 simplifies  to the terminal loss \eqref{eq:cs_terminal},  
and  is   estimated by  a particle approximation of size $5000$.

Figure \ref{fig:cs_2d_different_beta} compares the   squared $L^2$-error of  numerical solutions with 
the estimated error from the   error estimator \eqref{eq:error_estimator_intro},
 for   different values of  $\beta$ and Adam iterations.
It can be observed that the error estimator  tracks the true error well 
starting from a fairly small number of Adam iterations.
The ratio of the estimated error to the true error consistently falls within   the range of $0.6-0.8$ throughout all iterations, 
and is robust to   the changes in the value     of $\beta$.
This suggests  that   the estimator provides  a reliable measure of the approximation error.    

 
\begin{figure}[!h]
\centering
\includegraphics[height=7cm,keepaspectratio]{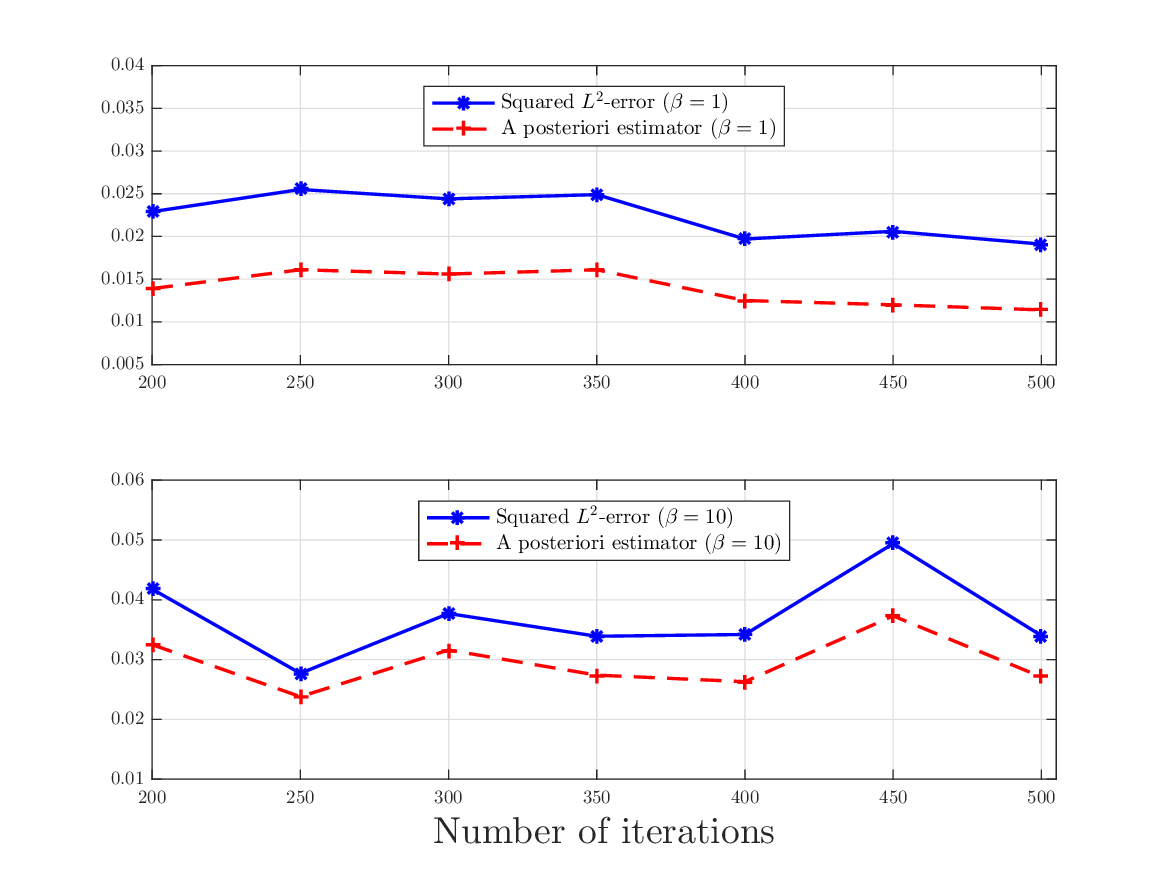}
\caption{
Comparison between the  squared $L^2$-error 
and  the {a posteriori} error estimator 
with  
$\beta\in \{1,10\}$ and different Adam iterations
 for two-dimensional MV-FBSDE   \eqref{eq:fbsde_cs}.}
\label{fig:cs_2d_different_beta}
\end{figure}

 \paragraph{High dimensional linear examples. }
We then perform    experiments with 
 $\beta =0$, $\gamma\in \{0,2, 0.3, 0.5\}$
 and   $n\in \{3, 6, 9\}$.  
As $\beta =0$, the interaction kernel $\kappa$ in \eqref{eq:commmunication_kernel} is independent of $x$ and $x'$, 
which along with $Y^2_T=0$ implies that $Y^1_t=Z^1_t=0$ for all $t\in [0,T]$. 
 Moreover, by the linearity of    $\kappa$ in $v$ and $v'$ and the terminal condition of $Y^2$, 
 one can show by It\^{o}'s formula  that  
 \bb\label{eq:cs_3d}
 Y_t^2=\alpha_t (V_t-\mathbb{E}[V_t]), 
\quad Z_t^2=\sigma \alpha_t,
\quad t\in [0,T],
\ee
 where 
 $\alpha:[0,T]\to \mathbb{R}$ satisfies 
 $ a'_t -2a_t -\frac{1}{2\gamma }a_t^2+2=0$ with $a_T=2$. This provides a reference solution against which the accuracy of the given approximate solutions can be evaluated.

In the sequel, 
for each $n$ and $\gamma$, 
we focus on solving the processes $V$, $Y^2$ and $Z^2$, since 
the processes $(Y^1, Z^1)$ are zero, 
and the process $X$ can be obtained by    integrating $V$ in time. This reduces \eqref{eq:fbsde_cs} to an $n$-dimensional linear MV-FBSDE, 
which  we solve  using the deep BSDE method described   above. 
In particular, we approximate $Y^2_0$ and $(Z^2)_{t\in [0,T]}$ by 
$$
Y^2_0 \approx f_\theta(V_0), \quad Z^2_t \approx g_\vartheta(t, V_t), \quad t\in [0,T], 
$$
where  
$f_\theta:\mathbb{R}^n\to \mathbb{R}^n$
and   $g_\vartheta:\mathbb{R}^{n+1}\to \mathbb{R}^n$
are    neural  networks  
with  the sigmoid  activation function
and 
1  hidden layer of widths 20 and 110,
 respectively.  
The network parameters $(\theta,\vartheta)$ are updated  by applying the Adam algorithm to minimize the terminal loss:
\bb\label{eq:cs_terminal_3d}
\mathcal{E}_\pi(\hat{\Theta})= 
\sE[|\hat{Y}^2_N-2(\hat{V}_N-\mathbb{E}[\hat{V}_N])|^2].
\ee
The other discretisation parameters, such as   the time grid for   the forward Euler scheme \eqref{eq:deep_fbsde_exp}  and the size of the particle system for each Adam iteration,  are chosen as in the above two-dimensional setting. 
For any given numerical solution $(\hat{V}, \hat{Y}^2,\hat{Z}^2)$, 
the  a posteriori error estimator   \eqref{eq:cs_terminal}
is estimated  by   a particle approximation of size $5000$,
and the squared      $L^2$-error is computed using the reference solution given in \eqref{eq:cs_3d}.
 
 Figure \ref{fig:cs_3d_different_gamma}   compares the   squared $L^2$-error of  numerical solutions with the a posteriori error estimator, for 
  $n=3$, and 
 different   values of  $\gamma\in \{0.2,0.3,0.5\}$ and Adam iterations. 
The results show that the estimated error and the true error have almost identical convergence behaviour. Moreover, as the number of iterations increases, the ratio of the estimated error to the true error decreases, indicating that the error estimator predicts the approximation error more accurately.
 One may also observe a slight increase  of the estimation ratio 
as $\gamma$ approaches    $0$. 
Specifically, over the last 150 iterations, the estimation ratios for $\gamma=0.5$ lie in the range of $0.92 - 1.05$,
whereas the estimation ratios for $\gamma=0.2$ lie in the range of $1.38 - 2$. 
This suggests  the generic equivalence constant in Theorem \ref{thm:a_posterior_conts}
may increase as  the coupling between the forward and backward equations
becomes stronger.

Figure \ref{fig:cs_different_n} investigates the impact of the problem dimension $n$ on the performance of the a posteriori error estimator. As the dimensionality increases, 
the squared $L^2$-error increases linearly, 
and 
the Adam algorithm requires more iterations to achieve the same level of accuracy. Despite this dependence on dimensionality, the ratio of the estimated error to the true error remains stable.
As the number of iterations increases, 
the estimation ratios decrease and eventually stabilize within the range of $1-1.5$.
This indicates  that the a posteriori error estimator is a reliable tool for assessing the accuracy of numerical solutions
 in high-dimensional problems. 


 
 \begin{figure}[!h]
 \begin{subfigure}{0.50\textwidth}
 \centering
\includegraphics[trim=43 10 3 30, clip, width=\linewidth,keepaspectratio]{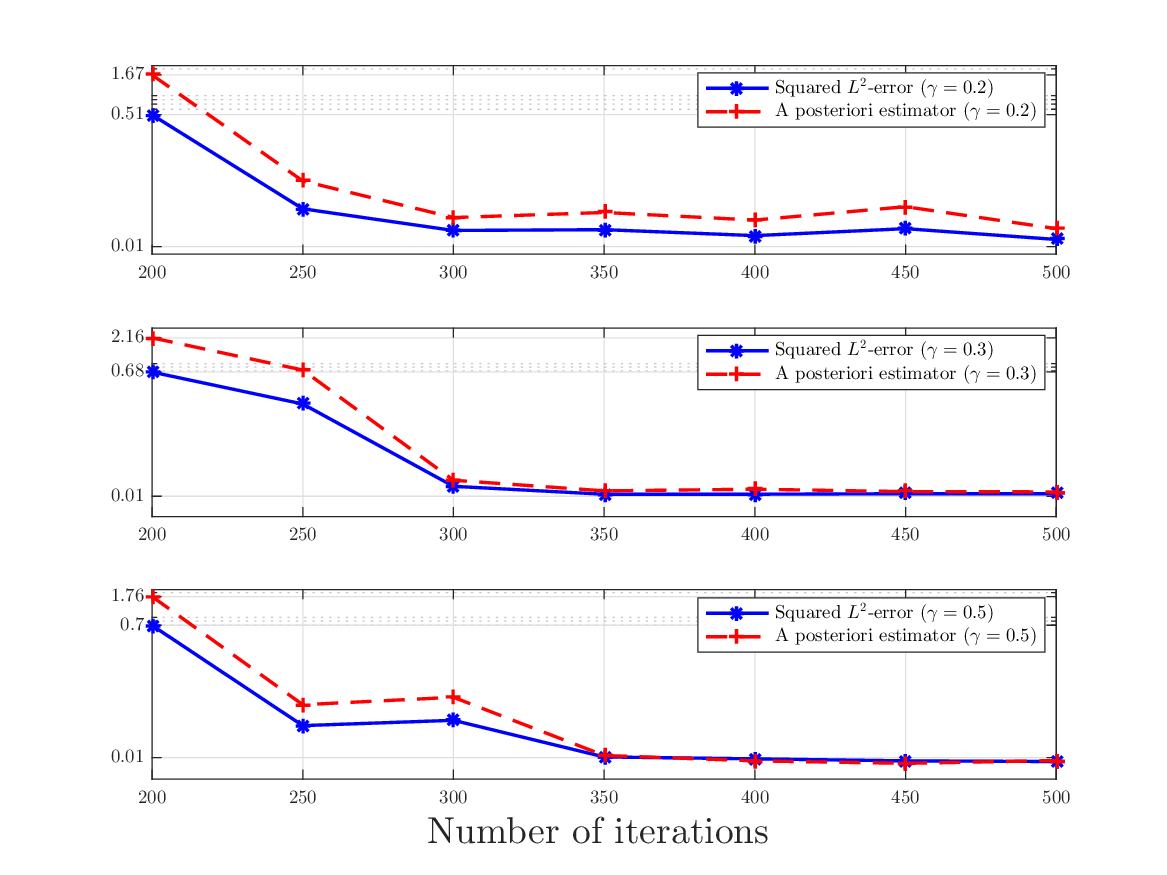}
 \caption{$n=3$, $\gamma\in \{0.2,0.3,0.5\}$}
 \label{fig:cs_3d_different_gamma}
    \end{subfigure}
%
\begin{subfigure}{0.50\textwidth}
 \centering
\includegraphics[trim=43 10 3 30, clip, width=\linewidth,keepaspectratio]{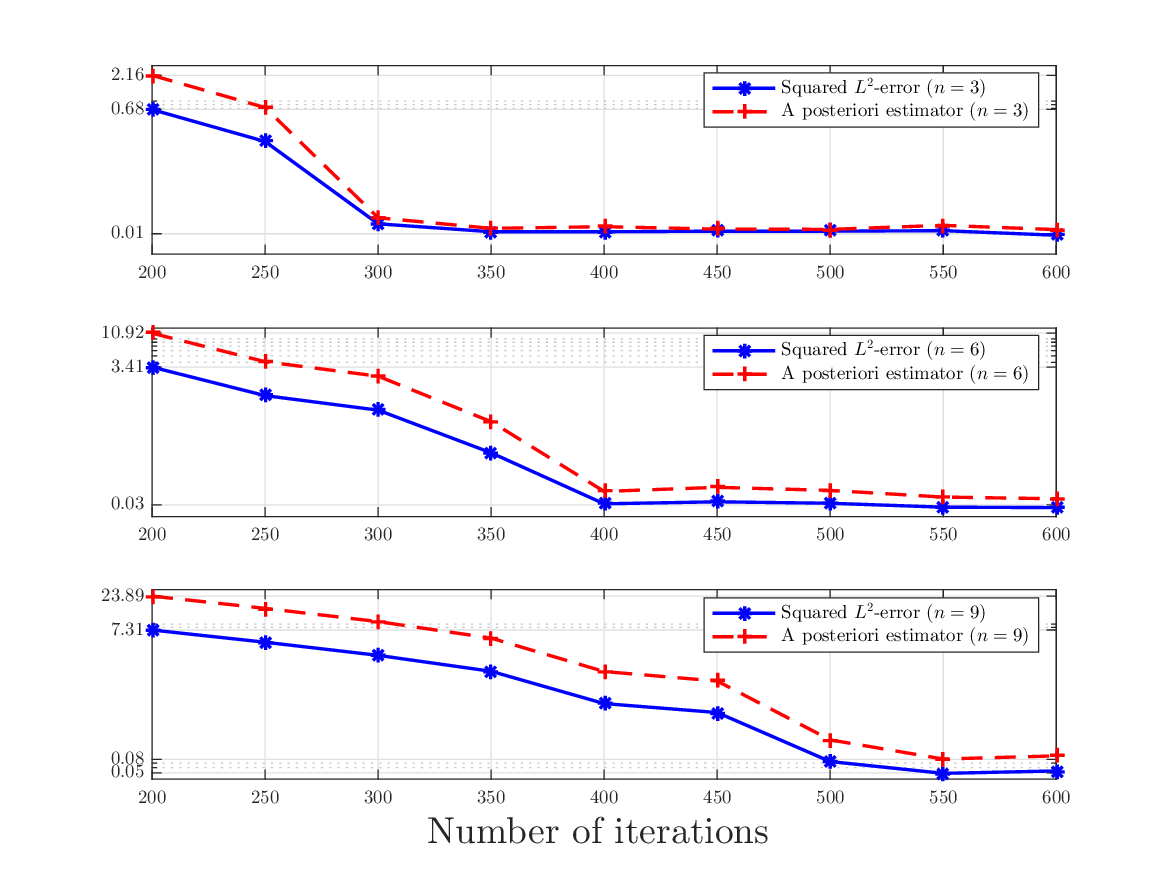}
 \caption{$\gamma=0.3$, $n\in \{3,6,9\}$}
  \label{fig:cs_different_n}
  \end{subfigure}

\caption{
Comparison between the  squared $L^2$-error 
and  the {a posteriori} error estimator 
with  different Adam iterations and values of $\gamma$ and $n$
for  linear  MV-FBSDE   \eqref{eq:fbsde_cs} (with $\beta =0$).}
\end{figure}

 }

\begin{appendices}

\section{Proofs of technical results}


\subsection{Proofs of  Proposition \ref{thm:stab_exp} and Lemma \ref{lem:linear_existence}}
\label{appendix:proof_section_discretebsde}

{
We first recall the discrete  Gronwall Lemma  given  in  \cite{holte2009}.

\begin{Lemma}
\label{lemma:gronwall}
Let $\{y_n\}_{n\in\mathbb{N}\cup\{0\}}$ and $\{g_n\}_{n\in\mathbb{N}\cup\{0\}}$ be  sequences of nonnegative real numbers,
and $c\ge 0$. If 
$y_n \le c+\sum_{k=0}^{n-1} g_k y_k$ for all $n\in\mathbb{N}\cup\{0\}$, then 
$$
\max_{k=1,\ldots n}y_k \le c\exp\left(\sum_{k=0}^{n-1}g_k\right),\quad \forall n\in\mathbb{N}\cup\{0\}.
$$ 
\end{Lemma}
}

We then establish    a preparatory   lemma for Proposition \ref{thm:stab_exp}.
\begin{Lemma}\l{lemma:stab_a_priori_exp}
Suppose the generator $(b,\sigma,f,g)$ satisfies (H.\ref{assum:fbsde_discrete_exp}),
and the generator $(\bar{b},\bar{\sigma},\bar{f},\bar{g})$ 
satisfies
(H.\ref{assum:fbsde_discrete_exp}\ref{item:integrable_exp}).
Let 
 $\a,\b_1,\b_2$ and $G$ be the constants in  (H.\ref{assum:fbsde_discrete_exp}\ref{item:monotone_exp}),
 $L$ be the constant in (H.\ref{assum:fbsde_discrete_exp}\ref{item:lipschitz_exp}),
 $N\in \sN$, $\lambda_0\in [0,1]$,
 let 
 $(\phi,\psi,\gamma),(\bar{\phi},\bar{\psi},\bar{\gamma})\in \cM^2(0,T; \sR^n\t \sR^{n\t d}\t \sR^m)$, 
 $\eta,\bar{\eta}\in L^2(\cF_T;\sR^m)$,
  $\xi_0,\bar{\xi}_0\in L^2(\cF_0;\sR^m)$,
 let 
$(X,Y, Z, M)\in 
\cS_N$
(resp.~$(\bar{X},\bar{Y}, \bar{Z}, \bar{M})\in 
\cS_N$)
satisfy \eqref{eq:MV-fbsde_exp_lambda}
defined on  $\pi_N$ 
corresponding to  
$\lambda=\lambda_0$,
 the generator $(b,\sigma,f,g)$
 and 
$(\phi,\psi,\gamma, \eta,\xi_0)$
(resp.~the generator $(\bar{b},\bar{\sigma},\bar{f},\bar{g})$
 and  
$(\bar{\phi},\bar{\psi},\bar{\gamma},\bar{\eta},\bar{\xi}_0)$),
and
let 
$(\delta X,\delta Y, \delta Z)= ( X-\bar{X}, Y-\bar{Y}, Z-\bar{Z})$,
$(\delta \phi, \delta \psi,\delta \gamma, \delta \eta, \delta \xi_0)= (\phi-\bar{\phi},\psi-\bar{\psi},\gamma-\bar{\gamma},\eta-\bar{\eta}, \xi_0-\bar{\xi}_0)$.
Then it holds for all 
 $\eps_1,\eps_2,\eps_3>0$
  that
\begin{align*}
\begin{split}
&\min\{1,\a\} \sE[
|G\delta X_{N}|^2]
+
 \sum_{i=0}^{N-1}
  \sE[
\beta_2|G\delta X_{i}|^2
+
\beta_1(
|G^*\delta Y_i|^2+|G^*\delta Z_{i}|^2
) ]\tau_N
\\
& \le
\eps_1\sE[ |G^*\delta Y_0|^2]
+\frac{1}{4\eps_1}\sE[ | \delta \xi_0|^2]
+
\eps_2
\sE[ |G\delta X_N|^2]
+
\frac{1}{4\eps_2}\sE[ |\lambda_0(g(\bar{X}_N,\sP_{\bar{X}_N})-\bar{g}(\bar{X}_N,\sP_{\bar{X}_N}))
+\delta \eta|^2]
\\
&\quad
 +\sum_{i=0}^{N-1}
 \sE\bigg[
\eps_2 | G \delta X_i|^2+\eps_3 (|G^*\delta Y_{i}|^2+| G^*\delta  {Z}_i|^2)
+
\bigg(
\frac{1}{4\eps_2}
+\tau_N
\bigg)
|
\lambda _0
(\hat{f}(t_{i})-\bar{f}(t_{i}))+\delta \gamma_i|^2
\\
&\quad
+
\bigg(
\frac{1}{4\eps_3}
+\tau_N\|G\|^2 \bigg)
 |
 \lambda_0(
 \hat{b}(t_{i})-\bar{b}(t_{i}))+\delta \phi_{i}|^2
+
\frac{1}{4\eps_3}
 | 
 \lambda_0( \hat{\sigma}(t_i)-\bar{\sigma}(t_i))+\delta \psi_{i}|^2
 \bigg]\tau_N
 \\
 &\quad
 +C_{(G,\b_1,\b_2, L)}\sum_{i=0}^{N-1}
 \sE[|  \delta X_i|^2+|\delta Y_{i}|^2+| \delta  {Z}_i|^2]\tau^2_N,
\end{split}
\end{align*}
  where 
 $C_{(G,\b_1,\b_2, L)}$ is  a constant depending  only on $G,\b_1,\b_2, L$,
 and
for each $i\in \cN$, $\phi=b, \sigma, f$,
we define
 $\hat{\phi}(t_i)\coloneqq \phi(t_i, \bar{X}_i, \bar{Y}_i, \bar{Z}_i,\sP_{(\bar{X}_i, \bar{Y}_i, \bar{Z}_i)})$
 and
$\bar{\phi}(t_i)\coloneqq \bar{\phi}(t_i, \bar{X}_i, \bar{Y}_i, \bar{Z}_i,\sP_{(\bar{X}_i, \bar{Y}_i, \bar{Z}_i)})$.

\end{Lemma}

\begin{proof}[Proof of Lemma \ref{lemma:stab_a_priori_exp}]
Throughout this proof, 
for each $\lambda\in [0,1]$ and $(t,x,y,z,\mu,\nu)\in [0,1]\t \sR^n\t\sR^m\t\sR^{m\t d}\t\cP_2(\sR^{n+m+m d})\t \cP_2(\sR^n)$ 
let 
$b^\lambda(t,x,y,z,\mu)=
(1-\lambda)\b_1 (-G^*y)+\lambda b(t,x,y,z,\mu)$,
$\sigma^\lambda(t,x,y,z,\mu)=
(1-\lambda)\b_1 (-G^*z)+\lambda \sigma(t,x,y,z,\mu)$,
$f^\lambda(t,x,y,z,\mu)=
(1-\lambda)\b_2 Gx+\lambda f(t,x,y,z,\mu)$
and $g^\lambda(x,\nu)=(1-\lambda)Gx+\lambda g(x,\nu)$.
Let 
$\Theta=(X, Y, Z)$,  
$\bar{\Theta}=(\bar{X}, \bar{Y}, \bar{Z})$,
$(\delta \Theta, \delta X,\delta Y, \delta Z, \delta M)= ( \Theta-\bar{ \Theta}, X-\bar{X}, Y-\bar{Y}, Z-\bar{Z}, M-\bar{M})$,
for each $t\in [0,T]$, 
$h=b, \sigma, f, b^\lambda, \sigma^\lambda, f^\lambda$
let 
${h}(t)= h(t, {\Theta}_t,\sP_{{\Theta}_t})$,
$\hat{h}(t)= h(t, \bar{\Theta}_t,\sP_{\bar{\Theta}_t})$,
$\delta {h}(t)=h(t)- \hat{h}(t)$
and $\bar{h}(t)= \bar{h}(t, \bar{\Theta}_t,\sP_{\bar{\Theta}_t})$.

Note that
we can deduce from  \eqref{eq:MV-fbsde_exp_lambda} that
$\delta X_0=\delta \xi_0$,
$\delta Y_N=g^{\lambda_0}(X_N,\sP_{X_N})-\bar{g}^{\lambda_0}(\bar{X}_N,\sP_{\bar{X}_N})
+\delta \eta$,
and 
for any given $i\in \cN_{<N}$ that
\begin{align}
\Delta (\delta {X})_{i}
&=
(b^{\lambda_0}(t_{i})-\bar{b}^{\lam_0}(t_{i})+\delta \phi_{i})\tau_N  +
(\sigma^{\lambda_0}(t_i)-\bar{\sigma}^{\lambda_0}(t_i)+\delta \psi_{i})\,\Delta W_i
\l{eq:delta_X_lambda}
\\
\Delta (\delta {Y})_{i}
&=
-[f^{\lambda_0}(t_{i})-\bar{f}^{\lambda_0}(t_{i})+\delta \gamma_i]\tau_N+\delta  {Z}_i\,\Delta W_i+\Delta (\delta {M})_i,
\l{eq:delta_Y_lambda}
\end{align}
which together with 
the definition of the backward operator $\Delta$ 
shows for all $i\in\cN_{<N}$
  that
\begin{align*}
\begin{split}
\Delta \la G\delta X,\delta Y\ra_i
&=  \la G\Delta (\delta X)_{i}, (\delta Y)_{i}\ra+\la  G (\delta X)_i, \Delta(\delta Y)_i\ra
+\la  G \Delta(\delta X)_i, \Delta(\delta Y)_i\ra
\\
&=
\la G[(b^{\lambda_0}(t_{i})-\bar{b}^{\lam_0}(t_{i})+\delta \phi_{i})\tau_N+(\sigma^{\lambda_0}(t_i)-\bar{\sigma}^{\lambda_0}(t_i)+\delta \psi_{i})\,\Delta W_i], (\delta Y)_{i}\ra
\\
& \quad
+\la  G (\delta X)_i, -[f^{\lambda_0}(t_{i})-\bar{f}^{\lambda_0}(t_{i})+\delta \gamma_i]\tau_N+\delta  {Z}_i\,\Delta W_i+\Delta (\delta {M})_i \ra
\\
&\quad
+
\la  G [(b^{\lambda_0}(t_{i})-\bar{b}^{\lam_0}(t_{i})+\delta \phi_{i})\tau_N],
-[f^{\lambda_0}(t_{i})-\bar{f}^{\lambda_0}(t_{i})+\delta \gamma_i]\tau_N
\ra
\\
&\quad
+
\la  G [(b^{\lambda_0}(t_{i})-\bar{b}^{\lam_0}(t_{i})+\delta \phi_{i})\tau_N], 
\delta  {Z}_i\,\Delta W_i+\Delta (\delta {M})_i
\ra
\\
&\quad
+
\la  G [(\sigma^{\lambda_0}(t_i)-\bar{\sigma}^{\lambda_0}(t_i)+\delta \psi_{i})\,\Delta W_i], 
-[f^{\lambda_0}(t_{i})-\bar{f}^{\lambda_0}(t_{i})+\delta \gamma_i]\tau_N+\Delta (\delta {M})_i
\ra
\\
&\quad
+
\la  G [(\sigma^{\lambda_0}(t_i)-\bar{\sigma}^{\lambda_0}(t_i)+\delta \psi_{i})\,\Delta W_i], 
\delta  {Z}_i\,\Delta W_i
\ra.
 \end{split}
  \end{align*}
  Then,   adding and subtracting the terms $\hat{b}^{\lambda_0}(t_{i})$, $\hat{\sigma}^{\lambda_0}(t_{i})$ and $\hat{f}^{\lambda_0}(t_{i})$ imply 
 for all $i\in \cN_{<N}$ that 
\begin{align*}
\begin{split}
&\Delta \la G\delta X,\delta Y\ra_i
\\
&=
\la G[(\delta b^{\lambda_0}(t_{i})+\hat{b}^{\lambda_0}(t_{i})-\bar{b}^{\lam_0}(t_{i})+\delta \phi_{i})\tau_N+(\sigma^{\lambda_0}(t_i)-\bar{\sigma}^{\lambda_0}(t_i)+\delta \psi_{i})\,\Delta W_i], (\delta Y)_{i}\ra
\\
& \quad
+\la  G (\delta X)_i, -[\delta f^{\lambda_0}(t_{i})+\hat{f}^{\lambda_0}(t_{i})-\bar{f}^{\lambda_0}(t_{i})+\delta \gamma_i]\tau_N+\delta  {Z}_i\,\Delta W_i+\Delta (\delta {M})_i \ra
\\
&\quad
+
\la  G [(\delta b^{\lambda_0}(t_{i})+\hat{b}^{\lambda_0}(t_{i})-\bar{b}^{\lam_0}(t_{i})+\delta \phi_{i})\tau_N],
-[\delta f^{\lambda_0}(t_{i})+\hat{f}^{\lambda_0}(t_{i})-\bar{f}^{\lambda_0}(t_{i})+\delta \gamma_i]\tau_N
\ra
\\
&\quad
+
\la  G [(b^{\lambda_0}(t_{i})-\bar{b}^{\lam_0}(t_{i})+\delta \phi_{i})\tau_N], 
\delta  {Z}_i\,\Delta W_i+\Delta (\delta {M})_i
\ra
\\
&\quad
+
\la  G [(\sigma^{\lambda_0}(t_i)-\bar{\sigma}^{\lambda_0}(t_i)+\delta \psi_{i})\,\Delta W_i], 
-[f^{\lambda_0}(t_{i})-\bar{f}^{\lambda_0}(t_{i})+\delta \gamma_i]\tau_N+\Delta (\delta {M})_i
\ra
\\
&\quad
+
\la  G [(\delta\sigma^{\lambda_0}(t_i)+\hat{\sigma}^{\lambda_0}(t_{i})-\bar{\sigma}^{\lambda_0}(t_i)+\delta \psi_{i})\,\Delta W_i], 
\delta  {Z}_i\,\Delta W_i
\ra.
 \end{split}
  \end{align*}
By further introducing   the following residual terms $\Phi_i$,
 $\Psi_i$ and $\Sigma_i$:
  \begin{align*}
\Phi_i
&=
\la G[(\sigma^{\lambda_0}(t_i)-\bar{\sigma}^{\lambda_0}(t_i)+\delta \psi_{i})\,\Delta W_i], (\delta Y)_{i}\ra
%
+\la  G (\delta X)_i, \delta  {Z}_i\,\Delta W_i+\Delta (\delta {M})_i \ra
\\
&\quad
+
\la  G [(b^{\lambda_0}(t_{i})-\bar{b}^{\lam_0}(t_{i})+\delta \phi_{i})\tau_N], 
\delta  {Z}_i\,\Delta W_i+\Delta (\delta {M})_i
\ra
\\
&\quad
+
\la  G [(\sigma^{\lambda_0}(t_i)-\bar{\sigma}^{\lambda_0}(t_i)+\delta \psi_{i})\,\Delta W_i], 
-[f^{\lambda_0}(t_{i})-\bar{f}^{\lambda_0}(t_{i})+\delta \gamma_i]\tau_N+\Delta (\delta {M})_i
\ra,
\\
\Psi_i
&=
\la G[(\hat{b}^{\lam_0}(t_{i})-\bar{b}^{\lam_0}(t_{i})+\delta \phi_{i})\tau_N], (\delta Y)_{i}\ra
+\la  G (\delta X)_i, -[\hat{f}^{\lambda_0}(t_{i})-\bar{f}^{\lambda_0}(t_{i})+\delta \gamma_i]\tau_N \ra
\\
& \quad
+ \la G[(\hat{\sigma}^{\lambda_0}(t_i)-\bar{\sigma}^{\lambda_0}(t_i)+\delta \psi_{i})\,\Delta W_i], 
\delta  {Z}_i\,\Delta W_i\ra,  
\\
\Sigma_i
&=
\la  G [(\delta b^{\lambda_0}(t_{i})+\hat{b}^{\lambda_0}(t_{i})-\bar{b}^{\lam_0}(t_{i})+\delta \phi_{i})\tau_N],
-[\delta f^{\lambda_0}(t_{i})+\hat{f}^{\lambda_0}(t_{i})-\bar{f}^{\lambda_0}(t_{i})+\delta \gamma_i]\tau_N
\ra,
 \end{align*}
we have for all $i\in \cN_{<N}$ that
\begin{align*}
\begin{split}
\Delta \la G\delta X,\delta Y\ra_i
&=
\la \delta b^{\lambda_0}(t_{i})\tau_N, G^*(\delta Y)_{i}\ra
+\la  G (\delta X)_i, -\delta f^{\lambda_0}(t_{i})\tau_N \ra
\\
&\quad
+
\la  G \delta\sigma^{\lambda_0}(t_i)\,\Delta W_i, 
\delta  {Z}_i\,\Delta W_i
\ra
+\Phi_i+\Psi_i+\Sigma_i.
 \end{split}
  \end{align*}
  
  We then compute  $\sE[\Delta \la G\delta X,\delta Y\ra_i]$.
By using the  definition of the backward operator $\Delta$,
and
the fact that 
  $M$, $\bar{M}$ are strongly orthogonal to  $W$,
we see 
$\sE_i[ \Delta (\delta M)_i(\Delta W_i)^*]=0$ for all $i\in \cN_{<N}$.
Thus, we can  deduce from the adaptedness of the coefficients
and the law of iterated expectations
 that 
$\sE[\Phi_i]=0$.
Moreover, the property that $\sE_i[\Delta W_i (\Delta W_i)^*]=\tau_N\sI_d$ 
implies that
$$
\sE_i[\la G\delta \sigma^{\lambda_0}(t_i)\Delta W_i,\delta Z_i\Delta W_i\ra]
=\sE_i[\tr(\Delta W_i(\Delta W_i)^*(G\delta \sigma^{\lambda_0}(t_i))^*\delta Z_i)]
=\tau_N\sE_i[\la \delta \sigma^{\lambda_0}(t_i),G^*\delta Z_i)\ra ],
$$
where we have used the fact that  the trace
commutes with conditional expectations.
Consequently, 
for each $i\in \cN_{<N}$, 
we have that
\begin{align*}
\sE[\Delta \la G\delta X,\delta Y\ra_i]
&=\tau_N
\sE[
\la \delta \sigma^{\lambda_0}(t_i),G^*\delta Z_i\ra
 +\la  \delta b^{\lambda_0}(t_{i}),G^*\delta Y_{i} \ra
+ \la -\delta f^{\lambda_0}(t_{i}),G\delta X_{i}\ra
 ]
\\
&\quad 
 +\sE[\Psi_i]+\sE[\Sigma_i],
\end{align*}
where we can deduce from the definitions of $(\hat{b}^{\lambda_0},\hat{\sigma}^{\lambda_0},\hat{f}^{\lambda_0},\hat{g}^{\lambda_0})$ 
  and $(\bar{b}^{\lambda_0},\bar{\sigma}^{\lambda_0},\bar{f}^{\lambda_0},\bar{g}^{\lambda_0})$
  that
 \begin{align}
\sE[\Psi_i]
&=
\tau_N\sE\big[
\la \lam_0(\hat{b}(t_{i})-\bar{b}(t_{i}))+\delta \phi_{i}, G^*\delta Y_{i}\ra
+\la  G (\delta X)_i, -[{\lambda_0}(\hat{f}(t_{i})-\bar{f}(t_{i}))+\delta \gamma_i] \ra
\nb
\\
& \quad
+ \la {\lambda_0}(\hat{\sigma}(t_i)-\bar{\sigma}(t_i))+\delta \psi_{i}, 
G^*\delta  {Z}_i\ra
\big],  
 \l{eq:exp_Psi}
\\
\sE[\Sigma_i]
&=
\tau_N^2
\sE\big[
\la  G [\delta b^{\lambda_0}(t_{i})+{\lambda_0}(\hat{b}(t_{i})-\bar{b}(t_{i}))+\delta \phi_{i}],
-[\delta f^{\lambda_0}(t_{i})+{\lambda_0}(\hat{f}(t_{i})-\bar{f}(t_{i}))+\delta \gamma_i]
\ra
\big].
 \l{eq:exp_Sigma}
 \end{align}

Note that  for each $\lambda\in [0,1)$, the coefficients
  $(b^\lambda,\sigma^\lambda, -f^\lam)$ 
  satisfy (H.\ref{assum:fbsde_discrete_exp}\ref{item:monotone_exp}) 
  with the same $G,\b_1,\b_2$.
Hence, by summing the above identity over the index $i$ from $0$ to $N-1$
and  applying the monotonicity condition (H.\ref{assum:fbsde_discrete_exp}\ref{item:monotone_exp}),
we have that
\begin{align*}
&\sE[ \la G\delta X_N, \delta Y_N\ra]
-\sE[ \la G\delta X_0, \delta Y_0\ra]
 \\
 & =
\sum_{i=0}^{N-1}
\tau_N
\sE[
\la -\delta f^{\lambda_0}(t_{i}),G\delta X_{i}\ra
 +\la \delta \sigma^{\lambda_0}(t_i), G^*\delta Z_i\ra
 +\la \delta b^{\lambda_0}(t_{i}),G^*\delta Y_{i} \ra
 ]
 +\sE[\Psi_i]
  +\sE[\Sigma_i]
 \\
  &\le
 \sum_{i=0}^{N-1}
  \tau_N
\sE[
-\beta_2|G\delta X_{i}|^2
-\beta_1(
|G^*\delta Z_i|^2+|G^*\delta Y_{i}|^2
)
 ]
 +\sum_{i=0}^{N-1}(\sE[\Psi_i]+\sE[\Sigma_i]).
\end{align*}
 Then by rearranging the terms
and using the fact that $(\delta X)_0=\delta \xi_0$,
\begin{align*}
\delta Y_N
&=g^{\lambda_0}(X_N,\sP_{X_N})-g^{\lambda_0}(\bar{X}_N,\sP_{\bar{X}_N})
+g^{\lambda_0}(\bar{X}_N,\sP_{\bar{X}_N})-\bar{g}^{\lambda_0}(\bar{X}_N,\sP_{\bar{X}_N})
+\delta \eta,
\end{align*}
 and the monotonicity of $g$,
we arrive at the estimate that
\begin{align*}
&
(1-\lambda_0+\lambda_0\a) \sE[
|G\delta X_{N}|^2]
+
 \sum_{i=0}^{N-1}
  \sE[
\beta_2|G\delta X_{i}|^2
+
\beta_1(
|G^*\delta Y_i|^2+|G^*\delta Z_{i}|^2
) ]\tau_N
 \\
& \le
\sE[ \la G\delta \xi_0, \delta Y_0\ra]
-\sE[ \la G\delta X_N, g^{\lambda_0}(\bar{X}_N,\sP_{\bar{X}_N})-\bar{g}^{\lambda_0}(\bar{X}_N,\sP_{\bar{X}_N})
+\delta \eta\ra]
 +\sum_{i=0}^{N-1}(\sE[\Psi_i]+\sE[\Sigma_i]),
\end{align*}
which together with
the fact that  
$(1-\lambda_0+\lambda_0\a)\ge \min\{1,\a\}$,
\eqref{eq:exp_Psi} and
 Young's inequality implies for all $\eps_1,\eps_2,\eps_3>0$ that 
\begin{align*}
\begin{split}
&\min\{1,\a\} \sE[
|G\delta X_{N}|^2]
+
 \sum_{i=0}^{N-1}
  \sE[
\beta_2|G\delta X_{i}|^2
+
\beta_1(
|G^*\delta Z_i|^2+|G^*\delta Y_{i}|^2
) ]\tau_N
\\
& \le
\eps_1\sE[ |G^*\delta Y_0|^2]
+\frac{1}{4\eps_1}\sE[ | \delta \xi_0|^2]
+
\eps_2
\sE[ |G\delta X_N|^2]
+
\frac{1}{4\eps_2}\sE[ |\lambda_0(g(\bar{X}_N,\sP_{\bar{X}_N})-\bar{g}(\bar{X}_N,\sP_{\bar{X}_N}))
+\delta \eta|^2]
\\
&\quad
 +\sum_{i=0}^{N-1}
 \sE\bigg[
\eps_2 | G \delta X_i|^2+\eps_3 (|G^*\delta Y_{i}|^2+| G^*\delta  {Z}_i|^2)
+\frac{1}{4\eps_2}
|
\lambda _0
(\hat{f}(t_{i})-\bar{f}(t_{i}))+\delta \gamma_i|^2
\\
& \quad
+
\frac{1}{4\eps_3}
\bigg(
 |
 \lambda_0(
 \hat{b}(t_{i})-\bar{b}(t_{i}))+\delta \phi_{i}|^2
 +
 | 
 \lambda_0( \hat{\sigma}(t_i)-\bar{\sigma}(t_i))+\delta \psi_{i}|^2
 \bigg) \bigg]\tau_N
 +\sum_{i=0}^{N-1} \sE[\Sigma_i].
\end{split}
\end{align*}
Finally, it remains to estimate $\sum_{i=0}^{N-1} \sE[\Sigma_i]$.
We obtain from
\eqref{eq:exp_Sigma} and
Young's inequality
 that
\begin{align*}
\sE[\Sigma_i]
&=
\tau_N^2
\sE\big[
\la  G [\delta b^{\lambda_0}(t_{i})+{\lambda_0}(\hat{b}(t_{i})-\bar{b}(t_{i}))+\delta \phi_{i}],
-[\delta f^{\lambda_0}(t_{i})+{\lambda_0}(\hat{f}(t_{i})-\bar{f}(t_{i}))+\delta \gamma_i]
\ra
\big]
\\
&\le
\tau_N^2
\bigg(
\sE\big[
\|G\|^2|\delta b^{\lambda_0}(t_{i})|^2
 +\|G\|^2|{\lambda_0}(\hat{b}(t_{i})-\bar{b}(t_{i}))+\delta \phi_{i}|^2
 +|\delta f^{\lambda_0}(t_{i})|^2
 \\
&\quad
 +|{\lambda_0}(\hat{f}(t_{i})-\bar{f}(t_{i}))+\delta \gamma_i|^2
\big]
\bigg),
\end{align*}
where $\|G\|$ denotes the spectral norm of $G$.
Moreover, we can deduce from the definitions of $(\delta{b}^{\lambda_0},\delta{f}^{\lambda_0})$ and 
$(\hat{b},\hat{f})$ and also
the assumption (H.\ref{assum:fbsde_discrete_exp}\ref{item:lipschitz_exp}) that
\begin{align*}
\sE[|\delta b^{\lambda_0}(t_{i})|^2]
&=\sE[|(1-\lambda_0)\b_1(-G^*\delta Y_i)+\lambda_0 (b(t_i)-\hat{b}(t_i))|^2]
\\
&\le 
2(1-\lambda_0)^2\b^2_1\|G^*\|^2\sE[\|\delta Y_i|^2]+8\lambda^2_0 L^2\sE[|\delta \Theta_i|^2],
\\
\sE[|\delta f^{\lambda_0}(t_{i})|^2]
&\le 
2(1-\lambda_0)^2\b^2_2\|G\|^2\sE[\|\delta X_i|^2]+8\lambda^2_0 L^2\sE[|\delta \Theta_i|^2],
\end{align*}
which
together with the fact that $\lam_0\in [0,1]$
gives us that 
$\sE[
\|G\|^2|\delta b^{\lambda_0}(t_{i})|^2+|\delta f^{\lambda_0}(t_{i})|^2]
\le C\sE[|\delta \Theta_i|^2]$
for a constant $C$ depending on $G$, $\b_1,\b_2, L$.
This finishes the proof of Lemma \ref{lemma:stab_a_priori_exp}.
\end{proof}

With Lemma \ref{lemma:stab_a_priori_exp} at hand, we now establish Proposition \ref{thm:stab_exp}
by separately discussing the following two cases: 
(1) $m< n$,  or $m=n$ with $\beta_1>0$;
 (2)
$m> n$, or $m=n$ with $\a>0,\beta_2>0$.

\begin{proof}[Proof of Proposition \ref{thm:stab_exp}]
Throughout this proof, 
let $N_0\in \sN$ be a sufficiently large natural number whose value will be specified later.
For each $\lambda\in [0,1]$ and $(t,x,y,z,\mu,\nu)\in [0,1]\t \sR^n\t\sR^m\t\sR^{m\t d}\t\cP_2(\sR^{n+m+m d})\t \cP_2(\sR^n)$ 
let 
$b^\lambda(t,x,y,z,\mu)=
(1-\lambda)\b_1 (-G^*y)+\lambda b(t,x,y,z,\mu)$,
$\sigma^\lambda(t,x,y,z,\mu)=
(1-\lambda)\b_1 (-G^*z)+\lambda \sigma(t,x,y,z,\mu)$,
$f^\lambda(t,x,y,z,\mu)=
(1-\lambda)\b_2 Gx+\lambda f(t,x,y,z,\mu)$
and $g^\lambda(x,\nu)=(1-\lambda)Gx+\lambda g(x,\nu)$.
Let 
$N\in \sN\cap [N_0,\infty)$, $\lambda_0\in [0,1]$,
$\Theta=(X, Y, Z)$,  
$\bar{\Theta}=(\bar{X}, \bar{Y}, \bar{Z})$,
$(\delta \Theta, \delta X,\delta Y, \delta Z, \delta M)= ( \Theta-\bar{ \Theta}, X-\bar{X}, Y-\bar{Y}, Z-\bar{Z}, M-\bar{M})$,
$(\delta \phi, \delta \psi,\delta \gamma, \delta \eta, \delta \xi_0)= (\phi-\bar{\phi},\psi-\bar{\psi},\gamma-\bar{\gamma},\eta-\bar{\eta}, \xi_0-\bar{\xi}_0)$,
for each $t\in [0,T]$, 
$h=b, \sigma, f, b^\lambda, \sigma^\lambda, f^\lambda$
let 
${h}(t)= h(t, {\Theta}_t,\sP_{{\Theta}_t})$,
$\hat{h}(t)= h(t, \bar{\Theta}_t,\sP_{\bar{\Theta}_t})$,
$\delta {h}(t)=h(t)- \hat{h}(t)$
and $\bar{h}(t)= \bar{h}(t, \bar{\Theta}_t,\sP_{\bar{\Theta}_t})$.
We   denote by $C$ a generic constant, 
which
depends only on   constants in (H.\ref{assum:fbsde_discrete_exp}) 
and  may take a different value at each occurrence.

We start by deriving several {a priori} estimates based on \eqref{eq:delta_X_lambda} and \eqref{eq:delta_Y_lambda}.
Note that for any given  $i\in \cN_{<N}$, 
by using  \eqref{eq:delta_X_lambda}, the Cauchy-Schwarz inequality
and 
the It\^{o} isometry,
\begin{align*}
\sE[|\delta X_{i+1}|^2]
&\le 3
\bigg(
\sE[|\delta X_{0}|^2]
+T\sum_{j=0}^{i}\sE[|b^{\lambda_0}(t_{j})-\bar{b}^{\lam_0}(t_{j})+\delta \phi_{j}|^2]\tau_N
\\
&\quad
+\sum_{j=0}^{i}\sE[|\sigma^{\lambda_0}(t_j)-\bar{\sigma}^{\lambda_0}(t_j)+\delta \psi_{j}|^2]\tau_N
\bigg)
\\
&\le 3
\bigg(
\sE[|\delta X_{0}|^2]
+2T\sum_{j=0}^{i}\sE[|\delta b^{\lambda_0}(t_{j})|^2+|\hat{b}^{\lambda_0}(t_{j})-\bar{b}^{\lam_0}(t_{j})+\delta \phi_{j}|^2]\tau_N
\\
&\quad
+2\sum_{j=0}^{i}\sE[|\delta \sigma^{\lambda_0}(t_j)|^2+|\hat{\sigma}^{\lambda_0}(t_j)-\bar{\sigma}^{\lambda_0}(t_j)+\delta \psi_{j}|^2]\tau_N
\bigg).
\end{align*}
Then by using the Lipschitz continuity of $b$ and $\sigma$,
the inequality that $\cW^2_2(\sP_{U},\sP_{U'})\le \sE[|U-U'|^2]$
for any $U,U'\in L^2(\cF;\sR^n\t \sR^m\t \sR^{m\t d})$,
 Gronwall's inequality {
 in Lemma \ref{lemma:gronwall}}
and the fact that $\delta X_0=\delta \xi_0$, 
 it holds for all $N\in \sN$
and  all $i\in \cN_{<N}$
that
\begin{align}\l{eq:deltaX_stab}
\begin{split}
\sup_{i\in \cN}\sE[|\delta X_{i}|^2]
&\le C
\bigg(
\sE[|\delta \xi_{0}|^2]
+
\sum_{i=0}^{N-1}\sE[|\delta{Y}_{i}|^2+|\delta{Z}_{i}|^2
]\tau_N
\\
&\quad
+\sum_{i=0}^{N-1}\sE[|\lambda_0(\hat{b}(t_{i})-\bar{b}(t_{i}))+\delta \phi_{i}|^2
+|\lambda_0(\hat{\sigma}(t_i)-\bar{\sigma}(t_i))+\delta \psi_{i}|^2]\tau_N
\bigg).
\end{split}
\end{align}

On the other hand, 
for each $i\in \cN_{<N}$, 
we can obtain from   \eqref{eq:delta_Y_lambda}
that
\begin{align*}
&\Delta \la \delta Y,\delta Y \ra_i
= \la \delta Y_{i+1},\Delta (\delta Y)_i \ra+  \la \Delta (\delta Y)_i , \delta Y_{i}\ra
\\
&=
 \la \delta Y_{i+1}+\delta Y_{i}, -[f^{\lambda_0}(t_{i})-\bar{f}^{\lambda_0}(t_{i})+\delta \gamma_i]\tau_N+\delta  {Z}_i\,\Delta W_i+\Delta (\delta {M})_i\ra
 \\
 &= \la 2\delta Y_{i}-[f^{\lambda_0}(t_{i})-\bar{f}^{\lambda_0}(t_{i})+\delta \gamma_i]\tau_N, -[f^{\lambda_0}(t_{i})-\bar{f}^{\lambda_0}(t_{i})+\delta \gamma_i]\tau_N \ra
 +|  \delta  {Z}_i\,\Delta W_i|^2
+|\Delta (\delta {M})_i|^2
\\
&\quad +
\la 2\delta Y_{i}-[f^{\lambda_0}(t_{i})-\bar{f}^{\lambda_0}(t_{i})+\delta \gamma_i]\tau_N, \delta  {Z}_i\,\Delta W_i+\Delta (\delta {M})_i \ra
\\
&\quad
+
\la \delta  {Z}_i\,\Delta W_i+\Delta (\delta {M})_i, -[f^{\lambda_0}(t_{i})-\bar{f}^{\lambda_0}(t_{i})+\delta \gamma_i]\tau_N \ra
+2\la \Delta (\delta {M})_i,\delta  {Z}_i\,\Delta W_i\ra. 
\end{align*}
Taking the expectation and using the orthogonality  between   martingales $\delta M$ and  $W$
yield 
\begin{align*}
&\sE[|\delta Y_{i+1}|^2-|\delta Y_{i}|^2]=\sE[\Delta \la \delta Y,\delta Y \ra_i]
\\
&= 
\sE[   \la 2\delta Y_{i}-(f^{\lambda_0}(t_{i})-\bar{f}^{\lambda_0}(t_{i})+\delta \gamma_i)\tau_N, -(f^{\lambda_0}(t_{i})-\bar{f}^{\lambda_0}(t_{i})+\delta \gamma_i)\tau_N \ra
]
\\
&\quad
 +\sE[| \delta {Z}_i|^2\tau_N]+\sE[|\Delta  (\delta {M})_i |^2].
\end{align*}
Rearranging the terms and summing over the index imply 
for all $i\in \cN_{<N}$ that 
\begin{align}\l{eq:delta Y_i_estimate}
\begin{split}
&\sE[|\delta Y_{i}|^2]+\sum_{j=i}^{N-1}\big(\sE[| \delta {Z}_j|^2\tau_N+|\Delta  (\delta {M})_j|^2]\big)
\\
&=
\sE[|\delta Y_{N}|^2]
+\sum_{j=i}^{N-1} 
 \sE[    \la 2\delta Y_{j}-(f^{\lambda_0}(t_{j})-\bar{f}^{\lambda_0}(t_{j})+\delta \gamma_j)\tau_N, (f^{\lambda_0}(t_{j})-\bar{f}^{\lambda_0}(t_{j})+\delta \gamma_j)\tau_N \ra ].
 \end{split}
\end{align}
We now derive an upper bound of the two terms on the right-hand side of \eqref{eq:delta Y_i_estimate}  separately.
One can see easily  from the Lipschitz continuity of $g$ 
that 
\begin{align}\l{eq:delta Y_N_estimate}
\begin{split}
&\sE[|\delta {Y}_{N}|^2]
=\sE[|g^{\lambda_0}(X_N,\sP_{X_N})-g^{\lambda_0}(\bar{X}_N,\sP_{\bar{X}_N})
+g^{\lambda_0}(\bar{X}_N,\sP_{\bar{X}_N})-\bar{g}^{\lambda_0}(\bar{X}_N,\sP_{\bar{X}_N})
+\delta \eta|^2]
\\
&\le C(
\sE[|{X}_{N}-\bar{X}_{N}|^2]+\sE[|g^{\lambda_0}(\bar{X}_N,\sP_{\bar{X}_N})-\bar{g}^{\lambda_0}(\bar{X}_N,\sP_{\bar{X}_N})
+\delta \eta
|^2]
)
\\
&=
 C(
\sE[|\delta {X}_{N}|^2]+\sE[|\lambda_0(g(\bar{X}_N,\sP_{\bar{X}_N})-\bar{g}(\bar{X}_N,\sP_{\bar{X}_N}))
+\delta \eta
|^2]
).
\end{split}
\end{align}
Moreover, 
using the Lipschitz continuity of $f$ 
and the fact that $\lambda_0\in [0,1]$
shows 
 for all $i\in \cN_{<N}$,
\begin{align*}
\sE[|f^{\lambda_0}(t_{i})-\bar{f}^{\lambda_0}(t_{i})+\delta \gamma_i|^2]
&\le
2(\sE[ |\delta f^{\lambda_0}(t_{i})|^2]
+\sE[|\hat{f}^{\lambda_0}(t_{i})-\bar{f}^{\lambda_0}(t_{i})+\delta \gamma_i|^2])
\\
&\le 
C
(\sE[|\delta \Theta_{i}|^2]
+\sE[|\hat{f}^{\lambda_0}(t_{i})-\bar{f}^{\lambda_0}(t_{i})+\delta \gamma_i|^2])
\\
&=
C
(\sE[|\delta \Theta_{i}|^2]
+\sE[|\lambda_0(\hat{f}(t_{i})-\bar{f}(t_{i}))+\delta \gamma_i|^2]),
\end{align*}
which, along with Young's inequality, implies for all $\eps>0$ that
\begin{align*}
&\sum_{j=i}^{N-1} 
 \sE[    \la 2\delta Y_{j}-(f^{\lambda_0}(t_{j})-\bar{f}^{\lambda_0}(t_{j})+\delta \gamma_j)\tau_N, 
 (f^{\lambda_0}(t_{j})-\bar{f}^{\lambda_0}(t_{j})+\delta \gamma_j)\tau_N \ra ]
 \\ 
 &\le \sum_{j=i}^{N-1} \bigg(
 \tfrac{1}{\eps}\sE[|\delta Y_{j}|^2]\tau_N+
 \eps\sE[|f^{\lambda_0}(t_{j})-\bar{f}^{\lambda_0}(t_{j})+\delta \gamma_j|^2]\tau_N
 +\sE[|f^{\lambda_0}(t_{j})-\bar{f}^{\lambda_0}(t_{j})+\delta \gamma_j|^2]\tau_N^2
 \bigg)
\\
 &\le \sum_{j=i}^{N-1} \bigg(
 \tfrac{1}{\eps}\sE[|\delta Y_{j}|^2]\tau_N+
 C(\eps+\tau_N)
 \big(
 \sE[|\delta \Theta_{j}|^2]\tau_N
+\sE[|\lambda_0(\hat{f}(t_{j})-\bar{f}(t_{j}))+\delta \gamma_j|^2]\tau_N
 \big)
 \bigg)
 \\
  &\le 
\big( \tfrac{1}{\eps}+C(\eps+\tau_N)\big)
\bigg( \sE[|\delta Y_{i}|^2]\tau_N+\sum_{j=i+1}^{N-1}  \sE[|\delta Y_{j}|^2]\tau_N\bigg)+
C(\eps+\tau_N)
\sum_{j=i}^{N-1}\sE[|\delta Z_{j}|^2]\tau_N
\\
&\quad +
C(\eps+\tau_N)\sum_{j=i}^{N-1} 
\big(
 \sE[|\delta X_{j}|^2]\tau_N
+\sE[|\lambda_0(\hat{f}(t_{j})-\bar{f}(t_{j}))+\delta \gamma_j|^2]\tau_N
 \big).
\end{align*}
Then by   choosing a sufficiently small  $\eps>0$,
we see from \eqref{eq:delta Y_i_estimate} that,  there exists $K_1\in \sN$, 
depending only on $T$ and $L$, such that  for all $N\in \sN\cap[K_1,\infty)$
and  all $i\in \cN_{<N}$,
\begin{align*}
\begin{split}
&\sE[|\delta Y_{i}|^2]+\sum_{j=i}^{N-1}\sE[| \delta {Z}_j|^2\tau_N+|\Delta  (\delta {M})_j|^2]
\\
&\le
\sE[|\delta Y_{N}|^2]
+
 C\bigg(
\sum_{j=i+1}^{N-1}  \sE[|\delta Y_{j}|^2]\tau_N+
\sum_{j=i}^{N-1} 
\big(
 \sE[|\delta X_{j}|^2]\tau_N
+\sE[|\lambda_0(\hat{f}(t_{j})-\bar{f}(t_{j}))+\delta \gamma_j|^2]\tau_N
 \big)
 \bigg).
   \end{split}
\end{align*}
Then a direct application of Gronwall's inequality {  in Lemma \ref{lemma:gronwall}},
 the estimate \eqref{eq:delta Y_N_estimate}
 and the fact that $\delta M$ is a martingale with $\delta M_0=0$
 shows that
\begin{align}\l{eq:deltaYZ_stab}
\begin{split}
&\max_{i\in \cN}\sE[|\delta Y_i|^2]+\sum_{i=0}^{N-1}\sE[| \delta {Z}_i|^2]\tau_N+\sE[|\delta {M}_N|^2]
\\
&\le
 C\bigg(
 \sE[|\delta {X}_{N}|^2]+
 \sum_{i=0}^{N-1}  \sE[|\delta X_{i}|^2]\tau_N
 +
 \sE[|\lambda_0(g(\bar{X}_N,\sP_{\bar{X}_N})-\bar{g}(\bar{X}_N,\sP_{\bar{X}_N}))
+\delta \eta
|^2]
\\
&\quad
+
 \sum_{i=0}^{N-1}
\sE[|\lambda_0(\hat{f}(t_{i})-\bar{f}(t_{i}))+\delta \gamma_i|^2]\tau_N
 \bigg).
   \end{split}
\end{align}

Now we are ready to establish the desired stability result in Proposition \ref{thm:stab_exp}
by assuming $N\in \sN\cap[K_1,\infty)$.
Note that (H.\ref{assum:fbsde_discrete_exp}\ref{item:monotone_exp}) implies that one of the following two cases must be true, i.e., 
(1) $m< n$,  or $m=n$ with $\beta_1>0$;
 (2)
$m> n$, or $m=n$ with $\a>0,\beta_2>0$
(recall that $\a,\beta_1,\beta_2\ge 0$, $\a+\b_1>0$ and $\beta_1+\beta_2>0$, 
hence  when $m=n$, we have
 either $\beta_1>0$ or $\a,\b_2>0$).
 
For the first case, 
the fact that $G\in \sR^{m\t n}$ is full-rank
and $n\ge m$
shows that 
$|\cdot|_{G^*}: \sR^{m\t m'}\ni x\mapsto |G^*x|\in \sR $ is a norm on $\sR^{m\t m'}$ for any $m'\in \sN$.
Thus the equivalence of norms on Euclidean spaces
and Lemma  \ref{lemma:stab_a_priori_exp} (with $\eps_3={\b_1}/{2}$)
imply for all $\eps_1,\eps_2>0$  that
\begin{align*}
\begin{split}
& \sum_{i=0}^{N-1}
  \sE[
|\delta Y_{i}|^2+|\delta Z_i|^2 ]\tau_N
\\
& \le
\eps_1\sE[ |G^*\delta Y_0|^2]
+\frac{C}{\eps_1}\sE[ | \delta \xi_0|^2]
+
\eps_2
\sE[ |G\delta X_N|^2]
+
\frac{C}{\eps_2}\sE[ |\lambda_0(g(\bar{X}_N,\sP_{\bar{X}_N})-\bar{g}(\bar{X}_N,\sP_{\bar{X}_N}))
+\delta \eta|^2]
\\
&\quad
 +\sum_{i=0}^{N-1}
 \sE\bigg[
\eps_2 | G \delta X_i|^2
+C\bigg(\frac{1}{\eps_2}+\tau_N\bigg)
|
\lambda _0
(\hat{f}(t_{i})-\bar{f}(t_{i}))+\delta \gamma_i|^2
\\
& \quad
+
{C}
\bigg(
(1+\tau_N) |
 \lambda_0(
 \hat{b}(t_{i})-\bar{b}(t_{i}))+\delta \phi_{i}|^2
 +
 | 
 \lambda_0( \hat{\sigma}(t_i)-\bar{\sigma}(t_i))+\delta \psi_{i}|^2
 \bigg) \bigg]\tau_N
 \\
 &\quad
 +C\sum_{i=0}^{N-1}
 \sE[|  \delta X_i|^2+|\delta Y_{i}|^2+| \delta  {Z}_i|^2]\tau^2_N.
\end{split}
\end{align*}
Observe that there exists $K_2\in \sN$, depending only on the constants in (H.\ref{assum:fbsde_discrete_exp}),
such that    for all $N\in \sN\cap[K_2,\infty)$,
$\tau_NC\le 1/2$, which implies 
the above estimate still holds 
without the last two terms 
$\sum_{i=0}^{N-1}
 \sE[|\delta Y_{i}|^2+| \delta  {Z}_i|^2]\tau^2_N$.
Choosing a small $\eps_2$ and   substituting the above estimate into \eqref{eq:deltaX_stab} 
yield  for all $\eps_1>0$,
\begin{align}\l{eq:deltaX_senario1_withY0}
\begin{split}
\sup_{i\in \cN}\sE[|\delta X_{i}|^2]
&\le 
\eps_1\sE[ |G^*\delta Y_0|^2]
+C
\bigg\{
\frac{1}{\eps_1}\sE[|\delta \xi_{0}|^2]
+
\sE[ |\lambda_0(g(\bar{X}_N,\sP_{\bar{X}_N})-\bar{g}(\bar{X}_N,\sP_{\bar{X}_N}))
+\delta \eta|^2]
\\
&\quad
+\sum_{i=0}^{N-1}
\bigg(
\sE[|\lambda_0(\hat{f}(t_{i})-\bar{f}(t_{i}))+\delta \gamma_i|^2]\tau_N
+
\sE[|\lambda_0(\hat{b}(t_{i})-\bar{b}(t_{i}))+\delta \phi_{i}|^2]\tau_N
\\
&
\quad 
+
\sE|\lambda_0(\hat{\sigma}(t_i)-\bar{\sigma}(t_i))+\delta \psi_{i}|^2]\tau_N
\bigg)
\bigg\}
+C\sum_{i=0}^{N-1}
 \sE[|  \delta X_i|^2]\tau^2_N.
\end{split}
\end{align}
which  still holds 
without the   term 
$C\sum_{i=0}^{N-1}
 \sE[|  \delta X_i|^2]\tau^2_N$, 
 as it holds for all sufficiently large $N$, 
$$
C\sum_{i=0}^{N-1} \sE[|  \delta X_i|^2]\tau^2_N\le CT\sup_{i\in \cN}\sE[|  \delta X_i|^2]\tau_N
 \le \tfrac{1}{2}\sup_{i\in \cN}\sE[|  \delta X_i|^2].$$
Then by further  substituting  \eqref{eq:deltaX_senario1_withY0} (with a small $\eps_1$) into \eqref{eq:deltaYZ_stab},
we obtain the desired upper bound for 
$\max_{i\in \cN}\sE[|\delta Y_i|^2]+\sum_{i=0}^{N-1}\sE[| \delta {Z}_i|^2]\tau_N+\sE[|\delta {M}_N|^2]$,
which together with \eqref{eq:deltaX_senario1_withY0} finishes the proof of the desired stability estimate for the  first scenario.

For the alternative case, 
we see from the fact that $G\in \sR^{m\t n}$ is full-rank
and $m\ge n$ that
$|\cdot|_{G}: \sR^{n\t m'}\ni x\mapsto |Gx|\in \sR $ is a norm on $\sR^{n\t m'}$ for any $m'\in \sN$.
Thus the equivalence of norms on Euclidean spaces
and Lemma  \ref{lemma:stab_a_priori_exp}  (with $\eps_2={\min\{1,\a,\b_2\}}/{2}$)
imply for all $\eps_1,\eps_3>0$  that
\begin{align*}
\begin{split}
& \sE[
|\delta X_{N}|^2]
+
 \sum_{i=0}^{N-1}
  \sE[
|\delta X_{i}|^2
 ]\tau_N
\\
& \le
\eps_1\sE[ |G^*\delta Y_0|^2]
+\frac{C}{\eps_1}\sE[ | \delta \xi_0|^2]
+
C\sE[ |\lambda_0(g(\bar{X}_N,\sP_{\bar{X}_N})-\bar{g}(\bar{X}_N,\sP_{\bar{X}_N}))
+\delta \eta|^2]
\\
&\quad
 +\sum_{i=0}^{N-1}
 \sE\bigg[
\eps_3 (|G^*\delta Y_{i}|^2+| G^*\delta  {Z}_i|^2)
+{C}
|
\lambda _0
(\hat{f}(t_{i})-\bar{f}(t_{i}))+\delta \gamma_i|^2
\\
& \quad
+
\frac{C}{\eps_3}
\bigg(
 |
 \lambda_0(
 \hat{b}(t_{i})-\bar{b}(t_{i}))+\delta \phi_{i}|^2
 +
 | 
 \lambda_0( \hat{\sigma}(t_i)-\bar{\sigma}(t_i))+\delta \psi_{i}|^2
 \bigg) \bigg]\tau_N
 \\
& \quad
 +C\sum_{i=0}^{N-1}
 \sE[|  \delta X_i|^2+|\delta Y_{i}|^2+| \delta  {Z}_i|^2]\tau^2_N.
\end{split}
\end{align*}
Observe that the term $\sum_{i=0}^{N-1}
 \sE[|  \delta X_i|^2]\tau^2_N$ on the right-hand side of the above estimate can be eliminated for all sufficiently small $\tau_N$.
Then by choosing a small $\eps_3$, we can see from \eqref{eq:deltaYZ_stab} that
it holds for all sufficiently small $\tau_N,\eps_1>0$ that
\begin{align*}
\begin{split}
&\max_{i\in \cN}\sE[|\delta Y_i|^2]+\sum_{i=0}^{N-1}\sE[| \delta {Z}_i|^2]\tau_N+\sE[|\delta {M}_N|^2]
\\
&\le
 \eps_1\sE[ |G^*\delta Y_0|^2]
+C
\bigg\{
\frac{1}{\eps_1}\sE[|\delta \xi_{0}|^2]
+
\sE[ |\lambda_0(g(\bar{X}_N,\sP_{\bar{X}_N})-\bar{g}(\bar{X}_N,\sP_{\bar{X}_N}))
+\delta \eta|^2]
\\
&\quad
+\sum_{i=0}^{N-1}
\bigg(
\sE[|\lambda_0(\hat{f}(t_{i})-\bar{f}(t_{i}))+\delta \gamma_i|^2]\tau_N
+
\sE[|\lambda_0(\hat{b}(t_{i})-\bar{b}(t_{i}))+\delta \phi_{i}|^2]\tau_N
\\
&
\quad 
+
\sE|\lambda_0(\hat{\sigma}(t_i)-\bar{\sigma}(t_i))+\delta \psi_{i}|^2]\tau_N
\bigg)
\bigg\}.
   \end{split}
\end{align*}
Hence  choosing a small $\eps_1$ in the above estimate
gives us the desired upper bound for 
the left-hand side.
We can then conclude from 
\eqref{eq:deltaX_stab}
the desired stability estimate for the  second scenario,
which subsequently finishes the proof of Proposition \ref{thm:stab_exp}.
\end{proof}


\begin{proof}[Proof of Lemma \ref{lem:linear_existence}]
Throughout this proof, 
for each $n'\in\sN$,
let $\sS^{n'}_>$ be the space of all ${n'}\t {n'}$ symmetric positive definite matrices.
We  separate the proof  into  two cases: $n \ge m$ and $n\le m$.

Let us start with the first case where $n\ge m$.
The fact that $n\ge m$ and  $G\in \sR^{m\t n}$ is full-rank imply that $GG^*\in \sS^m_>$.
Let
 $\bar{X}$ satisfy the following  S$\Delta$E:
$$
\Delta \bar{X}_i=(\sI_n-G^*(GG^*)^{-1}G)(\phi_{i}\tau_N+\psi_i\Delta W_i), \; i\in \cN_{<N}; 
\q \bar{X}_0=(\sI_n-G^*(GG^*)^{-1}G)\xi_0,
$$
and assume that $(\tilde{X},\tilde{Y},\tilde{Z},\tilde{M})\in \cS_N$
 solve the FBS$\Delta$E: for all $i\in \cN_{<N}$,
\begin{subequations}\label{eq:MV-fbsde_tilde_n>m}
\begin{align}
\Delta \tilde{X}_i&=(-\b_1 GG^*\tilde{Y}_{i}+G\phi_{i})\tau_N  +(-\b_1 GG^*\tilde{Z}_i+G\psi_i)\, \Delta W_i, 
\l{eq:fbsde_fwd_tilde_n>m}\\
\Delta \tilde{Y}_i&=-(\b_2\tilde{X}_i+\gamma_i)\tau_N+\tilde{Z}_i\,\Delta W_i+\Delta \tilde{M}_i,
\l{eq:fbsde_bwd_tilde_n>m}\\
\tilde{X}_0&=G\xi_0,\q \tilde{Y}_N=\tilde{X}_N+\eta,
\l{eq:fbsde_terminal_tilde_n>m}
\end{align}
\end{subequations}
then 
 one can easily check
by using the linearity of  equations
 that 
$(X,Y, Z, M)\coloneqq (G^*(GG^*)^{-1}\tilde{X}+\bar{X},\tilde{Y},\tilde{Z},\tilde{M})\in \cS_N$ is a solution to 
\eqref{eq:MV-fbsde_exp_lambda} with $\lambda=0$
(note that $G\bar{X}\equiv 0$
and $ \tilde{X}\equiv GX$ on $[0,T]$).
Hence it suffices to construct a solution to \eqref{eq:MV-fbsde_tilde_n>m}.
For notational simplicity, we shall write $K=GG^*\in \sS^m_>$, $\tilde{\xi}_0=G\xi_0$, $\tilde{\phi}=G\phi$ and $\tilde{\psi}=G\psi$ in the subsequent analysis.

Let us consider the matrices $(P_i)_{i\in \cN}$ satisfying
$P_N=\sI_m$ and
 for each $i\in \cN_{<N}$ that
\bb\l{eq:riccati_n>m}
{P_i-P_{i+1}}
=(\b_2\sI_m-\b_1P_{i+1}KP_{i}){\tau_N}.
\ee
We shall show by induction that it holds for all $i\in \cN$ that $P_{i}\in \sS^m_>$ is uniquely defined and commutes with $K$.
The  induction hypothesis  clearly holds  for the index $N$, and we shall assume it holds for some index $i+1$ with $i\in \cN_{<N}$.
The fact that $K,P_{i+1}\in \sS^m_>$ and $KP_{i+1}=P_{i+1}K$ implies that 
$P_{i+1}K\in \sS^m_>$ and $\sI_m+\b_1 P_{i+1}K\tau_N\in \sS^m_>$, which along with $\b_1,\b_2\ge 0$ shows that $P_i$  is well-defined and can be written as 
\begin{align}\l{eq:P_i_n>m}
\begin{split}
P_i
&=(\sI_m+\b_1 P_{i+1}K\tau_N)^{-1}(P_{i+1}+\b_2\tau_N\sI_m).
\end{split}
\end{align}
Moreover, the fact that $KP_{i+1}=P_{i+1}K$ gives us the  identities
that
$P_{i+1}(\sI_m+\b_1 P_{i+1}K\tau_N)
=(\sI_m+\b_1 P_{i+1}K\tau_N)P_{i+1}$
and 
$
K(\sI_m+\b_1 P_{i+1}K\tau_N)
=(\sI_m+\b_1 P_{i+1}K\tau_N)K$,
which show that 
both 
$P_{i+1}$ and $K$ commute with $(\sI_m+\b_1 P_{i+1}K\tau_N)^{-1}$.
Therefore, we see that 
$P_i\in \sS^m_>$,
and
$P_i$ commutes with $K$,
which shows the induction hypothesis also holds for the index $i\in \cN$.

With the above matrices $(P_i)_{i\in \cN}$ at hand, we consider the following
linear BS$\Delta$E:
$p_N=\eta$, and for all  $i\in \cN_{<N}$ that
\begin{align}\l{eq:dp_n>m}
\Delta p_i=-[P_{i+1}(-\b_1Kp_{i}+\tilde{\phi}_{i})+\gamma_i]\tau_N
+q_i\Delta W_i+\Delta m_i,
\end{align}
where 
 $(p,q,m)\in \cM^2(0,T;\sR^m\t \sR^{m\t d}\t \sR^m)$ are piecewise-constant processes defined on $\pi_N$ satisfying
 $m_0=0$, 
and for all $i\in \cN_{<N}$ that
$\sE_{i}[\Delta m_i] =0$ and  $\sE_{i}[(\Delta m_i)(\Delta W_i)^*] =0$.
The existence of such solutions follows from a standard  backward induction together with the Kunita--Watanabe decomposition (see e.g.~\cite[Theorem 2.2]{bielecki2015}).
Then we define the processes $(\tilde{X},\tilde{Y},\tilde{Z},\tilde{M})$ such that 
$\tilde{M}\equiv m$, 
$\tilde{Z}_i=(\sI_m+\b_1P_{i+1}K)^{-1}(P_{i+1}\tilde{\psi}_i+q_i)$
for all $i\in\cN_{<N}$, 
$\tilde{X}_0=\tilde{\xi}_0$,
\bb\l{eq:dX_n>m}
\Delta \tilde{X}_i=[-\b_1K(P_{i}\tilde{X}_{i}+p_{i})+\tilde{\phi}_{i}]\tau_N
+(-\b_1K\tilde{Z}_i+\tilde{\psi}_i)\Delta W_i
\q \fa i\in \cN_{<N},
\ee
and $\tilde{Y}_i=P_i\tilde{X}_i+p_i$ for all $i\in \cN$.
Note that 
$ \b_1\ge 0$
and $P_{i+1}K\in \sS^m_>$
imply that  $(\tilde{X},\tilde{Y},\tilde{Z},\tilde{M})$ are well-defined adapted processes and satisfy both \eqref{eq:fbsde_fwd_tilde_n>m} and \eqref{eq:fbsde_terminal_tilde_n>m}.
Moreover,
we have for each $i\in \cN_{<N}$ that  
$\Delta \tilde{Y}_i=\Delta P_i\tilde{X}_{i}+ P_{i+1}\Delta \tilde{X}_i+\Delta p_i$.
Hence by substituting   \eqref{eq:riccati_n>m}, \eqref{eq:dp_n>m} and \eqref{eq:dX_n>m}
into the identity, 
we can verify via a straightforward calculation
that 
$(\tilde{X},\tilde{Y},\tilde{Z},\tilde{M})$ also satisfies \eqref{eq:fbsde_bwd_tilde_n>m}.
This proves  the  existence of solutions to  \eqref{eq:MV-fbsde_exp_lambda} 
with $\lambda=0$
for the case where $n\ge m$.

We now proceed to establish the existence of solutions for the second case where $m\ge n$, whose proof is similar to the above analysis.
The fact that $m\ge n$ and  $G\in \sR^{m\t n}$ is full-rank imply that $G^*G\in \sS^n_>$.
%
%
Let 
$(\bar{Y},\bar{Z},\bar{M})$  (where the martingale $\bar{M}$ is strongly orthogonal to $W$)  satisfy the following  
BS$\Delta$E: for all $i\in \cN_{<N}$,
\begin{align*}
\Delta \bar{Y}_i&=-(\sI_m-G(G^*G)^{-1}G^*)\gamma_i\tau_N+\bar{Z}_i\,\Delta W_i+\Delta \bar{M}_i,
\\
 \bar{Y}_N&=(\sI_m-G(G^*G)^{-1}G^*)\eta,
\end{align*}
and assume that  $(\tilde{X},\tilde{Y},\tilde{Z},\tilde{M})\in \cS_N$ 
solve the FBS$\Delta$E: for all $i\in \cN_{<N}$,
\begin{subequations}\label{eq:MV-fbsde_tilde_m>n}
\begin{align}
\Delta \tilde{X}_i&=(-\b_1 \tilde{Y}_{i}+\phi_{i})\tau_N  +(-\b_1 \tilde{Z}_i+\psi_i)\, \Delta W_i, 
\l{eq:fbsde_fwd_tilde_m>n}\\
\Delta \tilde{Y}_i&=-(\b_2G^*G\tilde{X}_i+G^*\gamma_i)\tau_N+\tilde{Z}_i\,\Delta W_i+\Delta \tilde{M}_i,
\l{eq:fbsde_bwd_tilde_m>n}\\
\tilde{X}_0&=\xi_0,\q \tilde{Y}_N=G^*G\tilde{X}_N+G^*\eta,
\l{eq:fbsde_terminal_tilde_m>n}
\end{align}
\end{subequations}
then
the linearity of the equations shows that 
 the $4$-tuple   $(X,Y, Z, M)\in \cS_N$ defined by $X\coloneqq \tilde{X}$, $(Y, Z, M)\coloneqq G(G^*G)^{-1}(\tilde{Y},\tilde{Z},\tilde{M})+(\bar{Y},\bar{Z},\bar{M})$ is a solution to 
\eqref{eq:MV-fbsde_exp_lambda} with $\lambda=0$
(note that $G^*\bar{Y}=G^*\bar{Z}=G^*\bar{M} =0$
 on $[0,T]$).
 Since a standard  backward induction argument together with the Kunita--Watanabe decomposition 
leads to  the  existence of  $(\bar{Y},\bar{Z},\bar{M})$ (see e.g.~\cite[Theorem 2.2]{bielecki2015}),
it remains to construct a solution to \eqref{eq:MV-fbsde_tilde_m>n}.
For notational simplicity, we shall write $K=G^*G\in \sS^n_>$, $\tilde{\gamma}=G^*\gamma\in M^2(0,T;\sR^n)$ and $\tilde{\eta}=G^*\eta\in L^2(\cF_T;\sR^n)$ in the subsequent analysis.

Let us consider the matrices $(P_i)_{i\in \cN}$ satisfying
$P_N=K$ and
 for each $i\in \cN_{<N}$ that
\bb\l{eq:riccati_n<m}
{P_i-P_{i+1}}=(\b_2 K-\b_1 P_{i+1}P_{i}){\tau_N}.
\ee
A straightforward inductive argument shows 
that $P_i\in \sS^n_>$ for all $i\in \cN$  and 
$P_i=(\sI_n+\b_1P_{i+1}\tau_N)^{-1}(P_{i+1}+\b_2K\tau_N)$
for all $i\in \cN_{<N}$.
We shall consider the piecewise-constant processes
$(p,q,m)\in \cM^2(0,T;\sR^n\t \sR^{n\t d}\t \sR^n)$   which satisfy the
linear BS$\Delta$E:
\begin{align}\l{eq:dp_n<m}
\Delta p_i=-[P_{i+1}(-\b_1p_{i}+{\phi}_{i})+\tilde{\gamma}_i]\tau_N
+q_i\Delta W_i+\Delta m_i,
\; i\in \cN_{<N}; \quad 
p_N=\tilde{\eta},
\end{align}
and
enjoy the properties that 
 $m_0=0$, 
and for all $i\in \cN_{<N}$,
$\sE_{i}[\Delta m_i] =0$ and  $\sE_{i}[(\Delta m_i)(\Delta W_i)^*] =0$.
The existence 
of $(p,q,m)$
 follows from a standard  backward induction and the Kunita--Watanabe decomposition. 
We further define the processes $(\tilde{X},\tilde{Y},\tilde{Z},\tilde{M})$ such that 
$\tilde{M}\equiv m$, 
{$\tilde{Z}_i=(\sI_n+\b_1P_{i+1})^{-1}(P_{i+1}\psi_i+q_i)$}
for all $i\in\cN_{<N}$,
$\tilde{X}_0={\xi}_0$,
$$
\Delta \tilde{X}_i=[-\b_1(P_{i}\tilde{X}_{i}+p_{i})+{\phi}_{i}]\tau_N
+(-\b_1\tilde{Z}_i+{\psi}_i)\Delta W_i
\q \fa i\in \cN_{<N},
$$
and $\tilde{Y}_i=P_i\tilde{X}_i+p_i$ for all $i\in \cN$.
Then by using the identity that 
$\Delta \tilde{Y}_i=\Delta P_i\tilde{X}_{i}+ P_{i+1}\Delta \tilde{X}_i+\Delta p_i$,
we can directly verify that $(\tilde{X},\tilde{Y},\tilde{Z},\tilde{M})$ satisfies \eqref{eq:fbsde_bwd_tilde_m>n}
for all $i\in \cN_{<N}$. 
This proves   that \eqref{eq:MV-fbsde_exp_lambda} with $\lambda=0$ admits a solution for the case where $m\ge n$.
\end{proof}

\subsection{Proof of Proposition \ref{thm:time_error}}
\label{appendix:proof_continuous_error}
%
We start by deriving  an upper bound of the squared $L^2$-error between 
 $(X_i,Y_i,\bar{Z}_i)_{i\in \cN}$
 and the solution $(X^{\pi}_i, Y^{\pi}_i, Z^{\pi}_i, M^\pi_i)_{i\in \cN}$ to \eqref{eq:MV-fbsde_exp}.

\begin{Lemma}\l{lemma:a_prior_time_error}
Suppose (H.\ref{assum:fbsde_discrete_exp})-(H.\ref{assum:time_reg}) hold.
Let 
 $\a,\b_1,\b_2$ and $G$ be the constants
 in (H.\ref{assum:fbsde_discrete_exp}\ref{item:monotone_exp}),
 $L$ be the constant
 in (H.\ref{assum:fbsde_discrete_exp}\ref{item:lipschitz_exp}),
 $N\in \sN$, $(X,Y, Z)\in \cM^2(0,T;\sR^n\t \sR^m\t \sR^{m\t d })$ be the solution to \eqref{eq:fbsde_conts_intro},
$(\bar{Z}_i)_{i\in \cN_{<N}}$ be the random variables 
satisfying for all $i\in \cN_{<N}$ that 
$\bar{Z}_i=\frac{1}{\tau_N}\sE_i\big[\int_{t_i}^{t_{i+1}} Z_s\, ds\big]$,
$\bar{Z}$ be a c\`{a}dl\`{a}g extension of $(\bar{Z}_i)_{i\in \cN_{<N}}$ on $\pi_N$,
$(X^\pi,Y^\pi, Z^\pi,M^\pi)\in \cS_N$
be a solution to  \eqref{eq:MV-fbsde_exp} defined on $\pi_N$
and
 $(\delta X,\delta Y, \delta Z)= ( X-{X}^\pi, Y-{Y}^\pi, \bar{Z}-{Z}^\pi)$.
Then 
for all $\eps_1,\eps_2>0$,
\begin{align*}
&
\a \sE[|G \delta X_N|^2]+
\sum_{i=0}^{N-1}
\sE[
\beta_2|G\delta X_{i}|^2
+
\beta_1(
|G^*\delta Y_{i}|^2+|G^*\delta {Z}_i|^2)
 ]\tau_N
\\
& \le
\sum_{i=0}^{N-1}
\sE[
\eps_1(|\delta Y_{i}|^2+|\delta Z_{i}|^2)
+\eps_2|\delta X_{i}|^2
]\tau_N
+
C\big(\ol{\om}(\tau_N)^2
+
\cR_\pi(X,Y,Z)
\big)
\\
&\quad +
C\tau_N^{1/2}
\bigg(
\sum_{i=0}^{N-1}
\sE[
|\delta X_{i}|^2+|\delta Y_{i}|^2+|\delta Z_{i}|^2
]\tau_N
+\sE[| M^\pi_N|^2]
\bigg),
\end{align*}
where 
$C$ is a constant depending only on $\eps_1,\eps_2$ and the constants in  (H.\ref{assum:fbsde_discrete_exp}),
$\ol{\om}$ is the modulus of continuity in (H.\ref{assum:time_reg}),
and $\cR_\pi(X,Y,Z)$ is  defined  in \eqref{eq:L2_Regularity_proj}.
\end{Lemma}

\begin{proof}[Proof of Lemma \ref{lemma:a_prior_time_error}]
The proof follows from a slight extension of the arguments in Lemma \ref{lemma:stab_a_priori_exp}.
Throughout this proof, 
let ${\Theta}^\pi=({X}^\pi, {Y}^\pi, {Z}^\pi)$,
$\Theta=(X, Y, Z)$,  
$\bar{\Theta}=({X}, {Y}, \bar{Z})$,
and for each $t\in [0,T]$ and $\phi=b, \sigma, f$, 
let ${\phi}^\pi(t)= \phi(t, {\Theta}^\pi_t,\sP_{{\Theta}^\pi_t})$,
${\phi}(t)= \phi(t, {\Theta}_t,\sP_{{\Theta}_t})$
$\bar{\phi}(t)= \phi(t, \bar{\Theta}_t,\sP_{\bar{\Theta}_t})$
and 
$\delta {\phi}(t)=\bar{\phi}(t)- {\phi}^\pi(t)$.

For any given $i\in \cN_{<N}$, 
we can deduce from 
the equations \eqref{eq:MV-fbsde_exp} and \eqref{eq:fbsde_conts_intro}
that
\begin{align}
\Delta (\delta {X})_{i}
&=
\int_{t_i}^{t_{i+1}}(b(t)-{b}^\pi(t_{i}))\,dt +
\int_{t_i}^{t_{i+1}}(\sigma(t)-{\sigma}^\pi(t_i))\,d W_t,
\label{eq:difference_fwd}
\\
\Delta (\delta {Y})_{i}
&=
-\int_{t_i}^{t_{i+1}}(f(t)-{f}^\pi(t_{i}))\,dt+\int_{t_i}^{t_{i+1}} ({Z}_t-Z^\pi_i)\,d W_t -\Delta M^\pi_i,
\label{eq:difference_bwd}
\end{align}
 which along with   $\delta X_0=0$,
$\delta Y_N=g(X_N,\sP_{X_N})-g(X^\pi_N,\sP_{X^\pi_N})$
and 
 the It\^{o} isometry gives 
\begin{align}\l{eq:DeltaGXY_N_conts}
\begin{split}
&\sE[\la G\delta X_N,g(X_N,\sP_{X_N})-g(X^\pi_N,\sP_{X^\pi_N})\ra]
 =\sum_{i=0}^{N-1}\sE[\Delta\la G\delta X,\delta Y\ra_i]
\\
&
=\sum_{i=0}^{N-1} \sE[\la G\Delta \delta X_i, \delta Y_i\ra
+\la G \delta X_i, \Delta \delta Y_i\ra
+\la G \Delta\delta X_i, \Delta \delta Y_i\ra]
\\
&
=\sum_{i=0}^{N-1}
\sE\bigg[ \int_{t_i}^{t_{i+1}}
\bigg(\la G^*(\delta Y)_{i},
b(t)-{b}^\pi(t_{i})\ra 
+
\la G(\delta X)_{i},-(f(t)-{f}^\pi(t_{i}))\ra
\\
&\quad
+\la G^*({Z}_t-Z^\pi_i), \sigma(t)-{\sigma}^\pi(t_i)
\ra
\bigg) \,dt\bigg]
+\sum_{i=0}^{N-1}{\Sigma_i},
\end{split}
\end{align}
where for each $i\in \cN_{<N}$, the term $\Sigma_i$ is defined by
\begin{align}\l{eq:Sigma_i_conts}
\begin{split}
{\Sigma_i}
&=
\sE\bigg[
\bigg\la 
\int_{t_i}^{t_{i+1}}G(b(t)-{b}^\pi(t_{i}))\,dt,\int_{t_i}^{t_{i+1}}-(f(t)-{f}^\pi(t_{i}))\,dt
\bigg\ra
\\
&\quad 
+\bigg\la 
\int_{t_i}^{t_{i+1}}G(b(t)-{b}^\pi(t_{i}))\,dt,\int_{t_i}^{t_{i+1}} ({Z}_t-Z^\pi_i)\,d W_t -\Delta M^\pi_i
\bigg\ra
\\
&\quad
+\bigg\la 
\int_{t_i}^{t_{i+1}}G(\sigma(t)-{\sigma}^\pi(t_i))\,d W_t,\int_{t_i}^{t_{i+1}}-(f(t)-{f}^\pi(t_{i}))\,dt
\bigg\ra
\bigg].
\end{split}
\end{align}
By first adding and subtracting the terms $\bar{b}(t_i)$, $\bar{f}(t_i)$,  $\bar{Z}_i$ and $\bar{\sigma}(t_i)$
in \eqref{eq:DeltaGXY_N_conts} and then applying the monotonicity condition (H.\ref{assum:fbsde_discrete_exp}\ref{item:monotone_exp}),
 \begin{align}
\a\sE[| G\delta X_N|^2]
&\le
\sum_{i=0}^{N-1}
\sE[ 
\la G^*\delta Y_{i},
\delta {b}(t_{i})\ra 
+
\la G\delta X_{i},-\delta {f}(t_{i}))\ra
+\la G^*\delta Z_i, \delta {\sigma}(t_i)
\ra
 ]\tau_N
+\sum_{i=0}^{N-1}\Psi_i
+\sum_{i=0}^{N-1}{\Sigma_i}
\nb\\
&\le
-\sum_{i=0}^{N-1}
\sE[ 
\b_1 
(| G^*\delta Y_{i}|^2
 +
|G^*\delta Z_i|^2)
+
\b_2|G\delta X_{i}|^2
]\tau_N
+\sum_{i=0}^{N-1}\Psi_i
+\sum_{i=0}^{N-1}{\Sigma_i},
\l{eq:DeltaGXYZ_N_conts}
\end{align}
where for each $i\in \cN_{<N}$, the term $\Psi_i$ is defined by
\begin{align*}
\begin{split}
{\Psi_i}
&=
\sE\bigg[ \int_{t_i}^{t_{i+1}}
\bigg(\la G^*\delta Y_{i},
b(t)-\bar{b}(t_{i})\ra 
+
\la G\delta X_{i},-(f(t)-\bar{f}(t_{i}))\ra
\\
&\quad
+\la G^*({Z}_t-\bar{Z}_i), \sigma(t)-\bar{\sigma}(t_i)
\ra
+\la G^*\delta Z_i, \sigma(t)-\bar{\sigma}(t_i)
\ra
\bigg) \,dt\bigg].
\end{split}
\end{align*}
Note that
the derivation 
 of $\Psi_i$ also used 
$\sE \big[\int_{t_i}^{t_{i+1}} \la {Z}_t-\bar{Z}_i, \delta \sigma (t_i)
\ra \,dt\big]=0$.

Now we proceed to estimate $\Psi_i$ and $\Sigma_i$ for a given $i\in \cN_{<N}$. 
By using Young's inequality,
(H.\ref{assum:fbsde_discrete_exp}\ref{item:lipschitz_exp}),
(H.\ref{assum:time_reg})  
and 
the inequality that $\cW^2_2(\sP_{\bar{\Theta}_i},\sP_{{\Theta^\pi_i}})\le \sE[|\bar{\Theta}_i-{\Theta^\pi_i}|^2]$,
it holds for all $\eps_1,\eps_2>0$ that, 
there exists a constant $C_{(\eps_1,\eps_2,G,L)}>0$, depending only on $\eps_1$, $\eps_2$, $G$ and $L$,
such that 
\begin{align}\l{eq:Psi_i_conts}
\begin{split}
|{\Psi_i}|
&\le 
\sE\bigg[ \int_{t_i}^{t_{i+1}}
\bigg(
\eps_1(|\delta Y_{i}|^2+|\delta Z_{i}|^2)
+
\eps_2|\delta X_{i}|^2
+
C_{(\eps_1,\eps_2,G,L)}(\bar{\om}(\tau_N)^2+
|\Theta_t-\bar{\Theta}_i|^2)
\bigg) \,dt\bigg]
\\
&= 
\sE[
\eps_1(|\delta Y_{i}|^2+|\delta Z_{i}|^2)
+\eps_2|\delta X_{i}|^2
]\tau_N
+
C_{(\eps_1,\eps_2,G,L)}\bigg(
\bar{\om}(\tau_N)^2\tau_N+
\sE
\bigg[
 \int_{t_i}^{t_{i+1}}
|\Theta_t-\bar{\Theta}_i|^2
 \,dt
 \bigg]\bigg).
\end{split}
\end{align}
We then turn to the term $\Sigma_i$ by referring the three quantities in \eqref{eq:Sigma_i_conts} as $\Sigma_{i,1}$,
$\Sigma_{i,2}$ and $\Sigma_{i,3}$. 
For notational simplicity, we shall denote by $C$  a generic  positive constant, 
which depends only on the constants   in (H.\ref{assum:fbsde_discrete_exp})
and  may take a different value at each occurrence.
We start by using
$\delta {b}(t_{i})\in L^2(\cF_{t_i};\sR^n)$,
Young's inequality 
and the It\^{o} isometry  
to estimate the term $\Sigma_{i,2}$:
\begin{align*}
\begin{split}
\Sigma_{i,2}
&= 
\sE\bigg[
\bigg\la 
\int_{t_i}^{t_{i+1}}G(b(t)-\bar{b}(t_i)+\delta {b}(t_{i}))\,dt,\int_{t_i}^{t_{i+1}} ({Z}_t-Z^\pi_i)\,d W_t -\Delta M^\pi_i
\bigg\ra
\bigg]
\\
&\le
\frac{\tau_N^{-1/2}}{4}
\sE\bigg[\bigg| 
\int_{t_i}^{t_{i+1}}G(b(t)-\bar{b}(t_i))\,dt
\bigg|^2\bigg]
+{\tau_N^{1/2}}
\sE\bigg[\bigg|
\int_{t_i}^{t_{i+1}} ({Z}_t-Z^\pi_i)\,d W_t -\Delta M^\pi_i
\bigg|^2\bigg]
\\
&=
\frac{\tau_N^{-1/2}}{4}
\sE\bigg[\bigg| 
\int_{t_i}^{t_{i+1}}G(b(t)-\bar{b}(t_i))\,dt
\bigg|^2\bigg]
+{\tau_N^{1/2}}
\sE\bigg[
\int_{t_i}^{t_{i+1}} |{Z}_t-Z^\pi_i|^2\,d t
+ |\Delta M^\pi_i|^2
\bigg],
\end{split}
\end{align*}
from which,  by using  H\"{o}lder's inequality,
the fact that 
$\sE \big[\int_{t_i}^{t_{i+1}} \la {Z}_t-\bar{Z}_i, \delta Z_i
\ra \,dt\big]=0$
and the assumptions
(H.\ref{assum:fbsde_discrete_exp}\ref{item:lipschitz_exp})
and
(H.\ref{assum:time_reg}),  
we can obtain that
\begin{align}\l{eq:Sigma_i2_conts}
\begin{split}
\Sigma_{i,2}
&\le
\frac{\tau_N^{1/2}}{4}
\sE\bigg[
\int_{t_i}^{t_{i+1}}|G(b(t)-\bar{b}(t_i))|^2\,dt
\bigg]
+{\tau_N^{1/2}}
\sE\bigg[
\int_{t_i}^{t_{i+1}} (|{Z}_t-\bar{Z}_i|^2+|\delta Z_i|^2)\,d t
+ |\Delta M^\pi_i|^2
\bigg]
\\
&\le 
C\tau_N^{1/2}
\bigg(
\bar{\om}(\tau_N)^2\tau_N+
\sE
\bigg[
 \int_{t_i}^{t_{i+1}}
|\Theta_t-\bar{\Theta}_i|^2
 \,dt
 \bigg]
+\sE[|\delta Z_i|^2]\tau_N
+ \sE[|\Delta M^\pi_i|^2]
\bigg).
\end{split}
\end{align}
Similarly,  
by using Young's inequality, the It\^{o} isometry,  H\"{o}lder's inequality,
 (H.\ref{assum:fbsde_discrete_exp}\ref{item:lipschitz_exp})
and
(H.\ref{assum:time_reg}),  
we can obtain the following upper bound of  $\Sigma_{i,3}$:
\begin{align}\l{eq:Sigma_i3_conts}
\begin{split}
\Sigma_{i,3}
&=
\sE\bigg[\bigg\la 
\int_{t_i}^{t_{i+1}}G(\sigma(t)-{\sigma}^\pi(t_i))\,d W_t,\int_{t_i}^{t_{i+1}}-(f(t)-\bar{f}(t_i)+\delta {f}(t_{i}))\,dt
\bigg\ra
\bigg]
\\
&\le
\tau_N^{1/2}\sE\bigg[\bigg| 
\int_{t_i}^{t_{i+1}}G(\sigma(t)-{\sigma}^\pi(t_i))\,d W_t\bigg|^2\bigg]
+\frac{\tau_N^{-1/2}}{4}
\sE\bigg[\bigg|\int_{t_i}^{t_{i+1}}(f(t)-\bar{f}(t_i))\,dt\bigg|^2
\bigg]
\\
&\le 
C\tau_N^{1/2}\sE\bigg[
\int_{t_i}^{t_{i+1}}|\sigma(t)-\bar{\sigma}(t_i)+\delta {\sigma}(t_i)|^2\,d t\bigg]
+C\tau_N^{1/2}
\sE\bigg[\int_{t_i}^{t_{i+1}}|f(t)-\bar{f}(t_i)|^2\,dt
\bigg]
\\
&\le 
C\tau_N^{1/2}
\bigg(
\bar{\om}(\tau_N)^2\tau_N+
\sE
\bigg[
 \int_{t_i}^{t_{i+1}}
|\Theta_t-\bar{\Theta}_i|^2
 \,dt
 \bigg]
+\sE[|\delta \Theta_i|^2]\tau_N
\bigg).
\end{split}
\end{align}
Furthermore,
by  Young's inequality, H\"{o}lder's inequality,
 (H.\ref{assum:fbsde_discrete_exp}\ref{item:lipschitz_exp}),
(H.\ref{assum:time_reg})
and the fact that $\tau_N\le T$, 
\begin{align}\l{eq:Sigma_i1_conts}
\begin{split}
\Sigma_{i,1}
&\le
\frac{1}{2}
\sE\bigg[\bigg| 
\int_{t_i}^{t_{i+1}}G(b(t)-{b}^\pi(t_{i}))\,dt\bigg|^2\bigg]
+\frac{1}{2}
\sE\bigg[\bigg|\int_{t_i}^{t_{i+1}}(f(t)-{f}^\pi(t_{i}))\,dt\bigg|^2
\bigg]
\\
&\le 
C\tau_N\sE\bigg[
\int_{t_i}^{t_{i+1}}|b(t)-\bar{b}(t_i)+\delta {b}(t_i)|^2\,d t\bigg]
+C\tau_N
\sE\bigg[\int_{t_i}^{t_{i+1}}|f(t)-\bar{f}(t_i)+\delta {f}(t_i)|^2\,dt
\bigg]
\\
&\le 
C\tau_N^{1/2}
\bigg(
\bar{\om}(\tau_N)^2\tau_N+
\sE
\bigg[
 \int_{t_i}^{t_{i+1}}
|\Theta_t-\bar{\Theta}_i|^2
 \,dt
 \bigg]
+\sE[|\delta \Theta_i|^2]\tau_N
\bigg).
\end{split}
\end{align}
The desired conclusion then follows by
combining  \eqref{eq:DeltaGXYZ_N_conts}, \eqref{eq:Psi_i_conts}, \eqref{eq:Sigma_i2_conts}
\eqref{eq:Sigma_i3_conts}, \eqref{eq:Sigma_i1_conts} and using
$\sum_{i=0}^{N-1}\sE\big[
 \int_{t_i}^{t_{i+1}}
|\Theta_t-\bar{\Theta}_i|^2
 \,dt \big]\le \max\{T,1\}\cR_\pi(X,Y,Z)$
and 
$\sum_{i=0}^{N-1}\sE[|\Delta M^\pi_i|^2]=\sE[| M^\pi_N|^2]$.
\end{proof}


\begin{proof}[Proof of Proposition \ref{thm:time_error}]
This proof follows from an analogue argument as that for Proposition \ref{thm:stab_exp}.
Throughout this proof, 
let $\bar{Z}$ be a c\`{a}dl\`{a}g extension of the random variables $(\bar{Z}_i)_{i\in \cN_{<N}}$,
let $N\in \sN$ be  sufficiently large 
such that  \eqref{eq:MV-fbsde_exp} defined on $\pi_N$
admits a unique solution $(X^\pi,Y^\pi, Z^\pi, M^\pi) \in \cS_N$,
let ${\Theta}^\pi=({X}^\pi, {Y}^\pi, {Z}^\pi)$,
$\Theta=(X, Y, Z)$,  
$\bar{\Theta}=({X}, {Y}, \bar{Z})$,
$( \delta X,\delta Y, \delta Z)= (  X-{X}^\pi, Y-{Y}^\pi, \bar{Z}-{Z}^\pi)$,
and for each $t\in [0,T]$, $\phi=b, \sigma, f$, 
let ${\phi}^\pi(t)= \phi(t, {\Theta}^\pi_t,\sP_{{\Theta}^\pi_t})$,
${\phi}(t)= \phi(t, {\Theta}_t,\sP_{{\Theta}_t})$
$\bar{\phi}(t)= \phi(t, \bar{\Theta}_t,\sP_{\bar{\Theta}_t})$,
and 
$\delta {\phi}(t)=\bar{\phi}(t)- {\phi}^\pi(t)$.
We   denote by $C$ a generic constant, 
which
depends on   constants   in (H.\ref{assum:fbsde_discrete_exp})
but independent of $N$,
and  may take a different value at each occurrence.

Note that
by slightly modifying the arguments for \eqref{eq:deltaX_stab} in Proposition \ref{thm:stab_exp},
we 
can obtain from \eqref{eq:difference_fwd},
 Gronwall's inequality {  in Lemma \ref{lemma:gronwall}}, 
 (H.\ref{assum:fbsde_discrete_exp}),
 (H.\ref{assum:time_reg}) 
 and
  $\delta X_0=0$ 
 that
\begin{align}
\max_{i\in \cN}\sE[  |\delta X_{i}|^2]
\nb
  &\le
 C\bigg(
 \ol{\om}(\tau_N)^2+\sum_{i=0}^{N-1}
\sE \bigg[\int_{t_i}^{t_{i+1}} 
 |\Theta_t-\bar{\Theta}_i|^2
\,dt\bigg]
+\sum_{i=0}^{N-1} \sE[
|\delta {Y}_{i}|^2
+|\delta {Z}_{i}|^2
]\tau_N
 \bigg)
 \nb
 \\
 &\le 
  C\bigg(
 \ol{\om}(\tau_N)^2+
 \cR_\pi(X,Y,Z)
+\sum_{i=0}^{N-1} \sE[
|\delta {Y}_{i}|^2
+|\delta {Z}_{i}|^2
]\tau_N
 \bigg),
 \label{eq:time_error_X}
  \end{align}
with the quantity $\cR_\pi(X,Y,Z)$ defined as in \eqref{eq:L2_Regularity_proj}.
On the other hand, 
with a slight modification of  the arguments for 
\eqref{eq:delta Y_i_estimate}
 in Proposition \ref{thm:stab_exp},
we  can obtain from \eqref{eq:difference_bwd}
and the product formula for $\Delta \la \delta Y,\delta Y\ra_i$
that
\begin{align*}
\begin{split}
&\sE[|\delta Y_{i}|^2]+\sum_{j=i}^{N-1}\sE\bigg[\int_{t_j}^{t_{j+1}}| Z_t- {Z}^\pi_j|^2\,dt +|\Delta M^\pi_j|^2
\bigg]
\\
&=
\sE[|\delta Y_{N}|^2]
+\sum_{j=i}^{N-1} 
 \sE\bigg[   \bigg \la 2\delta Y_{j}- \int_{t_j}^{t_{j+1}}(f(t)-{f}^\pi(t_{j}))\,dt,
  \int_{t_j}^{t_{j+1}}(f(t)-{f}^\pi(t_{j}))\,dt
 \bigg \ra \bigg] 
 \\
 & \quad
 +\sum_{j=i}^{N-1} 
 \sE\bigg[   \bigg \la  \int_{t_j}^{t_{j+1}}(f(t)-{f}^\pi(t_{j}))\,dt,
  \int_{t_j}^{t_{j+1}} ({Z}_t-Z^\pi_j)\,d W_t
  -\Delta M^\pi_j
 \bigg \ra \bigg].
 \end{split}
\end{align*}
Note that for each $i\in \cN_{<N}$, by 
applying Young's inequality,
the It\^{o} isometry,  H\"{o}lder's inequality,
 (H.\ref{assum:fbsde_discrete_exp}\ref{item:lipschitz_exp}),
we see for all $\eps>0$ that
 the last term in the above inequality can be estimated as:
\begin{align*}
\begin{split}
 &\sE\bigg[   \bigg \la  \int_{t_i}^{t_{i+1}}(f(t)-{f}^\pi(t_{i}))\,dt,
  \int_{t_i}^{t_{i+1}} ({Z}_t-Z^\pi_i)\,d W_t
    -\Delta M^\pi_i
 \bigg \ra \bigg]
 \\
 &\quad 
 \le
\frac{1}{4\eps} \sE\bigg[   \bigg|  \int_{t_i}^{t_{i+1}}(f(t)-{f}^\pi(t_{i}))\,dt\bigg|^2\bigg]
+
\eps\sE\bigg[ \bigg| \int_{t_i}^{t_{i+1}} ({Z}_t-Z^\pi_i)\,d W_t
  -\Delta M^\pi_i
 \bigg|^2 \bigg]
 \\
 &\quad 
 \le
\frac{C\tau_N}{\eps} \sE\bigg[     \int_{t_i}^{t_{i+1}}
\bigg(\ol{\om}(\tau_N)^2+|\Theta_t-\bar{\Theta}_i|^2+|\delta \Theta_i|^2
\bigg)
\,dt\bigg]
+
\eps\sE\bigg[  \int_{t_i}^{t_{i+1}} |{Z}_t-Z^\pi_i|^2\,dt
 +  |\Delta M^\pi_i|^2
 \bigg].
 \end{split}
\end{align*}
Hence, by using Gronwall's inequality {  in Lemma \ref{lemma:gronwall}}
and 
the identity that
$\sE[\int_{t_i}^{t_{i+1}}| Z_t- {Z}^\pi_i|^2\,dt]
=
\sE[\int_{t_i}^{t_{i+1}}| Z_t- \bar{Z}_i|^2\,dt]
+\sE[|\delta Z_i|^2]\tau_N$ for all $i\in \cN_{<N}$,
  a similar argument as that for \eqref{eq:deltaYZ_stab}
 in Proposition \ref{thm:stab_exp} shows 
 that for all sufficiently large $N$,
\begin{align}
&\max_{i\in \cN}\sE[  |\delta Y_{i}|^2]+\sum_{i=0}^{N-1}\sE[|\delta Z_i|^2]\tau_N +\sE[|M^\pi_N|^2]
\nb
\\
  &\le
 C\bigg(
 \sE[|\delta Y_N|^2]
 +\sum_{i=0}^{N-1}\sE[|\delta X_i|^2]\tau_N
 +
 \ol{\om}(\tau_N)^2
+\sum_{i=0}^{N-1}
\bigg[\int_{t_i}^{t_{i+1}} 
 |\Theta_t-\bar{\Theta}_i|^2
\,dt\bigg]
 \bigg)
 \nb
 \\
  &\le
 C\bigg(
 \sE[|\delta X_N|^2]
 +\sum_{i=0}^{N-1}\sE[|\delta X_i|^2]\tau_N
 +
 \ol{\om}(\tau_N)^2
+\cR_\pi(X,Y,Z)
 \bigg).
\l{eq:time_error_YZ}
  \end{align}

Hence, in the case where 
  either $m< n$ or $m=n$ with $\beta_1>0$ holds,
  we can obtain from Lemma \ref{lemma:a_prior_time_error} that
  it holds for all $\eps>0$
  and all sufficiently large $N\in \sN$ that
\begin{align*}
&
\sum_{i=0}^{N-1}
\sE[
|\delta Y_{i}|^2+|\delta {Z}_i|^2
 ]\tau_N
\\
 &\le
\sum_{i=0}^{N-1}
\eps\sE[
|\delta X_{i}|^2]\tau_N
+ C_{(\eps)}
\bigg(\ol{\om}(\tau_N)^2
+
\cR_\pi(X,Y,Z)
+\tau_N^{1/2}
\sum_{i=0}^{N-1}
\sE[
|\delta X_{i}|^2
]\tau_N +\tau_N^{1/2}\sE[|M^\pi_N|^2]
\bigg),
\end{align*}  
for some constant  $C_{(\eps)}$ depending on $\eps$.
Then we can conclude the desired estimate by first using \eqref{eq:time_error_X} and then \eqref{eq:time_error_YZ};
see the proof of   Proposition \ref{thm:stab_exp} for detailed arguments.
For the alternative case  where 
  either  $m>n$ or $m=n$ with $\a, \beta_2>0$ holds,
Lemma \ref{lemma:a_prior_time_error} shows that 
 for all $\eps>0$, 
\begin{align*}
 \sE[|\delta X_N|^2]+
\sum_{i=0}^{N-1}
\sE[
|\delta X_{i}|^2
 ]\tau_N
 &
 \le
\eps
\sum_{i=0}^{N-1}
\sE[
|\delta Y_{i}|^2+|\delta {Z}_i|^2]\tau_N
+
C_{(\eps)} \bigg(\ol{\om}(\tau_N)^2
+
\cR_\pi(X,Y,Z)
\\
&\quad
+\tau_N^{1/2}
\sum_{i=0}^{N-1}
\sE[
|\delta Y_{i}|^2+|\delta Z_{i}|^2
]\tau_N
+\tau_N^{1/2}\sE[| M^\pi_N|^2]
\bigg),
\end{align*}
for some constant  $C_{(\eps)}$ depending on $\eps$.
Then we can conclude the desired estimate by first using  \eqref{eq:time_error_YZ} and then \eqref{eq:time_error_X}.
\end{proof}

\end{appendices}

\section*{Acknowledgements}
\noindent
Wolfgang Stockinger is supported by a special Upper Austrian  Government grant.


\end{document}